\documentclass[11pt, reqno, oneside]{amsart}
\usepackage{amsmath,amsfonts,amssymb,amsthm}

\usepackage{enumerate}
\usepackage{esint}
\usepackage{xcolor}
\usepackage[normalem]{ulem}
\usepackage[pagebackref,colorlinks=true,linkcolor=blue,citecolor=blue,urlcolor=blue]{hyperref}
\usepackage{mathtools}  

\usepackage{mathtools}
\usepackage{setspace}

\usepackage{comment}

\usepackage{geometry}
\numberwithin{equation}{section}

\theoremstyle{plain}
\newtheorem{theorem}{Theorem}[section]
\newtheorem{corollary}[theorem]{Corollary}
\newtheorem{lemma}[theorem]{Lemma}
\newtheorem{proposition}[theorem]{Proposition}

\theoremstyle{definition}
\newtheorem{definition}[theorem]{Definition}

\newtheorem{example}[theorem]{Example}
\newtheorem{remark}[theorem]{Remark}

\makeatother
\makeatletter
\let\Enumerate=\enumerate
\renewcommand{\enumerate}{\Enumerate%
	\setlength{\@topsep}{0pt}
	\setlength{\itemsep}{0pt}%
	\setlength{\parskip}{0pt plus 1pt}%
	\renewcommand{\theenumi}{\textup{(\arabic{enumi})}}%
	\renewcommand{\labelenumi}{\theenumi}%
}
\let\endEnumerate=\endenumerate
\renewcommand{\endenumerate}{\endEnumerate\unskip}

\newcommand{\R}{{\mathbb R}}
\newcommand{\N}{{\mathbb N}}

\newcommand{\Q}{{\mathbb Q}}

\newcommand{\Om}{\Omega}

\providecommand{\vint}[1]{\mathchoice
	{\mathop{\vrule width 5pt height 3 pt depth -2.5pt
			\kern -9pt \kern 1pt\intop}\nolimits_{\kern -5pt{#1}}}
	{\mathop{\vrule width 5pt height 3 pt depth -2.6pt
			\kern -6pt \intop}\nolimits_{\kern -3pt{#1}}}
	{\mathop{\vrule width 5pt height 3 pt depth -2.6pt
			\kern -6pt \intop}\nolimits_{\kern -3pt{#1}}}
	{\mathop{\vrule width 5pt height 3 pt depth -2.6pt
			\kern -6pt \intop}\nolimits_{\kern -3pt{#1}}}}

\newcommand{\eps}{\varepsilon}
\newcommand{\loc}{\mathrm{loc}}

\newcommand{\BV}{\mathrm{BV}}

\newcommand{\liploc}{\mathrm{Lip}_{\mathrm{loc}}}

\newcommand{\ch}{\text{\raise 1.3pt \hbox{$\chi$}\kern-0.2pt}}

\newcommand{\mres}{\mathbin{\vrule height 2ex depth 2.2pt width
		0.12ex\vrule height -0.3ex depth 2.2pt width .5ex}}

\DeclareMathOperator{\Mod}{Mod}

\DeclareMathOperator{\capa}{Cap}
\DeclareMathOperator{\rcapa}{cap}

\DeclareMathOperator{\dist}{dist}
\DeclareMathOperator{\diam}{diam}

\DeclareMathOperator{\Lip}{Lip}

\DeclareMathOperator{\lip}{lip}
\DeclareMathOperator{\supp}{spt}
\DeclareMathOperator*{\aplim}{\mathrm{ap\,lim}}

\DeclareMathOperator{\osc}{osc}

\DeclareMathOperator{\fineint}{fine-int}
\DeclareMathOperator{\finelim}{fine-lim}

\begin{document}
\title[Lipschitz numbers and Fine Differentiability on Metric Measure Spaces]
    {Lipschitz numbers and Fine Differentiability on Metric Measure Spaces}
	\author{Panu Lahti \and Yuxin Li}
	\address{Panu Lahti,  Academy of Mathematics and Systems Science, Chinese Academy of Sciences,
		Beijing 100190, PR China, {\tt panulahti@amss.ac.cn}}
    \address{Yuxin Li,  Academy of Mathematics and Systems Science, Chinese Academy of Sciences,
		Beijing 100190, PR China, {\tt liyuxin@amss.ac.cn}}


\begin{abstract}
On a metric space equipped with a doubling measure supporting a Poincar\'e inequality, we introduce capacity-relaxed infinitesimal Lipschitz numbers and characterize Newtonian functions in $N^{1,p}$, $1\le p\le\infty$,  and functions of bounded variation. 
If $X$ is additionally $Q$-regular, then 
every BV function is $1$-finely differentiable at $\mu$-almost every point, and its fine differential agrees with its approximate Cheeger differential. 
We also extend a Sobolev regularity criterion from continuous to $p$-quasicontinuous functions.
\end{abstract}

	\date{\today}
	\maketitle

     \section{Introduction}

Rademacher's theorem \cite{Rad} states that every Lipschitz function on $\R^n$ is differentiable almost everywhere. 
A related pointwise result is Stepanoff's theorem \cite{Ste}: a function is differentiable almost everywhere in the set where its upper infinitesimal Lipschitz number
\[
\Lip_u(x)
\coloneqq
\limsup_{r\to0}\sup_{y\in B(x,r)}
\frac{|u(y)-u(x)|}{r}
\]
is finite. 
The corresponding lower infinitesimal Lipschitz number $\lip_u(x)$ is obtained by replacing $\limsup$
by $\liminf$. The smaller number $\lip_u$ does not satisfy a direct analog of Stepanoff's theorem,
but it can be shown that if $u$ is continuous,
$\lip_u<\infty$ outside a set of $\sigma$-finite $n-1$-dimensional Hausdorff measure, and also
$\lip_u\in L^p_{\loc}(\R^n)$, then $u\in W^{1,p}_{\loc}(\R^n)$,
with $1\le p\le \infty$, 
see Balogh--Cs\"ornyei \cite{BC}, Z\"urcher \cite{Zu}, and the metric-space extensions in \cite{WZ}.

The theory of differentiability and Sobolev functions extends beyond Euclidean spaces. 
Throughout this paper, $(X,d,\mu)$ is a complete metric measure space equipped with a doubling measure and supporting a $1$-Poincar\'e inequality. 
Newtonian spaces, defined using weak upper gradients, provide the natural metric counterpart of Sobolev spaces; see Shanmugalingam \cite{Shan} and the monographs \cite{BB,HKSTbook}. 
Functions of bounded variation defined on metric measure spaces were developed by Miranda \cite{Mir}.

Lipschitz numbers are very sensitive to the singularities of Sobolev and BV functions. 
For example, a Sobolev function can be discontinuous everywhere and have $\lip_u=\infty$ everywhere, and so a converse to the above result by Balogh--Cs\"ornyei \cite{BC} does not hold.
Instead of continuity, one can consider the relaxed notions of \emph{quasicontinuity} and \emph{fine continuity}. 

The main idea of this paper is to similarly relax the pointwise oscillation in the Lipschitz number by discarding a set of small relative $1$-capacity. 
A further normalization by a Radon measure makes it possible to treat the absolutely continuous
and singular parts of the variation measure of a BV function in a common framework. 
More precisely, let $w\colon\Om\to[-\infty,\infty]$, let $\nu$ be a positive Radon measure in the open set $\Om\subset X$, let $0\le\delta<1$, and let $M\ge1$.
Denote $\osc_{\emptyset}w =0$.
We define the generalized Lipschitz number
\begin{equation}
\label{eq:generalized Lip definition}
\begin{split}
\Lip_{w,\nu}^{\delta,M}(x)
&\coloneqq\inf\left\{
\limsup_{r\to0}\frac{\osc_{U_r}w}{2r}\frac{\mu(B(x,Mr))}{\nu(B(x,Mr))}\colon{}\right.\\
&\hspace{4em}\left.
\limsup_{r\to0}
\frac{\rcapa_1(B(x,r)\setminus U_r,B(x,2r))}{\rcapa_1(B(x,r),B(x,2r))}
\le\delta
	\right\},
\end{split}
\end{equation}
where the infimum is taken over collections of sets $U_r\subset B(x,r)$ defined for all sufficiently small $r>0$.
 If $\nu(B(x,Mr))=0$, we interpret the product in \eqref{eq:generalized Lip definition} as $0$ when $\osc_{U_r}w=0$ and as $\infty$ otherwise.

We fix the structural constant
\begin{equation}
\label{eq:cn definition}
\widetilde c_*
\coloneqq\min\left\{
\frac{c_*}{600C_S^4C_d^9},
\frac{1}{1600C_{\mathcal{M}} C_1C_S^2C_d^6}
\right\},
\end{equation}
where $C_d$, $C_S$, $c_*$, $C_{\mathcal{M}}$, and $C_1$ are the constants from \eqref{eq:doubling-def}, \eqref{eq:Sobolev-ineq.}, \eqref{eq:localized removability density}, Lemma \ref{lem:measures belonging to M}, and \eqref{eq:BV capa and 1-capa-relative}, respectively, and depend only on the doubling and Poincar\'e constants.

When $M=1$, we write $\Lip_{w,\nu}^{\delta}\coloneqq \Lip_{w,\nu}^{\delta,1}$.
If $d\nu=a\,d\mu$, we write $\Lip_{w,a}^{\delta,M}\coloneqq \Lip_{w,\nu}^{\delta,M}$, and if $a\equiv1$, we write simply $\Lip_w^{\delta,M}$.
The case $\delta=0$, denoted $\Lip_w^0(x)=\Lip_w^{0,M}(x)$, is closely related to differentiability with respect to the $1$-fine topology.

We denote by $\mathcal M(\Om,\cdot)$ the class of positive Radon measures introduced in Definition \ref{def:class of measures}. 
Throughout the paper, $u^*$ denotes the precise representative of $u$ defined in Subsection \ref{sec:BV functions}.
Our main results characterize both BV and Newtonian functions in terms of the generalized Lipschitz numbers.

\begin{theorem}\label{thm:main BV}
		Let $\Om\subset X$ be an open set and let $u\in L^1(\Om)$.
		Then the following are equivalent:
		\begin{enumerate}
			\item \label{item-1-thm 1.3}
            $u\in \BV(\Om)$;
			\item \label{item-2-thm 1.3}
            There exist $\nu\in\mathcal M\bigl(\Om,(800C_1C_S^2C_d^6\widetilde c_*)^{-1}\bigr)$ and $M\ge1$ such that $\nu(\Om)<\infty$ and $\Lip_{u^*,\nu}^{\widetilde c_*,M}(x)\le 1$ for every $x\in \Om$.
		\end{enumerate}
	\end{theorem}
	
	\begin{theorem}\label{thm:main Sob}
		Let $\Om\subset X$ be an open set and let $u\in L^p(\Om)$, where $1\le p\le \infty$.
		Then the following are equivalent:
		\begin{enumerate}
			\item \label{item-1-thm 1.4}
            There exists $v\in N^{1,p}(\Om)$ such that $u=v$ $\mu$-almost everywhere in $\Om$;
			\item \label{item-2-thm 1.4}
			There exist a nonnegative $a\in L^p(\Om,\mu)$ and $M\ge1$ such that $\Lip_{u^*,a}^{\widetilde c_*,M}(x)\le 1$ for every $x\in \Om$.
		\end{enumerate}
	\end{theorem}
When $X=\R^n$ equipped with the Euclidean metric and Lebesgue measure, a function belongs to the classical Sobolev space $W^{1,p}(\Om)$ if and only if it has a representative in $N^{1,p}(\Om)$; see \cite[Theorem~7.4.5]{HKSTbook} for $1\le p<\infty$ and \cite[Theorem 4.5]{EvGa} for $p=\infty$.
Condition \ref{item-1-thm 1.4} in Theorem \ref{thm:main Sob} is then simply $u\in W^{1,p}(\Om)$.
As a corollary of Theorem \ref{thm:main BV},
 we obtain a Federer-type characterization of sets of finite perimeter using a suitable capacitary boundary, see Corollary \ref{cor:finite perimeter char}.
This connects the present results with the metric space theory of sets of finite perimeter developed in \cite{A1,L-Fed,L-Fed-new}.

Our second result concerns the $1$-fine differentiability of BV functions. 
Cheeger \cite{Che} first proved that complete doubling metric measure spaces supporting a Poincar\'e inequality admit a countable measurable differentiable structure $\{(X_\alpha,\xi^\alpha)\}_\alpha$.
Keith subsequently developed this theory in two directions: in \cite{Kei} he weakened the assumptions 
by means of the Lip--lip condition, while in \cite{Kei-P} he considered the corresponding approximate differentiability framework, obtained by replacing the ordinary limit in the chartwise expansion by an approximate limit.
Ranjbar--Motlagh \cite[Theorems 4.5 and 4.6]{RM08} obtained related $L^p$-type differentiability results for Sobolev and BV functions on metric measure spaces. 
Later, Durand-Cartagena--Ihnatsyeva--Korte--Szuma\'nska \cite[Corollary 4.1]{DCIKS} proved, under the assumption of a Poincar\'e inequality,
that functions in $N^{1,p}$ and $\BV$ are approximately differentiable almost everywhere.
For Newtonian functions, the Sobolev differential is compatible with the Cheeger structure, and its intrinsic norm agrees with the minimal weak upper gradient; see \cite{ES24}.

Fine differentiability in Euclidean spaces and its connections with maximal functions and quasiconformal mappings have been studied in \cite{Gar10,L-FD,L-FQ,Lav08}. 
Under an Ahlfors $Q$-regularity assumption, we prove the following $1$-fine differentiability theorem, which can be viewed as a fine-topological strengthening of \cite[Corollary 4.1]{DCIKS}.

\begin{theorem}
\label{thm:BV-finely-diff-main theorem}
Let $(X,d,\mu)$ be an Ahlfors $Q$-regular metric measure space with $Q\ge1$, supporting a $1$-Poincar\'e inequality. 
Let $\Om\subset X$ be open and let $u\in\BV(\Om)$. 
Let $\{(X_\alpha,\xi^\alpha)\}_\alpha$, where $\xi^\alpha=(\xi^\alpha_1,\ldots,\xi^\alpha_{N_\alpha})$, be a Cheeger measurable differentiable structure.
Then, for $\mu$-a.e. $x\in\Om\cap X_\alpha$, the representative $u^*$ is $1$-finely differentiable at $x$, and
\[
d^\alpha_{\mathrm{fine}}u^*(x)
=d^\alpha u(x).
\]
\end{theorem}

Here $d^\alpha u(x)$ denotes the approximate Cheeger differential of $u^*$ at $x$, see the definition in Section \ref{sec:fine differentiability}.

Finally, we return to regularity criteria of Balogh--Cs\"ornyei \cite{BC}.
In \cite[Remark 3.15]{Zu}, Z\"urcher asked whether the continuity assumption in criteria involving the lower Lipschitz number $\lip_u$ could be replaced by $p$-quasicontinuity.
Proposition \ref{prop:three lip conditions} gives a related criterion for $\Lip_u^0$: if $X$ is Ahlfors $Q$-regular with $Q\ge1$, $1\le p<\infty$, $\Om$ is bounded, $u\colon\Om\to[-\infty,\infty]$ is $p$-quasicontinuous, $\Lip_u^0<\infty$ outside a set of $\sigma$-finite $\mathcal H$-measure, and $\Lip_u^0\in L^p(\Om)$, then $u\in D^p(\Om)$.


The paper is organized as follows. 
Section \ref{sec:prelis} recalls Newtonian and BV spaces, relative $1$-capacity, and the $1$-fine topology.
In Section \ref{sec:main theorems}, we prove the generalized Lipschitz characterizations
of Theorems \ref{thm:main BV} and \ref{thm:main Sob}, 
and the resulting finite perimeter criterion.
We then study fine differentiability with respect to Cheeger coordinates in Section \ref{sec:fine differentiability}, and conclude with the applications to quasicontinuous and continuous functions in Section \ref{sec: lipf}.
	
	\section{Preliminaries}\label{sec:prelis}
	\subsection{Metric measure space, doubling measure, and Poincar\'e inequality}
	
	Throughout this paper, we assume that $(X, d,\mu)$ is a complete metric measure space equipped with a doubling measure and supporting a $1$-Poincar\'e inequality. 
	We also assume throughout that $X$ contains at least two points and that $\Omega\subset X$ is open.
	Unless otherwise stated, $C$ denotes a positive constant whose value may change from one occurrence to another. Constants with subscripts are fixed within the context in which they are introduced.
	
	A triple $(X,d,\mu)$ is said to be a metric measure space if $(X,d)$
	is a nonempty separable metric space and $\mu$ is a locally finite Borel-regular
	outer measure on $X$.
	Hence for every set $A\subset X$, we have
	\[
	\mu(A)= \inf \{ \mu(O): A\subset O\subset X,\, O\  {\rm is\ open} \},
	\]
	see e.g. \cite[Proposition 3.3.37]{HKSTbook}.
If a property holds outside a set of measure zero with respect to $\mu$, we say that it holds almost everywhere, or ``$\mu$-a.e.''.
	
	We denote the characteristic function of a set $E\subset X$ by $\ch_E\colon X\to \{0,1\}$.
	We write $B(x,r)$ for an open ball in $X$ with center $x$
	and radius $r$, that is, $\{y \in X \colon d(x,y)<r\}$.
	We denote the closed ball with center $x$ and radius $r$ by $\overline{B}(x,r)$.
	If $B=B(x,r)$ and $\tau>0$, we write $\tau B\coloneqq B(x,\tau r)$.
	
	We say that $(X,d,\mu)$ is a \emph{metric measure space equipped with a doubling measure} if there exists a doubling constant $C_d>0$ (we assume that $C_d>1$)  such that
	\begin{equation}\label{eq:doubling-def}
		0<\mu(B(x,2r))\leq C_d \mu(B(x,r)) <\infty
	\end{equation}
	for all balls $B(x,r)$. 
	By iterating the doubling condition, we obtain for any $x\in X$ and any $y\in B(x,R)$ with $0<r\le R<\infty$ that
	\begin{equation}\label{eq: doubling-ball-diff-center}
		\frac{\mu(B(y,r))}{\mu(B(x,R))}\ge \frac{1}{C_d^2}\left(\frac{r}{R}\right)^{Q_d},
	\end{equation}
	where we choose $Q_d\coloneqq\log_2 C_d$; a smaller doubling exponent may work as well, see \cite[Lemma 3.3]{BB}.
    We say that $\mu$ (or $X$) is Ahlfors $Q$-regular if there is a constant $C_A\ge1$ such that
    \begin{equation}\label{eq: Q-regular-def}
        C_A^{-1} r^Q\le \mu(B(x,r))\le C_A r^Q
    \end{equation}
	   for all $x\in X$ and $0<r<2\diam X$.

	For any set $A\subset X$ and $0<R<\infty$, the restricted spherical Hausdorff content of
	codimension one is defined to be 
	\[
	\mathcal{H}_R(A)\coloneqq\inf\left\{\sum_{i}\frac{\mu(B(x_i,r_i))}{r_i}\colon A\subset \bigcup_{i}B(x_i,r_i),\, r_i\le R \right\},
	\]
	where the infimum is taken over finite and countable collection of balls.
	The \emph{codimension one Hausdorff measure} of $A\subset X$ is defined to be 
	\[
	\mathcal{H}(A)\coloneqq\lim_{R\to0}\mathcal{H}_R(A).
	\]
	
	A \emph{curve} is a continuous mapping from an interval,
	and a \emph{rectifiable} curve is a curve with finite length.
	A rectifiable curve can
	be parameterized by its arc length $ds$.
	We say that a Borel function $g:X\to[0,\infty]$ is an \emph{upper gradient} of $u$ in a set $A\subset X$ if
	\[
	|u(\gamma(0))-u(\gamma(\ell_{\gamma}))|\leq \int_\gamma g\, ds:=\int_{0}^{\ell_{\gamma}}g(\gamma(s))\,ds
	\]
	for every rectifiable curve $\gamma\colon [0,\ell_{\gamma}]\to A$.

	For $1\le p<\infty$, the $p$-modulus of a family $\Gamma$ of rectifiable curves is
	\[
	\operatorname{Mod}_p(\Gamma)\coloneqq\inf_{\rho}\int_X\rho^p\,d\mu,
	\]
	where the infimum is taken over all nonnegative Borel functions $\rho$ such that $\int_\gamma\rho\,ds\ge1$ for every $\gamma\in\Gamma$.
	A nonnegative Borel function $g$ is a $p$-weak upper gradient of $u$ if the upper gradient inequality holds outside a curve family of $p$-modulus zero.
	Given an open set $\Om\subset X$, we set
	\[
	\|u\|_{N^{1,p}(\Om)}\coloneqq\left\{\|u\|_{L^p(\Om)}^p+\inf_g\|g\|_{L^p(\Om)}^p\right\}^{1/p}
	\quad\text{and}\quad
	N^{1,p}(\Om)\coloneqq\{u:\|u\|_{N^{1,p}(\Om)}<\infty\},
	\]
	where the infimum is taken over all $p$-weak upper gradients of $u$ in $\Om$.
	We understand every Newtonian function to be defined at every $x\in \Om$.
	For $u\in N^{1,p}_{\loc}(\Om)$, the \emph{minimal} $p$-weak upper gradient is denoted by $g_u$.
	It satisfies $g_u\le g$ almost everywhere for every $p$-weak upper gradient $g\in L^p_{\loc}(\Om)$, see \cite[Theorem 2.25]{BB}.

	For $p=\infty$, we use the convention that $N^{1,\infty}(\Om)$ consists of functions $u\in L^\infty(\Om)$ admitting an upper gradient in $L^\infty(\Om)$.
	We write $u\in N^{1,\infty}_{\loc}(\Om)$ if $u|_V\in N^{1,\infty}(V)$ for every open $V\Subset\Om$;
	the latter means that $V$ is open and $\overline{V}$ is a compact subset of $\Om$.
	By the local pointwise estimate in \cite[Theorem 8.1.55]{HKSTbook} applied with $p=1$ and an essentially bounded upper gradient, every $u\in N^{1,\infty}_{\loc}(\Om)$ has a locally Lipschitz representative.
	In what follows, we always take such $u$ to be its locally Lipschitz representative,
	and we set $g_u=g_{u,\infty}\coloneqq\Lip_u$.

	We say that  $\mu$ supports a \emph{$1$-Poincar\'e inequality} if there exist  Poincar\'e constants $C_P>0$,  $\lambda\ge1$ such that
	\begin{equation}\label{eq: (1,1)-Poincare}
		\vint{B(x,r)} |u - u_{B(x,r)}|\,d\mu
		\leq C_P r \vint{B(x,\lambda r)} g\, d\mu
	\end{equation}
	for every ball $B(x,r)$, every function $u$ integrable on balls, and every upper gradient $g$ of $u$, where $u_B\coloneqq\vint{B}u\,d\mu=\mu(B)^{-1}\int_Bu\,d\mu$.
	We refer the reader to \cite{BB,Ha03,HK00,HKSTbook,Shan}
	for a discussion on upper gradients, Poincar\'e inequalities, and their applications in complete, doubling metric measure spaces. 
	
	We will need the following Vitali-Carath\'eodory theorem;
	for a proof see e.g. \cite[p. 108]{HKSTbook}.
	By a positive Radon measure on $\Om$ we mean a nonnegative Borel measure that is finite on compact subsets of $\Om$ and is both inner regular and outer regular.
	
	\begin{theorem}\label{thm:VitaliCar}
		Let $\mu$ be a positive Radon measure on $\Om$, and let 
		$\rho\in L^1(\Om,\mu)$, with values in $[0,\infty]$.
Then there exists a sequence $\{\rho_i\}_{i=1}^{\infty}$
		of lower semicontinuous functions
		on $\Om$ such that $\rho\le \rho_{i+1}\le\rho_i$ for all $i\in\N$, and
		$\rho_i\to\rho$ in $L^1(\Om,\mu)$.
	\end{theorem}
	
	
	\subsection{Functions of bounded variation}\label{sec:BV functions}
	
	We recall the definition and basic properties of functions of bounded variation on metric spaces, see 
	\cite{Mir}.
See also e.g. \cite{AFP,EvGa} for the classical theory in the Euclidean setting.
As before let $\Om\subset X$ be an open set.
Given a function $u\in L_{\loc}^1(\Om)$, we define the \emph{total variation of $u$} in $\Om$ by
	\begin{equation}\label{eq: def-BV-metric}
		\|Du\|(\Om)\coloneqq\inf\left\{\liminf_{i\to \infty} \int_{\Om}g_{u_i}\,d\mu\colon u_i\in \Lip_{\loc}(\Om),\,u_i\to u \text{ in } L_{\loc}^1(\Om) \right\},
	\end{equation}
	where each $g_{u_i}$ is an upper gradient of $u_i$ in $\Om$.
	If $u\in L^1(\Om)$ and $\|Du\|(\Om)<\infty$, we say that $u$ is a function of bounded variation and denote $u\in \BV(\Om)$.  
	The Dirichlet space $D^{\BV}(\Om)$ is defined in the same way, except that the integrability requirement for the function itself is relaxed to $u\in L_{\loc}^1(\Om)$.
	For an arbitrary set $A\subset\Om$, we define
	\[
	\|Du\|(A)\coloneqq\inf\{\|Du\|(W):A\subset W\subset\Om,\ W\text{ open}\}.
	\]
	If $u\in L_{\loc}^1(\Om)$ and $\|Du\|(\Om)<\infty$, $\|Du\|(\cdot)$ is a Radon measure on $\Om$ by \cite[Theorem 3.4]{Mir}.
A $\mu$-measurable set $E\subset X$ is said to be of finite perimeter if $\|D\ch_E\|(X)<\infty$.
The perimeter of $E$ in $\Om$ is also denoted by 
	\[
	P(E,\Om)\coloneqq\|D\ch_E\|(\Om).
	\]
	Similarly, if $P(E,\Om)<\infty$, then $P(E,\cdot)$ is a finite Radon measure on $\Om$.

	The measure-theoretic interior of a set $E\subset X$ is defined by 
	\[
	I_E\coloneqq\left\{x\in X\colon \lim_{r\to 0} \frac{\mu(B(x,r)\setminus E)}{\mu(B(x,r))}=0 \right\},
	\]
	and the measure-theoretic exterior by 
	\[
	O_E\coloneqq\left\{x\in X\colon \lim_{r\to 0} \frac{\mu(B(x,r)\cap E)}{\mu(B(x,r))}=0 \right\}.
	\]
	The measure-theoretic boundary $\partial^*E$ is defined as the set of points $x\in X$ at which both $E$ and its complement have nonzero upper density, i.e.
	\[
	\limsup_{r\to 0} \frac{\mu(B(x,r)\cap E)}{\mu(B(x,r))}> 0 
	\quad
	\textrm{and}
	\quad
	\limsup_{r\to 0}\frac{\mu(B(x,r)\setminus E)}{\mu(B(x,r))}>0.
	\]
	Note that the space $X$ is always partitioned into the disjoint sets $I_E, O_E$ and $\partial^*E$.
	By Lebesgue's differentiation theorem (see e.g. \cite[Section 3.4]{HKSTbook}), for a $\mu$-measurable set $E$ we have $\mu(E\Delta I_E)=0$, where $\Delta$ is the symmetric difference.
	
	Given a number $0<\gamma\le 1/2$, we also define the strong boundary
	\begin{equation}\label{eq:strong boundary}
		\Sigma_{\gamma}E\coloneqq\left\{
		x\in X\colon \liminf_{r\to 0}\frac{\mu(B(x,r)\cap E)}{\mu(B(x,r))}\ge \gamma
		\quad
		\textrm{and}
		\quad
		\liminf_{r\to 0} \frac{\mu(B(x,r)\setminus E)}{\mu(B(x,r))}\ge \gamma
		\right\}.
	\end{equation}
	For an open set $\Om\subset X$ and a $\mu$-measurable set $E\subset X$ with $P(E,\Om)<\infty$, we have
	\begin{equation}\label{eq: measure theoretic boundary-gamma}
		\mathcal{H}((\partial^*E\setminus \Sigma_{\gamma}E)\cap \Om)=0
	\end{equation}
	for  $\gamma\in (0,1/2]$ that depends on doubling and Poincar\'e constants, see \cite[Theorem 5.4]{A1}.
	Moreover, for any Borel set $A\subset \Om$, we have
	\begin{equation}\label{eq: P(E,A)}
	    P(E,A)=\int_{\partial^*E\cap A}\theta_E\,d\mathcal{H},
	\end{equation}
	where $\theta_E\colon \Om\to [\alpha,C_d]$ with $\alpha=\alpha(C_d,C_P,\lambda)>0$, see \cite[Theorem 5.3]{A1} and \cite[Theorem 4.6]{AMP}.

	This implies that
	\begin{equation}\label{eq:relation-Hausdorff-perimeter}
		\alpha \mathcal{H}(\partial^*E\cap A)\le P(E, A)\le C_d\mathcal{H}(\partial^*E\cap A) 
		\quad\textrm{for every Borel set } A\subset\Om.
	\end{equation}

	
	The coarea formula states that for a function $u\in \BV(\Om)$ and a Borel set $A\subset \Om$, we have
	\begin{equation}\label{eq:coarea}
		\|Du\|(A)=\int_{-\infty}^{\infty}\|D\ch_{\{u>t\}}\|(A)\,dt.
	\end{equation}
	Here we abbreviate $\{u>t\}\coloneqq \{x\in \Om\colon u(x)>t\}$.
	
    The $1$-Poincar\'e inequality \eqref{eq: (1,1)-Poincare} implies the following Sobolev inequality.
If $x\in X$, $0<r<\frac{1}{4}\diam X$, and $u\in N^{1,1}(X)$ with $u=0$ in $X\setminus B(x,r)$, then 
	\begin{equation}\label{eq:Sobolev-ineq.}
		\int_{B(x,r)}|u|\,d\mu\le C_S r\int_{B(x,r)}g_u\,d\mu
	\end{equation}
	for any upper gradient $g_u$ of $u$ and a constant $C_S=C_S(C_d,C_P,\lambda)\ge1$, see \cite[Theorem 5.51]{BB}. 

    Applying the Poincar\'e inequality \eqref{eq: (1,1)-Poincare} to sequences of approximating $N_{\loc}^{1,1}$-functions in the definition of the total variation, we get the following BV version: for every ball $B(x,r)$ and every $u\in L_{\loc}^1(X)$, we have
    \begin{equation}\label{eq: BV-Poincare}
    \int_{B(x,r)} |u-u_{B(x,r)}|\,d\mu\le C_Pr\|Du\|(B(x,\lambda r)).
    \end{equation}
	
	For every measurable set $E\subset X$ of finite perimeter and every ball $B(x,r)$, from the above we get the relative isoperimetric inequality
	\begin{equation}\label{eq:relative isoperimetric inequality}
		\min\{\mu(B(x,r)\cap E),\mu(B(x,r)\setminus E)\}
		\le 2C_P rP(E,B(x,\lambda r)).
	\end{equation}
	
	Let $u\in L^1_{\loc}(\Om)$.
	We will often consider a particular pointwise representative, which we call the precise representative,
	\[
	u^*(x)\coloneqq \limsup_{r\to 0}\vint{B(x,r)}u\,d\mu,\quad x\in \Om.
	\]
	This is easily seen to be a Borel function.
	It is worth noting that BV functions, usually denoted by $u$, 
	are understood to be defined only almost everywhere,
	but in many cases we need pointwise defined functions such as $u^*$. 
	The symbol $w$ will often denote a pointwise defined function, for example in the definition
	\eqref{eq:generalized Lip definition} of $\Lip_{w,\nu}^{\delta,M}$.
The oscillation of
	$w\colon U\to[-\infty,\infty]$ in a set $U\subset X$ is defined by
	\[
	\underset{U}{\osc}\,w\coloneqq \sup\{|w(x)-w(y)|,\,x,y\in U\}.
	\]
	If $|w(x)|=\infty$ for some $x\in U$, we interpret $\osc_Uw=\infty$.
	
	We say that $x\in\Om$ is a Lebesgue point of $u$ if
	\[
	\lim_{r\to 0}\,\vint{B(x,r)}|u(y)-\widetilde{u}(x)|\,d\mu(y)=0
	\]
	for some $\widetilde{u}(x)\in\R$. 
	
	
	The lower and upper approximate limits of a function $u\in\BV_{\loc}(\Om)$
	are defined respectively by
	\[
	u^{\wedge}(x)\coloneqq 
	\sup\left\{t\in\R\colon \lim_{r\to 0}\frac{\mu(B(x,r)\cap\{u<t\})}{\mu(B(x,r))}=0\right\}
	\]
	and
	\[
	u^{\vee}(x)\coloneqq 
	\inf\left\{t\in\R\colon \lim_{r\to 0}\frac{\mu(B(x,r)\cap\{u>t\})}{\mu(B(x,r))}=0\right\},
	\]
	for all $x\in\Om$.
	We interpret the supremum and infimum of an empty set to be $-\infty$ and $\infty$, respectively.

For a measurable function $u:\Om\to\R$, $x\in\Om$, and $L\in\R$, we say that $u$ has \emph{approximate limit} $L$ at $x$, and write $\aplim_{y\to x}u(y)=L$, if for every $\varepsilon>0$,
\begin{equation}\label{eq:def-approx-limit}
\lim_{r\to0}\frac{\mu(B(x,r)\cap\{|u-L|>\varepsilon\})}{\mu(B(x,r))}=0.
\end{equation}



	We need the following pointwise property of BV functions from \cite[Theorem 1.1]{KKST14}.
	\begin{lemma}\cite[Theorem 1.1 \& Lemma 3.2]{KKST14}
		If $u\in \BV(X)$, then for $\mathcal{H}$-almost every $x\in X$, we have
		\begin{equation}\label{eq:pointwise-properties}
			\begin{split}
				(1-\gamma)&u^\wedge(x)+\gamma u^\vee(x)\le \liminf_{r\to0}\vint{B(x,r)} u\,d\mu\\
				&\le \limsup_{r\to 0}\vint{B(x,r)}u\,d\mu\le \gamma u^\wedge(x)+(1-\gamma)u^\vee(x),
			\end{split}
		\end{equation}
	with all quantities finite,
	and	where $0<\gamma\le 1/2$ and $\gamma $ depends only on the doubling and Poincar\'e inequality constants.
	\end{lemma}

	\subsection{Capacities and fine topology}
	
	We will often use the notion of capacity,
	and in this subsection we record relevant background facts and some elementary lemmas that will be needed.
	Recall that capacity is needed already
	in the definition \eqref{eq:generalized Lip definition} of $\Lip_{w,\nu}^{\delta,M}$.
	
	Consider $1\le p< \infty$.
	The (Sobolev) $p$-capacity of a set $A\subset X$ is defined by
	\[
	\capa_p(A)\coloneqq \inf \Vert u\Vert_{N^{1,p}(X)}^p,
	\]
	where the infimum is taken over all functions $u\in N^{1,p}(X)$ satisfying
	$u\ge 1$ on $A$.
We know that $\capa_p$ is an outer capacity, meaning that
	\[
	\capa_p(A)=\inf\{\capa_p(W)\colon W\supset A \text{ is open}\}
	\]
	for any $A\subset X$, see e.g. \cite[Theorem 5.31]{BB}.
If a property holds outside a set $A\subset X$ with $\capa_p(A)=0$, we say that it holds at $p$-quasi-every point.
	
	Given sets $A\subset W\subset X$, where $W$ is open, the relative $p$-capacity is defined by
	\[
	\rcapa_p(A,W)\coloneqq \inf \int_{X}g_u^p\,d\mu,
	\]
	where the infimum is taken over functions $u\in N^{1,p}(X)$ such that $u\ge 1$ on $A$ and $u=0$ on $X\setminus W$, and their upper gradients $g_u$.
For basic properties satisfied by capacities, such as monotonicity and countable subadditivity, see e.g. \cite{BB, HKSTbook}.

	The $1$-capacity and the Hausdorff contents are closely related: it follows from \cite[Theorem 4.3, Theorem 5.1]{HaKi} that 
	\begin{equation}\label{eq:capa and Hausdorff measure}
		\capa_1(A)=0 
		\quad \textrm{if and only if} \quad
		\mathcal{H}(A)=0.
	\end{equation}
	For any $R>0$ and $A\subset X$, from the proof of 
	\cite[Lemma 3.3]{KKST} it follows that
	\begin{equation}\label{eq: capacity-and-Hausdorff measure}
	    \capa_1(A)\le C(C_d,C_P,\lambda,R)\mathcal{H}_R(A),
	\end{equation}
	and by combining \cite[Theorem 4.3]{HaKi} and the proof of \cite[Theorem 5.1]{HaKi} we obtain that conversely 
	\[
	\mathcal{H}_R(A)\le C(C_d,C_P,\lambda, R)\capa_1(A).
	\]

    We also have a corresponding result with the relative capacity.
    \begin{lemma}
    For a ball $B(x,r)$, and $A\subset B(x,r)$ with $0<r<\frac{1}{8}\diam X$, we have 
    \begin{equation}\label{eq: relative capacity-Hausdorff}
    \rcapa_1(A,B(x,2r))\le C_d \mathcal{H}(A).
    \end{equation}
    \end{lemma}
    \begin{proof}
    We may assume that $\mathcal{H}(A)<\infty$. 
	Let $\varepsilon>0$ and let $B(x_i,r_i)$, with $r_i<r/3$, be a covering of $A$ such that every $B(x_i,r_i)$ meets $A$ and
	\[
	\mathcal{H}(A)\ge\mathcal{H}_{r/3}(A)\ge\sum_i\frac{\mu(B(x_i,r_i))}{r_i}-\varepsilon.
	\]
	Then $B(x_i,2r_i)\subset B(x,2r)$.
	Setting 
    \[
    u_i(x)=\left(1-\frac{\dist(x,B(x_i,r_i))}{r_i} \right)_+,
    \]
    we have that
    \[
    g_i=\frac{1}{r_i}\chi_{B(x_i,2r_i)}
    \]
    is a weak upper gradient of $u_i$.
    It follows that 
    \begin{align*}
    \rcapa_1(A,B(x,2r))
    &\le \sum_i\rcapa_1(B(x_i,r_i),B(x,2r))\\
    &\le \sum_i \int_X g_i\,d\mu\le \sum_i \frac{\mu(B(x_i,2r_i))}{r_i}\\
    &\le C_d\sum_i\frac{\mu(B(x_i,r_i))}{r_i}\le C_d (\mathcal{H}(A)+\varepsilon).
    \end{align*}
    Letting $\varepsilon\to 0$, we complete the proof.
    \end{proof}
	
	The $\BV$-capacity of a set $A\subset X$ is defined by
	\[
	\capa_{\BV}(A)\coloneqq \inf\|u\|_{\BV(X)},
	\]
	where the infimum is taken over all $u\in \BV(X)$ with $u\ge 1$ in a neighborhood of $A$.
	By \cite[Theorem 4.3]{HaKi}, we know that for some constant $C(C_d,C_P,\lambda)\ge 1$ and any $A\subset X$, we have
	\begin{equation}\label{eq:BV capa and 1-capa}
		\capa_{\BV}(A)\le \capa_1(A)\le C\capa_{\BV}(A).
	\end{equation}
	
	Given sets $A\subset W\subset X$, where $W$ is open, we define a variational BV-capacity by
	\[
	\rcapa_{\BV}(A, W)\coloneqq \inf\|Du\|(X),
	\]
	where the infimum is taken over functions $u\in \BV(X)$ such that $u^\vee=u^\wedge=0$ $\mathcal{H}$-a.e. in $X\setminus W$ and $u^\wedge\ge 1$ $\mathcal{H}$-a.e. on $A$.
	
	Combining \cite[Eq. (4.2)]{L-20} with \cite[Theorem 5.7]{La-20} gives a constant $C_1$, depending only on $C_d,C_P,\lambda$, such that for any $A\subset B(x,r)$, we have
	\begin{equation}\label{eq:BV capa and 1-capa-relative}
		\rcapa_{\BV}(A,B(x,2r))\le \rcapa_1(A,B(x,2r))
		\le C_1\rcapa_{\BV}(A,B(x,2r)).
	\end{equation}
	From the equivalence of $1$-capacity and BV capacity, we have the following result.
	\begin{lemma} \label{lem:capacities continuity}
		Fix a ball $B(x,r)$ with $0<r<\frac{1}{8}\diam X$.
Let $A_1\subset A_2\subset\ldots$ be subsets of $X$,
		and denote $A\coloneqq \bigcup_{j=1}^{\infty}A_j$.
Then
		\begin{equation}\label{eq:capacities continuity-comparable}
			\begin{split}
				\lim_{j\to \infty}\rcapa_1(A_j\cap B(x,r), B(x,2r))
				&\le \rcapa_1(A\cap B(x,r), B(x,2r))\\
				&\le C_1\lim_{i\to \infty}\rcapa_1(A_i\cap B(x,r), B(x,2r)).
			\end{split}
		\end{equation}
        where $C_1$ is the constant in \eqref{eq:BV capa and 1-capa-relative}.
	\end{lemma}
		
		\begin{proof}
			Write $B=B(x,r)$. The first inequality follows from monotonicity.
			For the second inequality, we first prove that
			\[
			\rcapa_{\BV}(A\cap B,2B)
			=\lim_{j\to\infty}\rcapa_{\BV}(A_j\cap B,2B).
			\]
			Fix $\varepsilon>0$. For each $j\ge1$, choose $u_j\in\BV(X)$ admissible for $\rcapa_{\BV}(A_j\cap B,2B)$ such that $0\le u_j\le1$ and
			\[
			\|Du_j\|(X)\le\rcapa_{\BV}(A_j\cap B,2B)+\varepsilon2^{-j}.
			\]
			By \cite[Theorem 2.7(iii)]{HaKi},
			\begin{equation}\label{eq:min and max}
			\|D\min\{u,v\}\|(X)+\|D\max\{u,v\}\|(X)
			\le\|Du\|(X)+\|Dv\|(X)
			\end{equation}
			for any $u,v\in\BV(X)$.
			Set $v_j=\max\{u_1,\ldots,u_j\}$. Then $v_j$ is admissible for $\rcapa_{\BV}(A_j\cap B,2B)$.
			For $j\ge2$, the function $\min\{v_{j-1},u_j\}$ is admissible for $\rcapa_{\BV}(A_{j-1}\cap B,2B)$, and hence by \eqref{eq:min and max},
			\[
			\begin{aligned}
				\|Dv_j\|(X)
				&\le\|Dv_{j-1}\|(X)+\|Du_j\|(X)
				-\rcapa_{\BV}(A_{j-1}\cap B,2B)\\
				&\le\|Dv_{j-1}\|(X)
				+\rcapa_{\BV}(A_j\cap B,2B)-\rcapa_{\BV}(A_{j-1}\cap B,2B)+\varepsilon2^{-j}.
			\end{aligned}
			\]
			By induction, we know
			\[
			\|Dv_j\|(X)\le\rcapa_{\BV}(A_j\cap B,2B)
			+\varepsilon\sum_{i=1}^j2^{-i}
			\le\rcapa_{\BV}(A_j\cap B,2B)+\varepsilon.
			\]
			
			Choose also $w\in\BV(X)$ admissible for $\rcapa_{\BV}(A\cap B,2B)$ such that $0\le w\le1$ and
			\[
			\|Dw\|(X)\le\rcapa_{\BV}(A\cap B,2B)+\varepsilon.
			\]
			Set $z_j=\min\{v_j,w\}$. Since $\max\{v_j,w\}$ is admissible for $\rcapa_{\BV}(A\cap B,2B)$, it follows that
			\[
			\begin{aligned}
				\|Dz_j\|(X)
				&\le\|Dv_j\|(X)+\|Dw\|(X)
				-\|D\max\{v_j,w\}\|(X)\\
				&\le\|Dv_j\|(X)+\|Dw\|(X)
				-\rcapa_{\BV}(A\cap B,2B)\\
				&\le\rcapa_{\BV}(A_j\cap B,2B)+2\varepsilon.
			\end{aligned}
			\]
			As $0\le z_j\le w$ and $w\in L^1(X)$, we know that $z_j\to z\coloneqq \sup_j z_j$ in $L^1(X)$.
			Also by the lower semicontinuity of the total variation, we know that $z\in\BV(X)$ and
			\[
			\|Dz\|(X)\le\liminf_{j\to\infty}\|Dz_j\|(X)
			\le\lim_{j\to\infty}\rcapa_{\BV}(A_j\cap B,2B)+2\varepsilon.
			\]
			Since $0\le z\le w$, we have $z^\wedge=z^\vee=0$ $\mathcal H$-a.e. on $X\setminus2B$.
			Moreover, since $z\ge z_j$ and $z_j$ is admissible for $\rcapa_{\BV}(A_j\cap B,2B)$, then $z^\wedge\ge1$ $\mathcal H$-a.e. on $A_j\cap B$ for every $j$.
			Since $A\cap B=\bigcup_{j=1}^\infty(A_j\cap B)$, we know $z$ is admissible for $\rcapa_{\BV}(A\cap B,2B)$.
			Thus
			\[
			\rcapa_{\BV}(A\cap B,2B)
			\le\|Dz\|(X)
			\le\lim_{j\to\infty}\rcapa_{\BV}(A_j\cap B,2B)+2\varepsilon.
			\]
			Letting $\varepsilon\to0$ we get the  continuity for the relative BV capacity.
			Finally, by \eqref{eq:BV capa and 1-capa-relative},
			\[
			\begin{aligned}
				\rcapa_1(A\cap B,2B)
				&\le C_1\rcapa_{\BV}(A\cap B,2B)\\
				&=C_1\lim_{j\to\infty}\rcapa_{\BV}(A_j\cap B,2B)\\
				&\le C_1\lim_{j\to\infty}\rcapa_1(A_j\cap B,2B),
			\end{aligned}
			\]
			which proves the second inequality as in \eqref{eq:capacities continuity-comparable}.
		\end{proof}

	
	We say that a set $U\subset X$ is $p$-quasiopen if for every $\eps>0$ there exists an open set $G\subset X$ such that $\capa_p(G)<\varepsilon$ and $U\cup G$ is open.
	We say that a function $w$ is $p$-quasicontinuous on $\Om$ if for every $\eps>0$ there exists
	an open set $G\subset \Om$ such that $\capa_p(G)<\eps$ and
	$w|_{\Om\setminus G}$ is finite and continuous.
We know that
	\begin{equation}\label{eq:quasicontinuous}
		\text{if }u\in N_{\loc}^{1,p}(\Om)\textrm{, then }u \textrm{ is }p\textrm{-quasicontinuous on }\Om,
	\end{equation}
	see e.g. \cite[Theorem 5.29]{BB}.
	
	For every open set $\Omega\subset X$ and every $u\in N^{1,1}(\Omega)$, $1$-quasievery point of $\Omega$ is a Lebesgue point of $u$; that is,
\begin{equation}\label{eq:Lebesgue points}
\lim_{r\to0}\,\vint{B(x,r)} |u(y)-u(x)|\,d\mu(y)=0
\quad\text{for $1$-q.e.\ }x\in\Omega;
\end{equation}
see \cite[Proposition~4.11 and Remark~4.8]{BBL-23}.

	By \cite[Proposition 6.16]{BB}, we know that for a ball $B(x,r)$ and $A\subset B(x,r)$ with $0<r<\tfrac{1}{8}\diam X$,
	and for all $1\le p<\infty$, we have
	\begin{equation}\label{eq:rcapa over ball-measure}
		\frac{\mu(A)}{C_{S,p}r^p}\le \rcapa_p(A,B(x,2r))\le \frac{C_d\mu(B(x,r))}{r^p}
	\end{equation}
	and
	\begin{equation}\label{eq:capa vs rcapa}
		\frac{\capa_p(A)}{C'_p(1+r^p)}
		\le \rcapa_p(A,B(x,2r))
		\le 2^p \left(1+\frac{1}{r^p}\right)\capa_p(A).
	\end{equation}
	Here $C_{S,p}$ and $C'_p$ depend only on $C_d,C_P,\lambda,p$.
	For $p=1$ we can take $C_{S,1}=C_S$, with $C_S$ fixed as specified after \eqref{eq:Sobolev-ineq.}.
	It follows that for $A\subset B(x,1)$ and $1\le p<q<\infty$, we have
	$\capa_p(A)^{1/p}\le C''\capa_q(A)^{1/q}$
	for another constant $C''$ depending only on $C_d,C_P,\lambda,p,q$,
	and so using a covering argument, we can show for any function $w$ that
	\begin{equation}\label{eq:pq quasicontinuity}
		\textrm{if }w\textrm{ is }q\textrm{-quasicontinuous on }\Om,\textrm{ then }w\textrm{ is }p\textrm{-quasicontinuous on }\Om.
	\end{equation}
	
	For every $A\subset B(x,r)$, $0<r<\tfrac18\diam X$, and $1\le p<\infty$, we then also have
	\begin{equation}\label{eq:rcapa in small ball}
	\begin{split}
	\frac{\mu(A)}{\mu(B(x,r))}
	&\le C_{S,p} r^p \frac{\rcapa_p(A,B(x,2r))}{\mu(B(x,r))}\\
	&\le C_{S,p} C_d \frac{\rcapa_p(A,B(x,2r))}{\rcapa_p(B(x,r),B(x,2r))} \textrm{\quad by \eqref{eq:rcapa over ball-measure}}.
	\end{split}
	\end{equation}
	
	We recall the following result in \cite[Lemma 11.22]{BB}. Note that the cutoff proof also applies when $p=1$.
	\begin{lemma}
		Let $x\in X$, $r>0$, and $A\subset B(x,r)$.
Then for every $1<s<t$ with $tr<\tfrac{1}{4}\diam X$, we have 
		\begin{equation}\label{eq: rcapa-compare-in-ball}
			\rcapa_1(A,B(x,tr))\le \rcapa_1(A,B(x,sr))\le C_S\left(1+\frac{t}{s-1}\right)\rcapa_1(A,B(x,tr)),
		\end{equation}
		where $C_S$ is the constant in \eqref{eq:Sobolev-ineq.}.
	\end{lemma}

	\begin{definition}\label{def:1 fine topology}
		We say that $A\subset X$ is $1$-thin at the point $x\in X$ if
		\[
		\lim_{r\to 0}\frac{\rcapa_1(A\cap B(x,r),B(x,2r))}{\rcapa_1(B(x,r),B(x,2r))}=0.
		\]
		We also say that a set $U\subset X$ is \emph{$1$-finely open} if $X\setminus U$ is $1$-thin at every $x\in U$.
Then we define the \emph{$1$-fine topology} as the collection of $1$-finely open sets on $X$.
		
		We denote the $1$-fine interior of a set $H\subset X$, i.e. the largest $1$-finely open set contained in $H$, by $\fineint H$.
We denote the $1$-fine closure of $H$, i.e. the smallest $1$-finely closed set
		containing $H$, by $\overline{H}^1$.
The $1$-fine boundary of $H$
		is $\partial^1 H\coloneqq \overline{H}^1\setminus \fineint H$.
		The $1$-base $b_1 H$ is defined as the set of points
		where $H$ is \emph{not} $1$-thin.
		
		For a function $w$, we then get the concepts of $1$-fine continuity
		and $1$-fine limit at a point $x\in X$.
The latter is denoted by $\finelim_{y\to x}w(y)$
		and defined using punctured neighborhoods.
	\end{definition}
	
	See \cite[Section 4]{L-FC} for discussion on Definition \ref{def:1 fine topology},
	and for a proof of the fact that the
	$1$-fine topology is indeed a topology.
	Using \eqref{eq:rcapa over ball-measure}, we see that a set $A\subset X$ is $1$-thin at $x\in X$ if and only if 
	\[
	\lim_{r\to 0} r\frac{\rcapa_1(A\cap B(x,r),B(x,2r))}{\mu(B(x,r))}=0,
	\]
	which was the formulation used in \cite{L-FC}.

	
	We note that the analogous concepts when $1<p<\infty$ are well known in the literature.
	In this case, the definition of $p$-thinness is formulated using a suitable integral,
	and in fact the $p$-fine topology is known to be the coarsest topology that makes all $p$-superharmonic
	functions defined on arbitrary open subsets of $X$ continuous; see \cite[Theorem 1.1]{BBL-WC15}.
	We will not discuss these concepts here,
	however see e.g. the monographs \cite{BB,MZ}.
	The case $p=1$ has been studied much less, but now we will record certain results that are known
	in this case.
	
	By \cite[Corollary 6.12]{L-CK}
	we know that for an arbitrary set $U\subset X$,
	\begin{equation}\label{eq:finely and quasiopen}
		U\textrm{ is }1\textrm{-quasiopen }\ \Longleftrightarrow\ \ U=V\cup N\
		\textrm{ where }\ V\textrm{ is 1-finely open and }\mathcal \capa_1(N)=0.
	\end{equation}
	According to \cite[Corollary 3.5]{L-Fed}, the $1$-fine closure
	of $A\subset X$ can be characterized as:
	\begin{equation}\label{eq:characterization of fine closure}
		\overline{A}^1=A\cup b_1 A.
	\end{equation}
	By \cite[Proposition 3.3]{L-Fed} we know that if $A\subset W$ for an open set $W\subset X$, then 
	\begin{equation}\label{eq:capa of fine int}
		\capa_1(\overline{A}^1)=\capa_1(A)\quad \textrm{and }\quad
		\rcapa_1(A,W)=\rcapa_1(\overline{A}^1\cap W,W) .
	\end{equation}
	
	\begin{lemma}\label{lem:wedge and vee capacity}
		Let $u\in\BV_{\loc}(\Om)$, let $x\in\Om$, and let $r>0$ be such that $B(x,2r)\Subset\Om$.
Then, for every $t\in\R$,
		\[
		\rcapa_1(\{u^{\wedge}>t\}\cap B(x,r),B(x,2r))=\rcapa_1(\{u^{\vee}>t\}\cap B(x,r),B(x,2r)).
		\]
	\end{lemma}
	\begin{proof}
		The inequality ``$\le$''
		follows from $u^\wedge\le u^\vee$.
		For the opposite inequality, without loss of generality, we may assume $u\in \BV(\Om)$.
		By the coarea formula \eqref{eq:coarea}, there is a countable dense set $T\subset\R$ such that $E_s\coloneqq\{u>s\}$ has finite perimeter in $\Om$ for every $s\in T$.
		Let $N\coloneqq\Om\cap \bigcup_{s\in T}(\partial^*E_s\setminus\Sigma_\gamma E_s)$; then $\mathcal H(N)=0$ by \eqref{eq: measure theoretic boundary-gamma}.

		Consider $z\in\Om\setminus N$ with $u^\vee(z)>t$, and choose $s\in T$ with $t<s<u^\vee(z)$.
		Then $z\notin O_{E_s}$.
		If $z\in I_{E_s}$, the set $E_s$ has density one at $z$; otherwise $z\in\partial^*E_s$, and the definition of $N$ gives $z\in\Sigma_\gamma E_s$.
		Consequently,
		\[
		\liminf_{\rho\to0}\frac{\mu(B(z,\rho)\cap E_s)}{\mu(B(z,\rho))}\ge\gamma.
		\]
		Since $u=u^\wedge$ almost everywhere, $E_s\subset\{u^\wedge>t\}$ up to a $\mu$-null set.
		By \eqref{eq:rcapa in small ball} we know that $\{u^\wedge>t\}$ is not $1$-thin at $z$.
		Hence $\{u^\vee>t\}\setminus N\subset\overline{\{u^\wedge>t\}}^1$.

		Using \eqref{eq:capa of fine int} with $\{u^\wedge>t\}\cap B(x,r)$, followed by $\capa_1(N)=0$, we obtain
		\begin{align*}
		\rcapa_1(\{u^\wedge>t\}\cap B(x,r),B(x,2r))
		&=\rcapa_1(\overline{\{u^\wedge>t\}}^1\cap B(x,r),B(x,2r))\\
		&\ge\rcapa_1((\{u^\vee>t\}\setminus N)\cap B(x,r),B(x,2r))\\
		&=\rcapa_1(\{u^\vee>t\}\cap B(x,r),B(x,2r)).
		\end{align*}
	\end{proof}

		
		
		
		
	
	For any $u\in N^{1,1}_{\loc}(X)$, and $B=B(x,r)$,
	we have the following Maz'ya's inequality from \cite[Theorem 6.21]{BB}:
	\begin{equation}\label{eq:Mazya inequality}
		\vint{2B}|u|\,d\mu\le \frac{C_{\textrm{maz}}}{\rcapa_1(B\cap \{u=0\},2B)}\int_{2\lambda B}g_u\,d\mu.
	\end{equation}
	Here $C_{\textrm{maz}}$ depends only on doubling and Poincar\'e constants.
	

	For BV functions, we have the following version.
	\begin{lemma}\label{lem:BV-Mazya}
		Let $u\in\BV_{\loc}(\Om)$ be nonnegative, and let $B=B(x,r)$ satisfy $B(x,2\lambda r)\Subset\Om$.
Then
		\begin{equation}\label{eq:Mazya inequality BV}
			\vint{B(x,2r)}u\,d\mu\le \frac{C_1C_{\textrm{maz}}}{\rcapa_1(B(x,r)\cap \{u^{\vee}=0\},B(x,2r))}\|Du\|(B(x,2\lambda r)),
		\end{equation}
        where $C_1$ is the constant in \eqref{eq:capacities continuity-comparable}.
	\end{lemma}
	\begin{proof}
		By \eqref{eq: def-BV-metric}, choose locally Lipschitz functions $u_j$ on
		$B(x,2\lambda r)$ and upper gradients $g_{u_j}$ such that
		\[
		u_j\to u\quad\text{in }L_{\loc}^1(B(x,2\lambda r))
		\quad\text{and}\quad
		\int_{B(x,2\lambda r)}g_{u_j}\,d\mu\to\|Du\|(B(x,2\lambda r)).
		\]
		By \cite[Theorem 3.2]{L-SP}, after passing to a subsequence, we have $\limsup_{j\to\infty}u_j(y)\le u^\vee(y)$ for every $y\in B(x,r)\setminus N$, where $\mathcal H(N)=0$.
		Fix $\eps>0$.
Denote
		\[
		A_j\coloneqq \{y\in B(x,r)\colon u_k(y)\le \eps \textrm{ for all }k\ge j\},\quad j\in\N.
		\]
		Then by Lemma \ref{lem:capacities continuity}, we have
		\begin{align*}
			\liminf_{j\to\infty}\rcapa_1(B(x,r)\cap \{u_j\le \eps\},B(x,2r))
			&\ge \liminf_{j\to\infty}\rcapa_1(A_j,B(x,2r))\\
			&\ge C_1^{-1}\rcapa_1(B(x,r)\cap \{u^{\vee}=0\}\setminus N,B(x,2r))\\
			&=C_1^{-1}\rcapa_1(B(x,r)\cap \{u^{\vee}=0\},B(x,2r)),
		\end{align*}
		using also \eqref{eq:capa and Hausdorff measure}.
		We denote the positive part of a function by $v_+\coloneqq\max\{v,0\}$. 
		Now we have
		\begin{equation}\label{eq:mazya}
			\begin{split}
				\vint{B(x,2r)}u\,d\mu-\eps
				&\le \vint{B(x,2r)}(u-\eps)_+\,d\mu
				\le \liminf_{j\to\infty}\,\vint{B(x,2r)}(u_j-\varepsilon)_+\,d\mu\\
				&\le \limsup_{j\to\infty}\,\frac{C_{\textrm{maz}}}{\rcapa_1(B(x,r)\cap \{u_j\le \eps\},B(x,2r))}\int_{B(x,2\lambda r)}g_{u_j}\,d\mu
				\quad\textrm{by }\eqref{eq:Mazya inequality}\\
				&\le \frac{C_1C_{\textrm{maz}}}{\rcapa_1(B(x,r)\cap \{u^{\vee}=0\},B(x,2r))}\|Du\|(B(x,2\lambda r)).
			\end{split}
		\end{equation}
		Letting $\eps\to 0$, we get the claim.
	\end{proof}
	
	It will be convenient to consider the following class of measures that are, in a suitable sense, absolutely
	continuous with respect to $\rcapa_1$.
	
	\begin{definition}\label{def:class of measures}
		Given $C_*\ge1$, let $\mathcal M(\Om,C_*)$ denote the class of positive Radon measures $\nu$ on $\Om$ with the following property: for every open set $W\subset\Om$ and every $\delta>0$, there exists $R=R(\nu,W,\delta)>0$ such that, for every $0<r\le R$ and every Borel set $A\subset X$, the condition
		\[
		\frac{\rcapa_1(A\cap B(x,r),B(x,2r))}{\rcapa_1(B(x,r),B(x,2r))}<\delta\quad\textrm{for every }x\in W
		\]
		implies $\nu(A\cap W)\le C_*\delta\nu(W)$. 
	\end{definition}
    
    

        

We conclude the preliminaries with a consequence of the BV--Poincar\'e inequality that will be used to obtain Newtonian representatives.

\begin{lemma}\label{lem:BV-Sobolev density}
Let $\Om\subset X$ be open, let $1\le p\le\infty$, and suppose that $u\in\BV_{\loc}(\Om)$
and $0\le a\in L^p(\Om)$ satisfy $d\|Du\|\le a\,d\mu$ in $\Om$.
Then there exists $v\in N^{1,p}_{\loc}(\Om)$ such that $v=u$ $\mu$-almost everywhere.

If $p<\infty$, then $v$ has a $p$-weak upper gradient $g\in L^p(\Om)$ satisfying $g\le Ca$ $\mu$-a.e.
If $p=\infty$, then $v$ can be chosen locally Lipschitz, with $\Lip_v\le Ca$ $\mu$-a.e.
Here $C$ depends only on the doubling and Poincar\'e constants.
\end{lemma}

\begin{proof}
By the BV--Poincar\'e inequality \eqref{eq: BV-Poincare} and doubling, we have
\begin{equation}\label{eq:BV-density-Poincare}
\vint{B}|u-u_B|\,d\mu
\le Cr\vint{\lambda B}a\,d\mu
\end{equation}
whenever $B=B(x,r)$ and $\lambda B\Subset\Om$.

Suppose first that $p<\infty$.
We have $u\in L^p_{\loc}(\Om)$ by assumption if $p=1$.
For $p>1$, fix $B=B(x_0,r)$ with $8\lambda B\Subset\Om$.
By \cite[Theorem~9.4]{Ha03} applied to \eqref{eq:BV-density-Poincare}, we have 
\[
|u(x)-u(y)|\le Cd(x,y)\bigl(\mathcal M(a\ch_{8\lambda B})(x)+\mathcal M(a\ch_{8\lambda B})(y)\bigr)
\quad\text{for }x,y\in B\setminus Z,
\]
where $\mu(Z)=0$.
Here $\mathcal M$ is the Hardy--Littlewood maximal operator defined in \eqref{eq:HL maximal function} and $a\ch_{8\lambda B}$ is extended by zero to $X$.
Since $p>1$, by the $L^p$ boundedness of the maximal operator, we have
\[
\|\mathcal{M}(a\ch_{8\lambda B})\|_{L^p(B)}\le C_p \|a\|_{L^p(8\lambda B)}<\infty.
\]
Fixing $y\in B\setminus Z$ with $|u(y)|+\mathcal M(a\ch_{8\lambda B})(y)<\infty$, for $\mu$-a.e. $x\in B$, we have
\[
|u(x)|\le |u(y)| +2Cr\bigl(\mathcal{M}(a\ch_{8\lambda B})(x)+\mathcal{M}(a\ch_{8\lambda B})(y) \bigr).
\]
Then 
\[
\|u\|_{L^p(B)}\le \bigl(|u(y)|+2Cr \mathcal{M}(a\ch_{8\lambda B})(y) \bigr)\mu(B)^{1/p}+ 2Cr\|\mathcal{M}(a\ch_{8\lambda B})\|_{L^p(B)}<\infty.
\]
This implies $u\in L^p(B)$, and hence $u\in L^p_{\loc}(\Om)$.

Choose a countable cover $\{B_j\}$ of $\Om$ consisting of balls as above.
By H\"older's inequality, \eqref{eq:BV-density-Poincare} implies the $p$-Poincar\'e inequality.
Applying \cite[Theorem~1.2]{He12} with $\Phi(t)=t^p$ on each $B_j$ we find $v_j\in N^{1,p}(B_j)$ with $v_j=u$ $\mu$-a.e. and a $p$-weak upper gradient $g_{v_j}\le Ca$ $\mu$-a.e.
The functions $v_j$ are $p$-quasicontinuous by \eqref{eq:quasicontinuous}, and agree $p$-q.e. on overlaps by \cite[Proposition~5.23]{BB}.
After modifications on a set of zero $p$-capacity, we define a representative $v\in N^{1,p}_{\loc}(\Om)$.
Choose a nonnegative Borel representative of $a$.
Then $Ca$ is a $p$-weak upper gradient on each $B_j$ and also on $\Om$.
Thus we may take $g=Ca\in L^p(\Om)$.

For $p=\infty$, fix $q=2+\log_2 C_d$.
By H\"older's inequality and the local Morrey estimate in \cite[Theorem~9.7(3)]{Ha03}, there is a continuous representative $v$ satisfying
\[
\osc_{B(x,r)}v
\le Cr\left(\vint{B(x,6\lambda r)}a^q\,d\mu\right)^{1/q}
\]
whenever $B(x,6\lambda r)\Subset\Om$; local representatives agree on overlaps by continuity.
Since $a\in L^\infty(\Om)$, we have that $v$ is locally Lipschitz.
Applying the Lebesgue differentiation theorem to $a^q$, we have $\Lip_v\le Ca$ almost everywhere.
\end{proof}

	\section{Proof of Theorems \ref{thm:main BV} and \ref{thm:main Sob}}\label{sec:main theorems}

\begingroup
\emergencystretch=3em

	In this section we prove Theorems \ref{thm:main BV} and  \ref{thm:main Sob}.

	\subsection{``Only if'' direction}\label{sec:only if}
	
	In this subsection we prove the ``only if'' direction of Theorems
	\ref{thm:main BV} and \ref{thm:main Sob}, in a slightly stronger form.
	
	First we record the following ``Borel regularity''.
	
	\begin{lemma}\label{lem:weak Borel regularity}
		Let $D\subset X$ be $\mu$-measurable.
Then there exists a Borel set $D^*\supset D$ such that for every $x\in X$ and $r>0$,
		\[
		\mu(D\cap B(x,r))=\mu(D^*\cap B(x,r))
		\]
		and
		\[
		\rcapa_1(D^*\cap B(x,r),B(x,2r))\le C_1\rcapa_1(D\cap B(x,r),B(x,2r))
		\]
		
	\end{lemma}
	\begin{proof}
		Choose a Borel set $D_\mu\supset D$ with $\mu(D_\mu\setminus D)=0$.
		Arguing as in the proof of \cite[Lemma 4.3]{L-CK}, choose a Borel set $D_{\BV}\supset D$ such that
	        \[
	        \rcapa_{\BV}(D_{\BV}\cap B(x,r),B(x,2r))=\rcapa_{\BV}(D\cap B(x,r),B(x,2r))
	        \]
	        for every ball $B(x,r)$, and set $D^*\coloneqq D_\mu\cap D_{\BV}$.
	        Then $D\subset D^*\subset D_\mu$, which proves the measure equality.
	        Combining this with \eqref{eq:BV capa and 1-capa-relative}, we obtain
        \[
	        \rcapa_1(D^*\cap B(x,r),B(x,2r))\le C_1\rcapa_1(D\cap B(x,r),B(x,2r)).
	        \]
	\end{proof}

	The following lemma makes the lower and upper density estimates for the measure-theoretic boundary uniform outside a set of arbitrarily small $\mathcal H$-measure.
		
	\begin{lemma}\label{lem: uniform density outside small set}
	Let $E$ be of finite perimeter in $\Om$.
	Then for every $\varepsilon>0$, there exist a Borel set $S_\varepsilon\subset \partial^*E\cap \Om$ and a radius $r_\varepsilon>0$ satisfying
	\[
	\mathcal H((\partial^*E\cap \Om)\setminus S_\varepsilon)<\varepsilon,
	\]
	and
	\[
	\frac{\lambda\gamma}{4C_PC_d^{3+\log_2 \lambda}}\frac{\mu(B(x,r))}{r}
	\le 
	\mathcal H(B(x,r)\cap \partial^*E\cap \Om)
	\le
	2C_d\frac{\mu(B(x,r))}{r}
	\]
			for every $x\in S_\varepsilon$ and $0<r<r_\varepsilon$.
		\end{lemma}
		
	\begin{proof}
Set $S\coloneqq\partial^*E\cap\Om$. By \eqref{eq:relation-Hausdorff-perimeter}, $\mathcal H(S)\le\alpha^{-1}P(E,\Om)<\infty$.
For each $x\in S$, we have $B(x,\lambda r)\subset \Om$ for all sufficiently small $r$.
	By \cite[Eq. (2.6)]{A1}, applied with $h(B(x,r))=\mu(B(x,r))/r$ and $\nu=\mathcal H\mres (\partial^*E\cap \Om)$, we have
\begin{equation}\label{eq: upper-bound}
\limsup_{r\to0}\frac{\mathcal H(B(x,r)\cap S)}{\mu(B(x,r))/r}\le1
\end{equation}
for $\mathcal H$-almost every $x\in \partial^*E\cap \Om$.
		On the other hand, recall from \eqref{eq: measure theoretic boundary-gamma} that
		for $\mathcal{H}$-a.e. $x\in \partial^*E\cap \Om$, we have
		\[
		\liminf_{r\to 0} \min\left\{\frac{\mu(B(x,r)\cap E)}{\mu(B(x,r))}, \frac{\mu(B(x,r)\setminus E)}{\mu(B(x,r))} \right\}\ge \gamma.
		\]
		Combining this with \eqref{eq:relation-Hausdorff-perimeter} and the relative isoperimetric inequality \eqref{eq:relative isoperimetric inequality}, for $\mathcal{H}$-a.e. $x\in S$, we can obtain that 
		\begin{equation}\label{eq: lower-bound}
		\begin{split}
		&\liminf_{r\to 0} \frac{\mathcal{H}(B(x,r)\cap S)}{\mu(B(x,r))/r}\\
		&= \liminf_{r\to 0} \frac{\mathcal{H}(B(x,\lambda r)\cap S)}{\mu(B(x,\lambda r))/\lambda r}\\
		&\ge \liminf_{r\to 0} \frac{\lambda}{C_d^{2+\log_2 \lambda }} \frac{P(E,B(x,\lambda r))}{\mu(B(x, r))/r}
		\quad \textrm{by \eqref{eq:relation-Hausdorff-perimeter} and \eqref{eq:doubling-def}}\\
		&\ge \frac{\lambda}{2C_P C_d^{2+\log_2 \lambda}}
		\liminf_{r\to 0} \min\left\{\frac{\mu(B(x,r)\cap E)}{\mu(B(x,r))}, \frac{\mu(B(x,r)\setminus E)}{\mu(B(x,r))} \right\}
		\quad \textrm{by }\eqref{eq:relative isoperimetric inequality}
		\\
		&\ge \frac{\lambda \gamma}{2C_P C_d^{2+\log_2 \lambda}}.
		\end{split}
		\end{equation}
		Set $c_0=\frac{\lambda \gamma}{2C_P C_d^{2+\log_2 \lambda}}$.
		Let $r_j=2^{-j}$ and define
		\[
		f_j(x)=\frac{\mathcal H(B(x,r_j)\cap S)}{\mu(B(x,r_j))/r_j}.
		\]
		Then by \eqref{eq: upper-bound} and  \eqref{eq: lower-bound}, we have
		\begin{equation}
			c_0\le \liminf_{j\to\infty}f_j(x)\le \limsup_{j\to \infty}f_j(x)\le 1
		\end{equation}
		for $\mathcal H$-a.e. $x\in S$.
		
		Set
		\[
		g_j(x)=\inf_{k\ge j} f_k(x),
		\qquad
		h_j(x)=\sup_{k\ge j} f_k(x).
		\]
		Then $g_j(x)$ is increasing and $g_j(x)\to g(x)$; $h_j(x)$ is decreasing and $h_j(x)\to h(x)$ with $c_0\le g(x)\le h(x)\le 1$
		for $\mathcal H$-a.e. $x\in S$.
		
		Fix $\varepsilon>0$.
By Egorov's theorem, there exists a Borel set $S_\varepsilon\subset S$ such that 
        $\mathcal H(S\setminus S_\varepsilon)<\varepsilon$
		and both $g_j\to g$ and $h_j\to h$ uniformly on $S_\varepsilon$.
		Hence there exists $J\in\mathbb N$ such that for all $x\in S_\varepsilon$,
		\[
		g_J(x)\ge \frac{c_0}{2}
		\qquad
		h_J(x)\le 2 .
		\]
		
		Since $g_J(x)=\inf_{k\ge J}f_k(x)$ and $h_J(x)=\sup_{k\ge J}f_k(x)$,
		we obtain for all $x\in S_\varepsilon$ and $k\ge J$,
		\[
		\frac{c_0}{2}
		\le
		\frac{\mathcal H(B(x,r_k)\cap S)}{\mu(B(x,r_k))/r_k}
		\le
		2 .
		\]
		
		Finally, for every $0<r<r_J$, there exists $k\ge J$ such that $r_{k+1}\le r\le r_k$. 
		By the doubling property of $\mu$, we have
		\begin{equation*}
			\begin{split}
				\frac{c_0}{2C_d}
				\le  \frac{\mathcal{H}(B(x,r_{k+1})\cap S)}{\mu(B(x,r_k))/r_{k+1}}
				\le \frac{\mathcal{H}(B(x,r)\cap S)}{\mu(B(x,r))/r}
				\le \frac{\mathcal{H}(B(x,r_k)\cap S)}{\mu(B(x,r_{k+1}))/r_k}\le 2C_d
			\end{split}
		\end{equation*}
		Setting $r_\varepsilon=r_J$ and recalling that $c_0=\frac{\lambda \gamma}{2C_P C_d^{2+\log_2 \lambda}}$,
		it follows that
		\[
		\frac{\lambda\gamma}{4C_PC_d^{3+\log_2\lambda}}
		\frac{\mu(B(x,r))}{r}
		\le
		\mathcal H(B(x,r)\cap S)
		\le
		\frac{2C_d\mu(B(x,r))}{r}
		\]
		for all $x\in S_\varepsilon$ and $0<r<r_\varepsilon$.
		\end{proof}

    Recall that we always assume $\Om\subset X$ to be an open set.
	Recall also the definition of the class of Radon measures $\mathcal M(\Om,C_*)$ from Definition \ref{def:class of measures}.
	In the following two lemmas we show that this class is quite broad.


    \begin{lemma}\label{lem:measures belonging to M}
	Let $E$ satisfy $P(E,\Om)<\infty$, and let $u\in\BV(\Om)$.
	Then $\mathcal H\mres(\partial^*E\cap\Om)$, $\|Du\|$, and $\mu$ belong to $\mathcal M(\Om,C_{\mathcal M})$, where
	\[
    C_{\mathcal M}\coloneqq \max\left\{\frac{2\hat{C}C_d}{\alpha}, 2C_SC_d^4\right\}=\max\left\{\frac{64(C_d^5+2)C_d^{2\log_2 \lambda+13}C_P^2 }{\lambda\gamma\alpha}, 2C_SC_d^4\right\},
    \]
	and $\gamma,\, \alpha$, $C_S$ are the constants in \eqref{eq:strong boundary}, \eqref{eq:relation-Hausdorff-perimeter} and \eqref{eq:Sobolev-ineq.} respectively and depend only on the doubling and Poincar\'e constants. 
	\end{lemma}
	\begin{proof}
	\textbf{Part 1: $\mathcal H\mres {\partial ^*E}$.}
    
	Fix $\delta>0$ and an open set $W\subset\Om$.
	By Lemma \ref{lem: uniform density outside small set}, there exist $S_0\subset\partial^*E\cap W$ and $r_0>0$ such that
		\begin{equation}\label{eq: exceptional-on boundary}
			\mathcal{H}((\partial^*E\cap W)\setminus S_0)\le  \frac{16(C_d^5+2)C_d^{2\log_2 \lambda+12}C_P^2}{\lambda\gamma}\delta \mathcal H(\partial^*E\cap W)
		\end{equation}
		and 
		\begin{equation}\label{eq: uniform-choose r}
			\frac{\lambda\gamma}{4C_PC_d^{3+\log_2 \lambda}}\frac{\mu(B(x,r))}{r}\le 
			\mathcal{H}(B(x,r)\cap \partial^*E\cap W)
			\le 2C_d\frac{\mu(B(x,r))}{r}
		\end{equation}
		for all $x\in S_0$ and $0<r<r_0$.
By decreasing $r_0$ if necessary, assume also that $r_0<\diam X/8$.
		
		Then consider $0<r< r_0/(20\lambda)$,
		and a Borel set $A\subset X$ such that
		\begin{equation}\label{eq:first capa estimate}
			\frac{\rcapa_1(B(x,r)\cap A, B(x,2r))}{\rcapa_1(B(x,r),B(x,2r))}<\delta 
			\quad
            \textrm{ for all } x\in W.
		\end{equation}
		We claim that for every $x\in W$,
		\begin{equation}\label{eq:Hausdorff-less than-rcapa}
			\mathcal{H}(S_0\cap A\cap B(x,r))\le C_2\rcapa_1(S_0\cap A\cap B(x,r),B(x,2r)),
		\end{equation}
		where
		\[
		C_2=4(C_d^5+2)C_d^{\log_2(5\lambda)+2}C_P.
		\]
		
	For every $x\in W$,
	by the definition of the variational $1$-capacity, we find a function $u_r\in N_0^{1,1}(B(x,2r))$ such that $u_r= 1$ in a neighborhood of $S_0\cap A\cap B(x,r)$, and 
    \[
    \|Du_r\|(X)\le \int_X g_{u_r}\,d\mu\le \rcapa_1(S_0\cap A\cap B(x,r),B(x,2r))+\varepsilon.
    \]
    Then by the coarea formula \eqref{eq:coarea}, we find a set $\{u_r>t\}$ for some $0<t<1$, for which 
    \begin{equation}\label{eq: choose ur}
        P(\{u_r>t\}, X)\le \|Du_r\|(X)\le \rcapa_1(S_0\cap A\cap B(x,r),B(x,2r))+\varepsilon.
    \end{equation}
    Let 
    \[
    E_{r,t}^*\coloneqq\left\{y\in \{u_r>t\}:\lim_{s\to 0}\frac{\mu(\{u_r>t\}\cap B(y,s))}{\mu(B(y,s))}=1 \right\}.
    \]
    By the Lebesgue differentiation theorem, it follows that $\mu(E_{r,t}^*)=\mu(\{u_r>t\})<\infty$, and consequently, we have $P(E_{r,t}^*,X)=P(\{u_r>t\},X)<\infty$.
We have 
    \[
    S_0\cap A\cap B(x,r)\subset \textrm{int} \{u_r=1\}\subset E_{r,t}^*.
    \]
	For every $y\in S_0\cap A\cap B(x,r)$, put 
	\[
	\theta_y(s)\coloneqq\frac{\mu(B(y,s)\cap\{u_r>t\})}{\mu(B(y,s))}.
	\]
	Since $y\in E_{r,t}^*$, we have $\theta_y(s)\to1$ as $s\to0$.
		Since $u_r=0$ $\mu$-a.e. in $X\setminus B(x,2r)$, we can find a ball $B(y',r/2)\subset B(y,4r)\setminus B(x,2r)$, and by \eqref{eq: doubling-ball-diff-center} we have that
		\[
		\frac{\mu(B(y, 4r)\cap \{u_r>t\})}{\mu(B(y,4r))}
		\le \frac{\mu(B(x,2r))}{\mu(B(y,4r))}
		\le 1-\frac{\mu(B(y',r/2))}{\mu(B(y,4r))}\le 1-\frac{1}{C_d^5}.
		\]
Let $n_y\ge1$ be the smallest integer such that
$\theta_y(2^{-n_y}4r)>1-(2+C_d^5)^{-1}$, and set $s_y\coloneqq2^{-(n_y-1)}4r$.
Then $0<s_y\le4r$, by the minimality and doubling we have
\[
\frac{1}{2+C_d^5}\le\frac{1-(2+C_d^5)^{-1}}{C_d}
<\frac{\theta_y(s_y/2)}{C_d}
\le\theta_y(s_y)\le1-\frac{1}{2+C_d^5}.
\]
Here $(1-(2+C_d^5)^{-1})/C_d\ge(2+C_d^5)^{-1}$ follows from $1+C_d^5\ge C_d$.
		By the relative isoperimetric inequality \eqref{eq:relative isoperimetric inequality}, we have
		\begin{equation}\label{eq: estimate-ball-perimeter}
			\begin{split}
				\frac{\mu(B(y,s_y))}{s_y}
				&\le\left(2+C_d^5\right)\frac{\min\{\mu(B(y,s_y)\cap \{u_r>t\}), \mu(B(y,s_y)\setminus \{u_r>t\})\}}{s_y}\\
				&\le (4+2C_d^5)C_P P(\{u_r>t\},B(y,\lambda s_y)).
			\end{split}
		\end{equation}
		By the $5$-covering lemma, we can find disjoint balls $\{B(y_j,\lambda s_{y_j})\}_{y_j\in S_0\cap A\cap B(x,r)}$ such that
		\[
		S_0\cap A\cap B(x,r)
		\subset\bigcup_{j=1}^{\infty}B(y_j,5\lambda s_{y_j}).
		\]
		Thus
		\begin{equation*}
			\begin{split}
			&\quad \mathcal{H}(S_0\cap A\cap B(x,r))\\
			&\le \sum_j\mathcal{H}(S_0\cap A\cap B(y_j,5\lambda s_{y_j})) \\
			&\le 2C_d\sum_j \frac{\mu(B(y_j,5\lambda s_{y_j}))}{5\lambda s_{y_j}}
			\quad \textrm{by }\eqref{eq: uniform-choose r} \textrm{ and } 5\lambda s_{y_j}\le 20\lambda r< r_0\\
				&\le 2C_d^{\log_2(5\lambda)+2}\sum_j\frac{\mu(B(y_j,s_{y_j}))}{s_{y_j}} 
				\quad \textrm{by doubling} \\
				&\le C_2
				\sum_j P(\{u_r>t\},B(y_j,\lambda s_{y_j}))
				\quad \textrm{by } \eqref{eq: estimate-ball-perimeter} \\
				&\le C_2 P(\{u_r>t\}, X)\\
				&\le 
				C_2 (\varepsilon+\rcapa_1(S_0\cap A\cap B(x,r),B(x,2r)))
				\quad \textrm{by }\eqref{eq: choose ur}.
			\end{split}
		\end{equation*}
	Letting $\varepsilon\to 0$, the claim follows with $C_2=4(C_d^5+2)C_d^{\log_2(5\lambda)+2}C_P$.
	
    Consider the covering $\{B(x,r/5)\}_{x\in S_0\cap A}$ . By the $5$-covering lemma, we can find countable pairwise disjoint balls $\{B(x_k,r/5)\}_{x\in S_0\cap A}$ such that
    \[
    S_0\cap A\subset \bigcup_{k=1}^{\infty}B(x_k,r).
    \]
    Thus 
    \begin{equation}
    \begin{split}
    &\quad \mathcal{H}(S_0\cap A)\\
    &\le \sum_k\mathcal{H}(S_0\cap A\cap B(x_k,r))\\
    &\le C_2\sum_{k}\rcapa_1(S_0\cap A\cap B(x_k,r), B(x_k,2r))
    \quad\textrm{by } \eqref{eq:Hausdorff-less than-rcapa}\\
    &\le  C_2\delta\sum_k\rcapa_1(B(x_k,r),B(x_k,2r))
    \quad\textrm{by } \eqref{eq:first capa estimate}\\
    &\le  C_2C_d\delta\sum_k \frac{\mu(B(x_k,r))}{r}
    \quad \textrm{by }\eqref{eq:rcapa over ball-measure}\\
    &\le \frac{4C_2C_PC_d^{7+\log_2\lambda}}{\lambda \gamma}\delta \sum_k \mathcal{H}(\partial^*E\cap W\cap B(x_k,r/5))
    \quad\textrm{by }\eqref{eq: uniform-choose r} \textrm{ and doubling}\\
    &\le \frac{4C_2C_PC_d^{7+\log_2\lambda}}{\lambda \gamma}\delta\mathcal{H}(\partial^*E\cap W).
    \end{split}
    \end{equation}
    Recall that 
    $C_2=4(C_d^5+2)C_d^{\log_2(5\lambda)+2}C_P$.
    Combining this with \eqref{eq: exceptional-on boundary},
    it follows that
    \[
    \mathcal H(\partial^*E\cap A\cap W)\le \mathcal{H}(A\cap S_0)+\mathcal{H}((\partial^*E \cap W)\setminus S_0)\le \hat{C}\delta \mathcal{H}(\partial^*E\cap W),
    \]
    with $\hat{C}=32(C_d^5+2)C_d^{2\log_2 \lambda+12}C_P^2/(\lambda\gamma)$.

    \textbf{Part 2: $\|Du\|$.}		
    
	Fix $\delta>0$ and an open set $W\subset\Om$.
	Assume towards a contradiction that there is a sequence $r_j\to 0$
and Borel sets $A_{j}\subset X$ such that
		\[
		\frac{\rcapa_1(B(x,r_j)\cap A_j,B(x,2r_j))}{\rcapa_1(B(x,r_j),B(x,2r_j))}<\delta
		\quad \textrm{for all } x\in W,
		\]
	and
	\begin{equation}\label{eq: contradiction-def}
		\|Du\|(A_{j}\cap W)>\frac{2\hat{C}C_d}{\alpha} \delta \|Du\|(W).
	\end{equation}
	By \eqref{eq:relation-Hausdorff-perimeter} and the coarea formula \eqref{eq:coarea}, we get 
	\[
	\|Du\|(A_j\cap W)=\int_{-\infty}^{\infty} \|D\chi_{\{u>t\}}\|(A_j\cap W)\,dt\le C_d\int_{-\infty}^{\infty} \mathcal{H}(\partial^*\{u>t\}\cap A_j\cap W)\,dt.
	\]
	By Part 1, we know that for a.e. $t\in \R$,
	\begin{equation}\label{eq:using part 1}
		\limsup_{j\to \infty} \mathcal{H}(\partial^*\{u>t\}\cap A_j\cap W)\le \hat{C} \delta \mathcal{H}(\partial^*\{u>t\}\cap W).
	\end{equation}
		
It follows that
	\begin{equation*}
		\begin{split}
			\limsup_{j\to \infty} \|Du\|(A_j\cap W)
			&\le\limsup_{j\to \infty} C_d \int_{-\infty}^{\infty}\mathcal{H}(\partial^*\{u>t\}\cap A_j\cap W)\,dt\\
			&\le \hat{C} C_d \delta\int_{-\infty}^{\infty}\mathcal{H}(\partial^*\{u>t\}\cap W)\,dt
			\quad\textrm{by } \eqref{eq:using part 1} \\
			&\le\frac{\hat{C}C_d}{\alpha} \delta\|Du\|(W)
			\quad\textrm{by } \eqref{eq:relation-Hausdorff-perimeter}.
		\end{split}
	\end{equation*}
	This contradicts \eqref{eq: contradiction-def}, and so in fact we have $\|Du\|\in \mathcal{M}(\Om, 2\hat{C}C_d/\alpha)$.
		
	\textbf{Part 3: $\mu$.}
		
	Fix $\delta>0$ and an open set $W\subset\Om$.
We can assume that $\mu(W)<\infty$.
Define 
	\[
	W_R\coloneqq\{x\in W\colon \dist(x, X\setminus W)>R\}.
	\]
	For sufficiently small $0<R<\tfrac18\diam X$, we have $\mu(W\setminus W_R)<C_SC_d^4 \delta \mu(W)$.
Consider $0<r\le R$ and a Borel set $A\subset X$ such that 
	\[
	\frac{\rcapa_1(A\cap B(x,r), B(x,2r))}{\rcapa_1(B(x,r),B(x,2r))}<\delta 
	\quad
    \textrm{ for all } x\in W.
	\]
	By \eqref{eq:rcapa in small ball}, we also have
	\[
	\frac{\mu(A\cap B(x,r))}{\mu(B(x,r))}\le C_SC_d\delta
	\quad 
    \textrm{ for all } x\in W.
	\]
	Consider the covering $\{B(x,r/5) \}_{x\in W_R}$.
By the $5$-covering theorem, see \cite[Section 3.3]{HKSTbook}, we can pick an at most countable collection of pairwise disjoint balls $B(x_k,r/5)$ such that the balls $B(x_k,r)$ cover $W_R$.
Thus
	\begin{align*}
		\mu(A\cap W_R)&\le \sum_k\mu(A\cap B(x_k,r))\\
		&\le C_SC_d\delta \sum_{k}\mu(B(x_k,r))\\
		&\le C_SC_d^4\delta\sum_k\mu(B(x_k,r/5))\\
		&\le C_SC_d^4\delta \mu(W).
	\end{align*}
	In total, we get 
	\[
	\mu(A\cap W)\le \mu(A\cap W_R)+\mu(W\setminus W_R)\le 2C_SC_d^4\delta \mu(W).
	\]		
\end{proof}
	
\begin{lemma}\label{lem:h nv is also M}
Let $C_*\ge 1$ and suppose $\nu\in\mathcal M(\Om,C_*)$, and that $h\in L^1(\Om,\nu)$ is nonnegative.
Then $h\,d\nu\in\mathcal M(\Om,2C_*)$.
\end{lemma}
\begin{proof}
Fix $0<\delta<1$ and an open set $W\subset\Om$.
We can assume that $\int_W h\,d\nu>0$.
Choose a simple nonnegative function 
\[
a=\sum_{k=1}^N \alpha_k\ch_{U_k},
\]
with $\alpha_k>0$,
such that $\Vert a-h\Vert_{L^1(W,\nu)}<(\delta/2) \int_W h\,d\nu$.
By the Borel regularity of $\nu$,
we can assume that the sets $U_k\subset W$ are open.
Choose $R=\min_{k=1,\ldots,N} R(\nu,U_k,\delta)$, and consider $0<r\le R$ and a Borel set $A\subset \Om$ such that
\[
\frac{\rcapa_1(B(x,r)\cap A, B(x,2r))}{\rcapa_1(B(x,r),B(x,2r))}<\delta\quad\textrm{for all }x\in W.
\]
We get
\begin{align*}
\int_{A\cap W}h\,d\nu-(\delta/2)\int_W h\,d\nu
\le \int_{A\cap W}a\,d\nu
&= \sum_{k=1}^N \int_{A\cap W}\alpha_k\ch_{U_k}\,d\nu\\
&= \sum_{k=1}^N \alpha_k\nu(A\cap U_k)\\
&\le C_{*}\delta\sum_{k=1}^N \alpha_k\nu(U_k)\\
&=  C_{*}\delta \int_{W}a\,d\nu\\
&\le C_{*}\delta \left(\int_{W}h\,d\nu+\frac{\delta}{2}\int_W h\,d\nu\right).
\end{align*}
Thus
\[
\int_{A\cap W}h\,d\nu\le 2C_{*}\delta \int_{W}h\,d\nu.
\]
\end{proof}
	
	The Hardy--Littlewood maximal function of a function $u\in L^1_{\loc}(X)$ is defined by
	\begin{equation}\label{eq:HL maximal function}
		\mathcal M u(x)\coloneqq \sup_{r>0}\,\vint{B(x,r)}|u|\,d\mu,\quad x\in X.
	\end{equation}
	We also define a restricted version $\mathcal M_R u(x)$, with $R>0$, by requiring $0<r\le R$ in the supremum.
	
	The following weak-type estimate is standard, see \cite[Lemma 4.3]{KKST} for a proof.
While \cite{KKST} makes the extra assumption $\mu(X)=\infty$, use of this assumption can be avoided by considering $\mathcal{H}_R$ instead of $\mathcal{H}_{\infty}$.
	
	
	\begin{lemma}\label{lem:weak standard estimate maximal function}
		Let $u\in \BV(X)$.
Then for some constant $C$ depending only on $C_d, C_P,\lambda, R$, we have
		\[
		\capa_1(\{\mathcal M_R u>t\}) \le C\frac{\Vert u\Vert_{\BV(X)}}{t}\quad\textrm{for all }t>0.
		\]
	\end{lemma}
	
	Now we prove the main result of this subsection.
	
	\begin{proposition}\label{prop:necessity}
		Suppose $u\in \BV_{\loc}(\Om)$, and let $x\in \Om$ and $0<\delta<1$.
Then we have
		\[
		\Lip_{u^*,\|Du\|}^{\delta,2\lambda}(x)\le C \delta^{-1}
		\]
		for a constant $C$ depending only on the doubling and Poincar\'e constants.
	\end{proposition}
    
	Recall that in the definition of the generalized Lipschitz number \eqref{eq:generalized Lip definition},
	we consider the quantity $\frac{\osc_{U_r} w}{2r}\,\frac{\mu(B(x,Mr))}{\nu(B(x,Mr))}$
	to be zero if $\osc_{U_r} w=0$, and $\infty$ if only $\nu(B(x,Mr))$ is zero.
	
	\begin{proof}
		Consider $0<r\le \min\{1,\tfrac18\diam X\}$ sufficiently small such that $\overline{B}(x,2\lambda r)\subset \Om$.
Let $M\in \R$ be a median of $u$ in $B(x,r)$, that is,
		\[
		\frac{\mu(B(x,r)\cap \{u\ge M\})}{\mu(B(x,r))}\ge \frac{1}{2}
		\quad\textrm{and}\quad
		\frac{\mu(B(x,r)\cap \{u\le M\})}{\mu(B(x,r))}\ge \frac{1}{2}.
		\]
		By \eqref{eq:pointwise-properties}, there is an $\mathcal{H}$-negligible and thus by \eqref{eq:capa and Hausdorff measure} also $\capa_1$-negligible set $N\subset \Om$ such that every $y\in \Om\setminus N$, we have
		$(1-\gamma) u^\wedge+\gamma u^\vee\le u^*(y)\le \gamma u^\wedge+(1-\gamma)u^\vee$;
		and 
		\begin{equation}\label{eq:wedge star vee}
			((u-M)_+)^*(y)\ge \gamma ((u-M)_+)^\vee (y).
		\end{equation}
		Let $v\in \BV(X)$ be an extension of $(u-M)_+$ from $B(x,r)$ to the whole space by a standard truncation and Lipschitz cut-off argument. 
		Fix $R>2r$.
		We have $\mathcal M_Rv\ge v^*$ everywhere, and then by \eqref{eq:wedge star vee} in fact
		\[
		\mathcal{M}_Rv\ge v^*=((u-M)_+)^*\ge \gamma((u-M)_+)^\vee=\gamma (u^\vee -M)_+ \quad \textrm{in } B(x,r)\setminus N.
		\]
		By Lemma \ref{lem:weak standard estimate maximal function}, combining with \eqref{eq:capa vs rcapa} we now see that for every $r>0$, 
		\[
		\lim_{t\to\infty}\rcapa_1(\{(u^{\vee}-M)_+\ge t\}\cap B(x,r),B(x,2r))=0.
		\]
		Thus we can choose $t_+\in [0,\infty)$ such that for all $t>t_+$, we have 
		\[
		\rcapa_1(\{(u^\vee -M)_+\ge t\}\cap B(x,r),B(x,2r))\le \frac{\delta}{2C_1}\rcapa_1(B(x,r),B(x,2r)),
		\]
		and for all $t<t_+$, we have 
		\[
		\rcapa_1(\{(u^\vee -M)_+\ge t\}\cap B(x,r),B(x,2r))\ge \frac{\delta}{2C_1}\rcapa_1(B(x,r),B(x,2r)),
		\]
		where $C_1$ is the constant in \eqref{eq:capacities continuity-comparable}.
		Since $\{(u^\vee-M)_+>t_+\}=\bigcup_{t>t_+}\{(u^\vee-M)_+\ge t\}$, by Lemma \ref{lem:capacities continuity}, we have
		\begin{equation}\label{eq:precise estimate union}
			\rcapa_1(\{(u^\vee -M)_+>t_+\}\cap B(x,r),B(x,2r))\le \frac{\delta}{2}\rcapa_1(B(x,r),B(x,2r)).
		\end{equation}
		And by Lemma \ref{lem:wedge and vee capacity}, for all $t<t_+$, we have 
		\begin{equation}\label{eq:estimate from vee to wedge}
			\rcapa_1(\{(u^{\wedge}-M)_+> t\}\cap B(x,r),B(x,2r))\ge \frac{\delta}{2C_1}\rcapa_1(B(x,r),B(x,2r)).
		\end{equation}
		By the fact that $M$ is the median, for any $t\ge 0$ we have
		\begin{equation}\label{eq:distance from t0}
			\begin{split}
				\vint{B(x,r)}(t-\min\{(u-M)_+,t\})\,d\mu
				&= t-\vint{B(x,r)}\min\{(u-M)_+,t\}\,d\mu\\
				&\ge \frac{t}{2}.
			\end{split}
		\end{equation}
		Note that for every $t\in\R$,
		\begin{equation}\label{eq:vee and wedge preliminary}
			(t-\min\{(u-M)_+,t\})^{\vee}=t-\min\{(u^{\wedge}-M)_+,t\}.
		\end{equation}
		By \eqref{eq:distance from t0}, we get
		\begin{equation}
			\begin{split}
				\frac{t_+}{2}
				&\le \vint{B(x,r)}(t_+-\min\{(u-M)_+,t_+\})\,d\mu
                \le C_d\vint{B(x,2r)}(t_+-\min\{(u-M)_+,t_+\})\,d\mu
                \\
				&=\lim_{t\nearrow t_+}C_d\vint{B(x,2r)}(t-\min\{(u-M)_+,t\})\,d\mu\\
				&\le \limsup_{t\nearrow t_+}\frac{C_1C_dC_{\textrm{maz}}}{\rcapa_1(\{(u^\wedge-M)_+\ge t\}\cap B(x,r),B(x,2r))}\|D(u-M)_+\|(B(x,2\lambda r))
				\\
				&\le \frac{2C_1^2C_dC_{\textrm{maz}}}{\delta \rcapa_1(B(x,r),B(x,2r))}\|D(u-M)_+\|(B(x,2\lambda r)) 
				\quad \textrm{by } \eqref{eq:estimate from vee to wedge}
			\end{split}
		\end{equation}
		where the third inequality follows from \eqref{eq:Mazya inequality BV}, \eqref{eq:vee and wedge preliminary}.
		
		Denote $V_+\coloneqq B(x,r)\setminus (\{(u^\vee-M)_+>t_+\}\cup N)$.
		Now 
		\[
		(u^\vee-M)_+\le t_+ \quad\textrm{in }V_+ ,
		\]
		and 
		\[
		\|D(u-M)_+\|(B(x,2\lambda r))\ge \frac{\delta}{4C_1^2C_dC_{\textrm{maz}}} t_+ \rcapa_1(B(x,r),B(x,2r)).
		\]
		By \eqref{eq:precise estimate union}, we also have
		\[
		\rcapa_1(B(x,r)\setminus V_+,B(x,2r))\le \frac{\delta}{2}\rcapa_1(B(x,r),B(x,2r)).
		\]
		
		We can perform the analogous reasoning with $(u^{\wedge}-M)_-$.
		Thus we find a number $ t_-\in [0,\infty)$
		such that the set $V_-\coloneqq  B(x,r)\setminus (\{(u^{\wedge}-M)_-> t_-\}\cup N)$ satisfies
		\[
		(u^\wedge-M)_-\le t_- \quad\textrm{in }V_- ,
		\]
		and 
		\[
		\|D(u-M)_-\|(B(x,2\lambda r))\ge \frac{\delta}{4C_1^2C_dC_{\textrm{maz}}} t_- \rcapa_1(B(x,r),B(x,2r)).
		\]
		By \eqref{eq:precise estimate union}, we also have
		\[
		\rcapa_1(B(x,r)\setminus V_-,B(x,2r))\le \frac{\delta}{2}\rcapa_1(B(x,r),B(x,2r)).
		\]
		
		Define $U_r\coloneqq V_+\cap V_-$. 
		Since $u^\wedge\le u^*\le u^\vee$ outside $N$,
		\[
		(u^*-M)_+ \le  (u^{\vee}-M)_+\le t_+
		\quad\textrm{and}\quad 
		(u^*-M)_-\le  (u^{\wedge}-M)_-\le t_-
		\quad\textrm{in }U_r,
		\]
		and so
		\[
		\underset{U_r}{\osc}\, u^*\le t_++t_-.
		\]
		Using the coarea formula \eqref{eq:coarea}, we see that
		\begin{align*}
			\|D u\|( B(x,2\lambda r))
			&= \|D (u-M)_+\|( B(x,2\lambda r))
			+\|D (u-M)_-\|( B(x,2\lambda r))\\
			&\ge \frac{\delta}{4C_1^2C_dC_{\textrm{maz}}} (t_++t_-) \rcapa_1(B(x,r),B(x,2r)),
		\end{align*}
		and we also have 
		\begin{align*}
			\rcapa_1(B(x,r)\setminus U_r,B(x,2r))
			&\le \rcapa_1(B(x,r)\setminus V_+,B(x,2r))+\rcapa_1(B(x,r)\setminus V_-,B(x,2r))\\
			&\le \delta \rcapa_1(B(x,r),B(x,2r)).
		\end{align*}
		
		If $t_++t_->0$, then we obtain
		\begin{equation}\label{eq:osc final estimate}
			\begin{split}
				\frac{\osc_{U_r}u^*}{2r}\frac{\mu(B(x,2\lambda r))}{\|Du\|(B(x,2\lambda r))}
				&\le \frac{t_++t_-}{r}\frac{\mu(B(x,r))}{\rcapa_1(B(x,r),B(x,2r))} \frac{2C_1^2C_d^{3+\log_2\lambda}C_{\textrm{maz}}}{\delta} (t_++t_-)^{-1}\\
				&\le  C \delta^{-1}
			\end{split}
		\end{equation}
		with $C=2C_1^2C_d^{3+\log_2\lambda}C_SC_{\textrm{maz}}$ by \eqref{eq:rcapa over ball-measure}.
		If $t_++t_-=0$, then $\osc_{U_r}u^*=0$ and recall that we interpret the left-hand side of
		\eqref{eq:osc final estimate} to be zero.
		Since we can do this for all sufficiently small $r>0$, we get
		\[
		\Lip_{u^*,\|Du\|}^{\delta,2\lambda}(x)\le C\delta^{-1}.
		\]
	\end{proof}
	
	\begin{proof}[Proof of $\ref{item-1-thm 1.3}\Rightarrow \ref{item-2-thm 1.3}$ in Theorem \ref{thm:main BV}]
		By assumption $u\in \BV(\Om)$, and now
		by Proposition \ref{prop:necessity} with the choice $\delta=\widetilde c_*<1$, for every $x\in \Om$ we have
		\[
		\Lip_{u^*,\|Du\|}^{\widetilde c_*,2\lambda}(x)\le C {\widetilde c_*}^{-1}.
		\]
		Then clearly
		\[
		\Lip_{u^*,\nu}^{\widetilde c_*,2\lambda}(x)\le 1
		\]
		with the choice $\nu=C \widetilde c_*^{-1}\|Du\|$, which belongs to $\mathcal M(\Om,(800C_1C_S^2C_d^6\widetilde c_*)^{-1})$ by Lemma \ref{lem:measures belonging to M}, \eqref{eq:cn definition}, and the fact that scalar multiplication leaves the $\mathcal M$-constant unchanged.
	\end{proof}
	
	\begin{proof}[Proof of $\ref{item-1-thm 1.4}\Rightarrow \ref{item-2-thm 1.4}$ in Theorem \ref{thm:main Sob}]
		Since $u^*$ is unchanged by modifications of $u$ on a $\mu$-null set,
			we may assume that $u\in N^{1,p}(\Om)$.
		And now by Proposition \ref{prop:necessity} with the choice $\delta=\widetilde c_*<1$, for every $x\in \Om$ we have
		\[
		\Lip_{u^*,g_u}^{\widetilde c_*,2\lambda}(x)\le C \widetilde c_*^{-1}.
		\]
		Then clearly
		\[
		\Lip_{u^*,a}^{\widetilde c_*,2\lambda}(x)\le 1
		\]
		with the choice $a=C \widetilde c_*^{-1}g_u\in L^p(\Om)$.
	\end{proof}

	\subsection{``If'' direction}
	
	In this subsection we prove the ``if'' direction of Theorems
	\ref{thm:main BV} and \ref{thm:main Sob}, again in a slightly stronger form.
	
	We will use Whitney-type coverings consisting of balls.
	
	\begin{lemma}\label{lem:Whitney}
		Let $A\subset D\subset W$, where $W\subset X$ is an open set and $A$ is dense in $D$.
		Given a scale $0<R<\infty$ and $M\ge1$, there exists a finite or countable Whitney-type covering $\{B_l=B(x_l,r_l)\}_l$ of $A$ in $W$,
		with $x_l\in A$, $r_l\le R$, and the following properties:
		\begin{enumerate}
			\item \label{Whitney-item-1} $MB_l\subset W$ and $D\subset \bigcup_{l}\tfrac 14 B_l$, and the balls $\tfrac{1}{40}B_l$
			are disjoint;
			\item \label{Whitney-item-2}  If $MB_l\cap MB_m\neq \emptyset$, then $r_{l}\le 2r_{m}$;
			\item \label{Whitney-item-3}  $\sum_l \chi_{MB_l}(x)\le C_M= C(C_d,M)$ for all $x\in W$;
			\item \label{Whitney-item-4}  If $y\in\tfrac 12 B_{l}$ for some $l$, then
			\begin{equation}\label{eq:rl and y estimate}
				\frac{r_{l}}{16} \le \min\{R/8,\dist(y,X\setminus W)/(64M)\}\le \frac{r_{l}}{4};
			\end{equation}
			if $y\in W$ is in no ball $\tfrac 12 B_{l}$, then
			\[
			B(y,\min\{R/8,\dist(y,X\setminus W)/(64M)\})\cap D=\emptyset.
			\]
		\end{enumerate}
	\end{lemma}
	
	\begin{proof}
		For every $x\in A$, let $r_x\coloneqq \min\{R,\tfrac{1}{8M} \dist(x,X\setminus W)\}$.
		Consider the covering $\{B(x,\frac{1}{40}r_x)\}_{x\in A}$.
		By the $5$-covering theorem (see e.g. \cite[Theorem 1.24]{EvGa} and \cite[Section 3.3]{HKSTbook}),
		we can pick an at most countable
		collection of pairwise disjoint balls $B(x_l,\tfrac{1}{40}r_l)$ such that the balls
		$B(x_l,\tfrac{1}{8} r_l)$ cover $A$.
        For every $y\in D$, approximating $y$ by points of $A$ and using the preceding $\tfrac18 B_l$-cover therefore gives $y\in\tfrac14B_l$ for some $l$.
Thus the balls $B(x_l,\frac{1}{4}r_l)$ cover $D$.
		Denote $B_l=B(x_l,r_l)$.
We have established property \ref{Whitney-item-1}.
		
		Suppose $MB_l\cap MB_m\neq \emptyset$.
If $r_m=\tfrac{1}{8M} \dist(x_m,X\setminus W)$, then
		\begin{align*}
			8Mr_m
			&=\dist(x_m,X\setminus W)\\
			&\ge \dist(x_l,X\setminus W)-Mr_m-Mr_l\\
			&\ge 8Mr_l-Mr_m-Mr_l
			= 7Mr_l-Mr_m,
		\end{align*}
		and so we get $2r_m\ge r_l$.
		If $r_m=R$, then $r_l\le R=r_m$.
Thus we get property \ref{Whitney-item-2}.
		
		If $x\in MB_l$ for some $l$, denote by $I\coloneqq\{m:x\in MB_m\}$ the set of all indices of enlarged balls containing $x$.
		For every $m\in I$, applying \ref{Whitney-item-2} twice shows that $r_l/2\le r_m\le2r_l$.
Consequently, the pairwise disjoint balls $\tfrac1{40}B_m$ lie in $7MB_l$, and doubling gives
		\[
		\# I\,\mu(\tfrac1{40}B_l)
		\le C_d^{\log_2(280M)+1}\sum_{m\in I}\mu(\tfrac1{40}B_m)
		\le C_d^{2\log_2(280M)+2}\mu(\tfrac1{40}B_l).
		\]
		Thus \ref{Whitney-item-3} holds with $C_M\coloneqq C_d^{2\log_2(280M)+2}$.
		
		To prove the first part of \ref{Whitney-item-4}, suppose $y\in \tfrac 12 B_l$ for some $l$.
		Recall that
		\[
		r_l=\min\{R,\tfrac{1}{8M} \dist(x_l,X\setminus W)\}.
		\]
		We estimate
		\begin{align*}
			\dist(y,X\setminus W)
			\ge \dist(x_l,X\setminus W)-r_l
			&\ge \dist(x_l,X\setminus W)-\tfrac{1}{8M}\dist(x_l,X\setminus W)\\
			&\ge \dist(x_l,X\setminus W)/2.
		\end{align*}
		Analogously, we have $\dist(y,X\setminus W)\le 2\dist(x_l,X\setminus W)$.
		From these, \eqref{eq:rl and y estimate} easily follows.
		
		To prove the second part of \ref{Whitney-item-4}, suppose instead that
		$B(y,\min\{R/8,\dist(y,X\setminus W)/(64M)\})\cap D\neq \emptyset$.
		Thus there exists a point $x\in D$ with
		\[
		d(x,y)< \min\{R/8,\dist(y,X\setminus W)/(64M)\}.
		\]
		By \ref{Whitney-item-1}, we have $x\in \tfrac 14 B_l$ for some $l$, and so we have
		\begin{equation}\label{eq: x y and Bj}
			d(x_l,y)< \min\{R/8,\dist(y,X\setminus W)/(64M)\}+\frac{1}{4}r_l.
		\end{equation}
		We have either
		\[
		r_l=R\ge 8\min\{R/8,\dist(y,X\setminus W)/(64M)\}> 8d(x_l,y)-2r_l,
		\]
		so that $y\in \tfrac 12 B_l$,
		or
		\begin{align*}
			r_l=\tfrac{1}{8M} \dist(x_l,X\setminus W)
			&\ge \tfrac{1}{8M} (\dist(y,X\setminus W)-d(x_l,y))\\
			&\ge \tfrac{1}{8M} (64Md(x_l,y)-16Mr_l-d(x_l,y))\quad\textrm{by }\eqref{eq: x y and Bj}\\
			&\ge  \tfrac{63}{8}d(x_l,y)-2r_l.
		\end{align*}
		We conclude that $y\in \tfrac 12 B_l$.
	\end{proof}
	
	\begin{definition}
		Let $A\subset X$ with $0<\mu(A)<\infty$, let $u$ be a function on $A$ and let $0<\beta<1$.
We say that
		$a\in\R$  is a $\beta$-median of $u$ in $A$ if 
		\[
		\mu(\{u\ge a\})\ge (1-\beta)\mu(A)
		\quad\textrm{and}\quad
		\mu(\{u\le a\})\ge \beta\mu(A).
		\]
		We denote $m_u(\beta, A):=a$.
	\end{definition}
	
	\begin{lemma}\label{lem:medians}
		Let $A\subset X$ be a bounded measurable set, $\mu(A)>0$, and let $u$ be a $\mu$-a.e. finite measurable function
		on $A$.
Then we find a point $x\in A$ such that $u(x)$ is a $\beta$-median of $u$ in $A$,
		with $2/5\le \beta\le 3/5$.
	\end{lemma}
	\begin{proof}
		Define the decreasing rearrangement
		\[
		u^\#(t)\coloneqq \inf\{a\in\mathbb R : \mu(\{x\in A:u(x)>a\})\le t\},\quad 0<t<\mu(A).
		\] 
		By \cite[Section 3]{CSS}, the function $u^\#$ is finite, non-increasing, and right-continuous on $(0,\mu(A))$.
		Thus $u^\#$ has at most countably many discontinuity points.
		
		We first show that if $t_0=\beta_0\mu(A)$ is a discontinuity point of $u^\#$ with $\beta_0\in[\frac{2}{5},\frac{3}{5}]$, then every $a\in [u^\#(t_0),u^\#(t_0^-)]$ is a $(1-\beta_0)$-median of $u$.
		
		Indeed, by definition of $u^\#$, for every $\varepsilon>0$,
		\[
		\mu(\{x\in A\colon u(x)>u^\#(t_0)+\varepsilon\})\le \beta_0\mu(A).
		\]
		Letting $\varepsilon\to 0$ gives that
		\begin{equation}\label{eq-median-1}
			\mu(\{x\in A:u(x)\le u^\#(t_0)\})\ge (1-\beta_0)\mu(A).
		\end{equation}
		On the other hand, for every sufficiently small $\varepsilon>0$,
		\[
		\mu(\{x\in A: u(x)\ge u^\#(t_0-\varepsilon)\})\ge t_0-\varepsilon.
		\]
		Letting $\varepsilon\to0$,
		we obtain
		\begin{equation}\label{eq-median-2}
			\mu(\{x\in A: u(x)\ge u^\#(t_0^-)\})\ge \beta_0\mu(A).
		\end{equation}
		Thus every $a\in[u^\#(t_0),u^\#(t_0^-)]$ satisfies the median inequalities.
		
		For every $a<u^\#(t)$, we have $\mu(\{x\in A:u(x)>a\})>t$.
Conversely, if $t<\mu(\{x\in A:u(x)>a\})$ and $u^\#(t)\le a$, then $\mu(\{x\in A:u(x)>a\})\le \mu(\{x\in A:u(x)>u^\#(t)\})\le t$, a contradiction.
		It follows that for all $a\in \R$, we have that
		\[
		|\{0<t<\mu(A):u^\#(t)>a\}|=\mu(\{x\in A:u(x)>a\}).
		\]
		
		Since $\mu(A)<\infty$ and $u$ is $\mu$-a.e. finite on $A$,
		choosing a sequence $(a_n)$ with $a-1<a_n<a_{n+1}<a$
		and $\lim_{n\to \infty}a_n=a$, we have
		\[
		|\{0<t<\mu(A):u^\#(t)>a_n\}|=\mu(\{x\in A:u(x)>a_n\})<\infty,
		\]
		and then
		\[
		|\{0<t<\mu(A):u^\#(t)\ge a\}|=\mu(\{x\in A:u(x)\ge a\}).
		\]
		Combining the above inequalities, we have that
		\begin{align*}
			\frac{1}{5}\mu(A)
			&\le|\{0<t<\mu(A):u^\#(\tfrac{3}{5}\mu(A))\le u^\#(t)\le u^\#(\tfrac{2}{5}\mu(A)) \}| \\
			&=\mu(\{x\in A: u^\#(\tfrac{3}{5}\mu(A))\le u(x)\le u^\#(\tfrac{2}{5}\mu(A))\}).
		\end{align*}
		Hence there exists $x\in A$ such that
		\[
		u^\#(\tfrac{3}{5}\mu(A))\le u(x)\le u^\#(\tfrac{2}{5}\mu(A)).
		\]
		Set $p\coloneqq\mu(\{u\le u(x)\})/\mu(A)$ and $q\coloneqq\mu(\{u\ge u(x)\})/\mu(A)$. 
		Then $p,q\ge2/5$, while $p+q\ge1$.
Hence the interval $[\max\{2/5,1-q\},\min\{3/5,p\}]$ is nonempty.
Any $\beta$ in this interval makes $u(x)$ a $\beta$-median of $u$ in $A$.
	\end{proof}

\begin{lemma}\label{lem:choice of new R}
Let $C_*\ge 1$, $M\ge1$ and let $\nu \in \mathcal M(\Om,C_*)$.
Let $\delta>0$, and let $W\subset \Om$ be an open set.
There exists $R_*=R_*(\nu,W,\delta)>0$ such that the following holds.
If $A\subset W$ is Borel, $0<r\le R_*$, and denoting $r_x\coloneqq \min \{r,\dist(x,X\setminus W)/(64M)\}$
we have
\begin{equation}\label{eq:rx assumption}
\frac{\rcapa_1(A\cap B(x,r_x),B(x,2r_x))}{\rcapa_1(B(x,r_x),B(x,2r_x))}<\delta\quad\textrm{for all }x\in W,
\end{equation}
then $\nu(A)\le 2C_{*}\delta \nu(W)$.
\end{lemma}
\begin{proof}
We can assume that $0<\nu(W)<\infty$.
For sufficiently small $\eps>0$, we have $\nu(W\setminus W_{\eps})<\delta\nu(W)$;
recall that
\[
W_{\eps}\coloneqq \{x\in W\colon \dist(x,X\setminus W)>\eps\}.
\]
Choose $R_*=\min\{\eps/(64M),R(\nu,W_{\eps},\delta)\}$.
If $A\subset W$ is Borel,
$0<r\le R_*$, and \eqref{eq:rx assumption} holds,
then in fact
\[
\frac{\rcapa_1(A\cap B(x,r),B(x,2r))}{\rcapa_1(B(x,r),B(x,2r))}<\delta\quad\textrm{for all }x\in W_{\eps}.
\]
It follows that $\nu(A\cap W_{\eps})\le C_{*}\delta \nu(W_\eps)$.
In total, we get
\[
\nu(A)\le \nu(A\cap W_{\eps})+\nu(W\setminus W_{\eps})\le 2C_{*}\delta \nu(W).
\]
\end{proof}

We will need the following removability theorem that is essentially from \cite[Theorem~4.12]{L-CD}.
In that paper the result is formulated in the whole space $X$,
but we prove the following localized version.

\begin{proposition}\label{prop:localized isometric removability}
Let $\Om\subset X$ be open as usual, and let $A\subset\Om$.
Suppose that
\begin{equation}\label{eq:localized removability density}
\liminf_{r\to0}
\frac{\rcapa_1(A\cap B(x,r),B(x,2r))}
{\rcapa_1(B(x,r),B(x,2r))}
<c_*
\end{equation}
for $1$-q.e.\ $x\in\Om$.
Then $\mu(A)=0$, and $A$ is isometrically removable for $N^{1,1}(\Om)$;
that is, for every $u\in N^{1,1}(\Om\setminus A)$ there exists $\widetilde u\in N^{1,1}(\Om)$ such that $\widetilde u=u$ $\mu$-a.e.\ in $\Om\setminus A$ and
\begin{equation}\label{eq:localized isometric gradient}
g_{\widetilde u,\Om}=g_{u,\Om\setminus A}
\quad\text{$\mu$-a.e.\ in $\Om$},
\end{equation}
where $g_{u,\Om\setminus A}$ is extended by zero to $A$.
If $\Om\setminus A$ is $1$-quasiopen, then the pointwise representative $\widetilde u$ can be chosen so that $\widetilde u=u$ at every point in $\Om\setminus A$.
\end{proposition}

\begin{proof}
Since $X$ is proper, we can choose bounded open sets $U_1\Subset U_2\Subset\cdots\Subset\Om$ such that $\Om=\bigcup_{j=1}^{\infty}U_j$.
For every $j\in\N$, choose $\eta_j\in\operatorname{Lip}_c(\Om)$ such that $0\le\eta_j\le1$ and $\eta_j=1$ in a neighborhood of $\overline{U_j}$, and set $K_j\coloneqq \supp\eta_j\Subset\Om$ and $A_j\coloneqq A\cap K_j$.
Then by \eqref{eq:localized removability density}, we have
\[
\liminf_{r\to0}
\frac{\rcapa_1(A_j\cap B(x,r),B(x,2r))}
{\rcapa_1(B(x,r),B(x,2r))}
<c_* 
\quad \textrm{for 1-q.e. } x\in \Om.
\]
If $x\in X\setminus K_j$, then $A_j\cap B(x,r)=\emptyset$ for all sufficiently small $r>0$, and hence the corresponding quotient is eventually zero.
Thus the density hypothesis holds for $A_j$ at $1$-q.e.\ $x\in X$, and \cite[Theorem~4.12]{L-CD} shows that $A_j$ is isometrically removable for $N^{1,1}(X)$.
In particular, $\mu(A_j)=0$ and consequently $\mu(A)=0$.

Now let $u\in N^{1,1}(\Om\setminus A)$.
For every $j\in\N$, define $v_j$ on $X\setminus A_j$ by
\[
v_j(x)\coloneqq 
\begin{cases}
\eta_j(x)u(x),&x\in\Om\setminus A,\\
0,&x\in (X\setminus A_j)\setminus(\Om\setminus A).
\end{cases}
\]
This is well defined because $\eta_j=0$ on $A\setminus A_j$.
It is easy to check $v_j\in N^{1,1}(X\setminus A_j)$.
By the isometric removability of $A_j$, there exists $u_j\in N^{1,1}(X)$ such that $u_j=v_j$ $\mu$-a.e.\ in $X\setminus A_j$ and
\begin{equation}\label{eq:localized removability uj}
g_{u_j,X}=g_{v_j,X\setminus A_j}
\quad\text{$\mu$-a.e.\ in $X$}.
\end{equation}
Since $\eta_j=1$ in a neighborhood of $\overline{U_j}$, by \eqref{eq:localized removability uj} we know $u_j=u$ $\mu$-a.e.\ in $U_j$ and
\begin{equation}\label{eq:localized removability gradient Uj}
g_{u_j,U_j}=g_{u,\Om\setminus A}
\quad\text{$\mu$-a.e.\ in $U_j$}.
\end{equation}

Replace each $u_j$ by its $1$-quasicontinuous representative.
The uniqueness of such representatives implies that $u_j=u_{j+1}$ $1$-q.e.\ in $U_j$.
After removing a set $Z\subset\Om$ with $\capa_1(Z)=0$, we may therefore define $\widetilde u=u_j$ on $U_j\setminus Z$ and $\widetilde u=0$ on $Z$.
Then $\widetilde u\in N^{1,1}_{\loc}(\Om)$ and by \eqref{eq:localized removability gradient Uj} we have 
\begin{equation}\label{eq:localized removability local final}
g_{\widetilde u,U_j}=g_{u,\Om\setminus A}
\quad\text{$\mu$-a.e.\ in $U_j$}
\end{equation}
for every $j\in\N$.

Let $g$ equal $g_{u,\Om\setminus A}$ on $\Om\setminus A$ and zero on $A$.
By \eqref{eq:localized removability local final}, the fact that $\Om=\bigcup_jU_j$, and $\mu(A)=0$, we have $\widetilde u\in N^{1,1}(\Om)$ and $g_{\widetilde u,\Om}= g$ $\mu$-a.e.
By the construction we have $\widetilde u=u$ $\mu$-a.e. in $\Om\setminus A$.
If $\Om\setminus A$ is $1$-quasiopen, quasicontinuity gives $\widetilde u=u$ $1$-q.e. in $\Om\setminus A$, and redefining $\widetilde u$ on the exceptional set of zero $1$-capacity implies pointwise equality.
\end{proof}

The next definition separates an exceptional set into an open part controlled by a barrier and a ball-covered part controlled by local oscillations.
The associated cost records the total additional $L^1$-mass introduced by these two mechanisms.

\begin{definition}\label{def:admissible control}
Let $E\subset\Om$, $0<\ell\le1$, and $w\colon\Om\to[-\infty,\infty]$ be pointwise defined.
An $(\ell,w)$-control of $E$ is a triple $\mathcal S=(G,\{B_i\}_i,b)$, where $G\subset\Om$ is open,
$\{B_i=B(z_i,r_i)\}_i\text{ is a finite or countable family of balls}$ with $2B_i\Subset\Om$ and
$0<r_i<\ell/8$, and $b\in L^1(X)$ is nonnegative and Borel measurable, such that $E\subset G\cup\bigcup_iB_i$.
The set $H(\mathcal S)\coloneqq G\cup\bigcup_i2B_i$ is called the control region.
Define
\begin{equation}\label{eq:control cost}
\begin{split}
\mathcal C_w(\mathcal S)
&\coloneqq
\int_Xb\,d\mu
+\sum_i\frac{\osc_{B_i\setminus G}w}{r_i}\mu(2B_i).
\end{split}
\end{equation}
The quantity $\mathcal C_w(\mathcal S)$ is called the control cost.
We require $\mathcal S$ to be a finite-cost control, meaning that $\mathcal C_w(\mathcal S)<\infty$.
We also require the existence of a curve family $\Gamma_0(\mathcal S)$ with $\Mod_1(\Gamma_0(\mathcal S))=0$ such that every rectifiable curve $\gamma\colon [0,\ell_{\gamma}]\to \Om$ with $\gamma\notin\Gamma_0(\mathcal S)$, $\ell_\gamma\ge\ell$, $\gamma(0),\gamma(\ell_\gamma)\notin H(\mathcal S)$, and $\gamma\cap G\ne\emptyset$ satisfies
\begin{equation}\label{eq:control curve condition}
|w(\gamma(0))-w(\gamma(\ell_\gamma))|
\le\int_\gamma b\,ds.
\end{equation}
The choice $G=\emptyset$, $\{B_i\}_i=\emptyset$, and $b=0$ is called the empty control.
\end{definition}

The following proposition shows that controls with $\mu$-vanishing control regions
and uniformly bounded control costs yield a quantitative BV estimate.

\begin{proposition}[Controlled BV criterion]\label{prop:controlled BV criterion}
Let $\Om$ be bounded, let $E,Z\subset\Om$ with $\capa_1(Z)=0$, and let $w\in L^1(\Om)$.
Suppose that there exist $\nu\in\mathcal{M}\bigl(\Om,(800C_1C_S^2C_d^6\widetilde c_*)^{-1}\bigr)$, $M\ge1$, and a nonnegative function $h\in L^1(\Om,\nu)$ such that $\nu(\Om)<\infty$ and
\begin{equation}\label{eq:controlled pointwise hypothesis}
\begin{split}
\Lip_{w,\nu}^{\widetilde c_*,M}(x)
&\le h(x)<\infty
\quad\textrm{for every }x\in\Om\setminus(E\cup Z).
\end{split}
\end{equation}
For $\ell_m\coloneqq 1/m$, suppose that there is an $(\ell_m,w)$-control $\mathcal S_m$ of $E$ for every $m\in\N$ such that
\begin{equation}\label{eq:controlled BV control assumptions}
\mu(H(\mathcal S_m))\to0,\quad
\liminf_{m\to\infty}\mathcal C_w(\mathcal S_m)<\infty.
\end{equation}
Then $w\in\BV(\Om)$ and
\begin{equation}\label{eq:controlled BV estimate}
\|Dw\|(\Om)
\le C\int_\Om h\,d\nu
+C\liminf_{m\to\infty}\mathcal C_w(\mathcal S_m),
\end{equation}
where $C$ depends only on the doubling and Poincar\'e constants and on $M$.
\end{proposition}

\begin{proof}
\noindent\textbf{Step 1. Preliminary reductions and uniform decomposition.}

 After passing to a subsequence, assume that $\mathcal C_w(\mathcal S_m)$ converges to the liminf in \eqref{eq:controlled BV estimate} and is uniformly bounded.
Set $\varepsilon_m\coloneqq 2^{-m}$, fix $m\in\N$, and abbreviate $\ell\coloneqq \ell_m$, $\varepsilon\coloneqq \varepsilon_m$, and $\mathcal S_m=(G,\{B_i\}_i,b)$.
By replacing $Z$ with a Borel set of the same zero capacity, we may assume that $Z$ is Borel.
We first prove the estimate when $h$ is lower semicontinuous. Set $\eta\coloneqq \varepsilon/(1+\nu(\Om))$.
For each $x\in\Om\setminus(E\cup Z)$, set
\[
r_x^0\coloneqq\min\{1,\dist(x,X\setminus\Om)/(2M),\diam X/8\}.
\]
In the definition \eqref{eq:generalized Lip definition} of the generalized Lipschitz number, we can interpret the set $U_{r}$ to be defined for every $0<r<r_x^0$ at each such $x$.
Outside the sufficiently small radii used in that definition, set $U_{x,r}=B(x,r)$ for the remaining radii below $r_x^0$.
Thus to every $x\in\Om\setminus(E\cup Z)$ and $0<r<r_x^0$, there corresponds a set $U_{x,r}\subset B(x,r)$ with
\begin{equation}\label{eq:1 over j capa}
	\frac{\rcapa_1(B(x,r)\setminus U_{x,r}, B(x,2r))}{\rcapa_1(B(x,r),B(x,2r))}\le 2\widetilde c_*,
\end{equation}
and
\[
\frac{1}{2}\limsup_{r\to 0}\frac{\osc_{U_{x,r}}w}{r}
\frac{\mu(B(x,Mr))}{\nu(B(x,Mr))}
< \Lip_{w,\nu}^{\widetilde c_*,M}(x)+\eta.
\]
Note that
\begin{align*}
	B(x,r)\setminus \fineint U_{x,r}
	& = B(x,r)\cap (X\setminus \fineint U_{x,r})\\
	& = B(x,r)\cap \overline{X\setminus  U_{x,r}}^1\\
	& \subset \overline{B(x,r)\setminus  U_{x,r}}^1\quad\textrm{by }\eqref{eq:characterization of fine closure}.
\end{align*}
Combining this with \eqref{eq:capa of fine int} and \eqref{eq:1 over j capa}, we see that
\[
\frac{\rcapa_1(B(x,r)\setminus \fineint U_{x,r},B(x,2r))}{\rcapa_1(B(x,r),B(x,2r))}\le 2\widetilde c_*,
	\]
and so we can actually assume that every $U_{x,r}$ is $1$-finely open.
Set $V\coloneqq G\cup \bigcup_i B_i$, so that $E\subset V$.
For each $j\in \N$, let $A_j$ consist of points $x\in\Om\setminus(V\cup Z)$ such that $1/j<r_x^0$ and
\begin{equation}\label{eq:Lip estimate}
\frac{1}{2}\sup_{0<r\le 1/j}\frac{\osc_{U_{x,r}}w}{r}
\frac{\mu(B(x,Mr))}{\nu(B(x,Mr))}
<\Lip_{w,\nu}^{\widetilde c_*,M}(x)+\eta
\end{equation}
and also
\begin{equation}\label{eq:h estimate with a}
\Lip_{w,\nu}^{\widetilde c_*,M}(x)
\le h(y)+\eta\quad \textrm{for all }y\in B(x,M/j).
\end{equation}
We have $\Om=V\cup Z\cup \bigcup_{j=1}^{\infty}A_j$.
For $j=1,2,\ldots$, we inductively define 
\[
S_j\coloneqq A_j\setminus \bigcup_{i=1}^{j-1} D_i,\quad
 D_j\coloneqq \left(\Om\setminus(V\cup Z)\right)\cap \overline{S_j}\setminus \bigcup_{i=1}^{j-1}D_i.
\]
Note that we do not know the sets $A_j$ to be measurable but the sets $D_j$ are Borel sets.
Moreover, the sets $D_j$ are disjoint and $D_j\supset A_j\setminus \bigcup_{i=1}^{j-1}D_i$, so that $\Om=V\cup Z\cup\bigcup_{j=1}^{\infty}D_j$.
We may pick open sets $W_j\supset D_j$ such that $W_j\subset \Om$, and 
\begin{equation}\label{eq: choose Wj}
\int_{W_j} (h+2\eta)\,d\nu \le \int_{D_j}(h+2\eta)\,d\nu +2^{-j}\varepsilon.
\end{equation}

\medskip

\noindent\textbf{Step 2. Iterated Whitney construction and isometric removability of the residual set.}

We define sets $D_{j,k}^{\ell}$ and $W_{j,k}^{\ell}$.
For each $j\in\N$, set $D_{j,1}^{\ell}=D_j$ and $W_{j,1}^{\ell}\coloneqq W_j$, and proceed inductively in $k\ge1$.
From Lemma \ref{lem:h nv is also M}, we know that $h\,d\nu\in \mathcal M\bigl(\Om,(400C_1C_S^2C_d^6\widetilde c_*)^{-1}\bigr)$.
Inductively, we do the following.
From Lemma \ref{lem:choice of new R}, we obtain parameters
\[
R_{*,1}=R_{*,1}(\nu,W_{j,k}^{\ell},100C_1C_S^2C_d^6 \widetilde c_*)
\quad\textrm{and}\quad
R_{*,2}=R_{*,2}(h\,d\nu,W_{j,k}^{\ell},100C_1C_S^2C_d^6 \widetilde c_*).
\]
Let $R_{*}\coloneqq \min\{R_{*,1},R_{*,2}\}$.
Since $W_{j,k}^\ell$ is open and $S_j$ is dense in $D_j$, we have that $S_j\cap W_{j,k}^\ell$
is dense in $D_j\cap W_{j,k}^\ell$.
Using Lemma \ref{lem:Whitney} with $A=S_j\cap W_{j,k}^\ell$ and $D=D_j\cap W_{j,k}^\ell$,
take a Whitney-type covering $\{B_{j,k,l}^{\ell}=B(x_{j,k,l}^{\ell},r_{j,k,l}^{\ell})\}_l$ at scale
\[
R_{j,k}^{\ell}\coloneqq \min\{\ell/(j+k),\,
R_*,\,\diam X/100\}.
\]
For each ball $B_{j,k,l}$, there is the corresponding $1$-finely open set $U_{j,k,l}^{\ell}\subset B_{j,k,l}^{\ell}$.
Define the new ``bad'' set
\begin{equation}\label{eq:Djk def}
D_{j,k+1}^{\ell}\coloneqq D_{j,k}^{\ell}\setminus \bigcup_l U_{j,k,l}^{\ell}.
\end{equation}
By induction, $D_{j,k}^\ell \subset D_j\cap W_{j,k}^\ell$.
By Lemma \ref{lem:Whitney}\ref{Whitney-item-1}, we know that 
$ D_{j,k}^\ell
\subset \bigcup_l\frac{1}{4}B_{j,k,l}^\ell$.
Note that $x_{j,k,l}^\ell\in S_j\subset A_j$ for all $k,l>0$.
		
By \eqref{eq:1 over j capa}, we have
\begin{equation}\label{eq:Djk capacity}
\frac{\rcapa_1( D_{j,k+1}^\ell\cap B_{j,k,l}^{\ell}, 2B_{j,k,l}^{\ell})}{\rcapa_1(B_{j,k,l}^{\ell},2B_{j,k,l}^{\ell})}
\le \frac{\rcapa_1(B_{j,k,l}^{\ell}\setminus U_{j,k,l}^{\ell},2B_{j,k,l}^{\ell})}{\rcapa_1(B_{j,k,l}^{\ell},2B_{j,k,l}^{\ell})}
\le 2\widetilde c_*
\quad\textrm{for all }l .
\end{equation}
Consider $y\in W_{j,k}^{\ell}$.
Suppose $y$ is in no ball $\tfrac 12 B_{j,k,l}^{\ell}$.
Let $r_y^{\ell}\coloneqq \min\{R_{j,k}^{\ell}/8,\dist(y,X\setminus W_{j,k}^{\ell})/(64M)\}$.
Then by Lemma \ref{lem:Whitney} \ref{Whitney-item-4}, we have
\[
B(y,r_y^{\ell})\cap D_{j,k}^\ell=\emptyset
\quad\textrm{and so also}\quad
B(y,r_y^{\ell})\cap D_{j,k+1}^\ell=\emptyset.
\]
If we have $y\in \tfrac 12 B_{j,k,l}^{\ell}$ for some $l$, then by Lemma \ref{lem:Whitney} \ref{Whitney-item-4} we have
\[
\frac{r_{j,k,l}^{\ell}}{16} \le r_y^{\ell}\le \frac{r_{j,k,l}^{\ell}}{4},
\]
and so $B(y,r_y^{\ell})\subset B_{j,k,l}^{\ell}\subset 2B_{j,k,l}^{\ell}\subset B(y,40 r_y^{\ell})$.
Here $40r_y^{\ell}\le5R_{j,k}^{\ell}\le\diam X/20<\diam X/4$ if $\diam X<\infty$.
Thus by \eqref{eq: rcapa-compare-in-ball}, we get
\begin{align*}
	\rcapa_1( D_{j,k+1}^\ell\cap B(y,r_y^{\ell}),B(y,2r_y^{\ell}))&
	\le 45C_S \rcapa_1( D_{j,k+1}^\ell\cap B(y,r_y^{\ell}), B(y,40r_y^{\ell}))\\
	&\le 45C_S\rcapa_1( D_{j,k+1}^\ell\cap B_{j,k,l}^{\ell}, 2B_{j,k,l}^{\ell})\\
	&\le 90C_S  \widetilde c_* \rcapa_1(B_{j,k,l}^{\ell},2B_{j,k,l}^{\ell})\quad\textrm{by }\eqref{eq:Djk capacity}\\
    &\le 90C_S^2C_d^6 \widetilde c_*\rcapa_1(B(y,r_y^\ell), B(y,2r_y^\ell))
     \quad
    \textrm{by }\eqref{eq:rcapa over ball-measure}.
	\end{align*}
Note that $\rcapa_1( D_{j,k+1}^\ell\cap B(y,r_y^{\ell}),B(y,2r_y^{\ell}))\le 90C_S^2C_d^6\widetilde c_*\rcapa_1(B(y,r_y^\ell), B(y,2r_y^\ell))$ thus holds for every $y\in W_{j,k}^{\ell}$.
By Lemma \ref{lem:weak Borel regularity}, we can find a Borel set $D_{j,k+1}^{\ell,*}\supset D_{j,k+1}^\ell$,
with $D^{\ell,*}_{j,k+1}\subset W_{j,k}^{\ell}$, such that for every $y\in W_{j,k}^{\ell}$ we have
\begin{equation}\label{eq:small capacity}
\rcapa_1(D^{\ell,*}_{j,k+1}\cap B(y,r_y^{\ell}),B(y,2r_y^{\ell}))
\le 90C_1C_S^2C_d^6\widetilde c_*\rcapa_1(B(y,r_y^\ell), B(y,2r_y^\ell)).
\end{equation}
Thus the capacity hypothesis of Lemma~\ref{lem:choice of new R} holds with $\delta=100C_1C_S^2C_d^6\widetilde c_*$ and $r=R_{j,k}^{\ell}/8\le R_*$.
Since $h\,d\nu\in \mathcal{M}\bigl(\Om,(400C_1C_S^2C_d^6\widetilde c_*)^{-1}\bigr)$, by Lemma \ref{lem:choice of new R}, first for $h\,d\nu$ and then for $\nu$, we have
\begin{equation}\label{eq:first Djk plus one estimate}
\begin{split}
\int_{D^{\ell,*}_{j,k+1}}(h+2\eta)\,d\nu
&\le 2\bigl((400C_1C_S^2C_d^6\widetilde c_*)^{-1}\bigr)
\left(100C_1C_S^2C_d^6\widetilde c_*\right)
\int_{W_{j,k}^{\ell}}h\,d\nu\\
&\quad+2\eta\,(400C_1C_S^2C_d^6\widetilde c_*)^{-1}
\left(100C_1C_S^2C_d^6\widetilde c_*\right)\nu(W_{j,k}^{\ell})\\
&\le \frac12\int_{W_{j,k}^{\ell}}(h+2\eta)\,d\nu.
\end{split}
\end{equation}
By outer regularity, choose an open set $W_{j,k+1}^{\ell}\subset W_{j,k}^{\ell}$ with $W_{j,k+1}^{\ell}\supset D^{\ell,*}_{j,k+1}$ such that
\begin{equation}\label{eq:estimate list}
\begin{split}
\int_{W_{j,k+1}^{\ell}} (h+2\eta) \,d\nu
\le \int_{D^{\ell,*}_{j,k+1}} (h+2\eta) \,d\nu
+\frac14\int_{W_{j,k}^{\ell}} (h+2\eta) \,d\nu.
\end{split}
\end{equation}

Define
\begin{equation}\label{eq:controlled residual set}
\begin{split}
N_{\ell,\varepsilon}
\coloneqq
\Om\setminus\left(
V\cup{Z}\cup\bigcup_{j,k,l}U_{j,k,l}^{\ell}\right)
=\bigcup_j D_j\setminus \bigcup_{j,k,l}U_{j,k,l}^\ell.
\end{split}
\end{equation}
Recall that $ D_{j,1}^\ell=D_j$.
By \eqref{eq:Djk def}, we have
\begin{equation}\label{eq:Aj minus Nj-new}
D_j\setminus \bigcup_{k,l}U_{j,k,l}^\ell
= \bigcap_{k=1}^\infty D_{j,k}^\ell.
\end{equation}

Now fix $x\in N_{\ell,\varepsilon}$.
There exists $j\in\N$ such that $x\in D_j\setminus \bigcup_{k,l}U_{j,k,l}^\ell$.
Hence for every $k\ge1$, by \eqref{eq:Aj minus Nj-new}, we have
\[
x\in D_{j,k}^\ell\subset \bigcup_l \frac{1}{4}B_{j,k,l}^\ell\subset W_{j,k}^\ell .
\]
Then we have $x\in \frac{1}{4}B_{j,k,l}^\ell\subset\frac{1}{2}B_{j,k,l}^\ell$ for some $l$.
Let $r_{x,k}^{\ell}\coloneqq \min\{R_{j,k}^{\ell}/8,\dist(x,X\setminus W_{j,k}^{\ell})/(64M)\}$.
Then $B(x,r_{x,k}^\ell)\subset W_{j,k}^\ell$.
And by Lemma \ref{lem:Whitney} \ref{Whitney-item-4}, we have 
\[
\frac{r_{j,k,l}^{\ell}}{16} \le r_{x,k}^{\ell}\le \frac{r_{j,k,l}^{\ell}}{4},
\]
and so
$B(x,r_{x,k}^\ell)\subset B_{j,k,l}^\ell\subset 2B_{j,k,l}^\ell\subset B(x,40r_{x,k}^\ell)$.
Thus, we get
\begin{equation}\label{eq: exceptional set-capacitary density}
\begin{split}
    \rcapa_1(B(x,r_{x,k}^\ell)\cap N_{\ell,\varepsilon}, B(x,2r_{x,k}^\ell))
    &\le 45C_S\rcapa_1(B(x,r_{x,k}^\ell)\setminus \bigcup_l U_{j,k,l}^\ell,B(x,40r_{x,k}^\ell))\\
    &\le 45C_S\rcapa_1(B_{j,k,l}^\ell\setminus U_{j,k,l}^\ell, 2B_{j,k,l}^\ell)\\
    &\le 90C_S\widetilde c_*\rcapa_1(B_{j,k,l}^\ell,2B_{j,k,l}^\ell) 
    \quad \textrm{by } \eqref{eq:1 over j capa}\\
    &\le 90C_S^2C_d^6\widetilde c_*\rcapa_1(B(x,r_{x,k}^\ell),B(x,2r_{x,k}^\ell))
    \quad \textrm{by }\eqref{eq:rcapa over ball-measure}.
\end{split}
\end{equation}
Since $R_{j,k}^\ell\le\ell/(j+k)\to0$ as $k\to\infty$, we also have $r_{x,k}^\ell\to0$.
It follows that for every $x\in N_{\ell,\varepsilon}$, by \eqref{eq: exceptional set-capacitary density}, we have 
\begin{equation}\label{eq:cap-removablity}
\begin{split}
    \liminf_{r\to0}
\frac{\rcapa_1(N_{\ell,\varepsilon}\cap B(x,r),B(x,2r))}{\rcapa_1(B(x,r),B(x,2r))}
&\le\liminf_{k\to\infty}\frac{\rcapa_1(N_{\ell,\varepsilon}\cap B(x,r_{x,k}^\ell),B(x,2r_{x,k}^\ell))}{\rcapa_1(B(x,r_{x,k}^\ell),B(x,2r_{x,k}^\ell))}\\
&\le
90C_S^2C_d^6\widetilde c_*
<
c_*
\quad\text{by }\eqref{eq:cn definition}.
\end{split}
\end{equation}
On the other hand, the set $V\cup \bigcup_{j,k,l}U_{j,k,l}^{\ell}$ is $1$-finely open and
disjoint from $N_{\ell,\varepsilon}$.
Consequently, $N_{\ell,\varepsilon}$ is $1$-thin at every
$x\in \Om\setminus(N_{\ell,\varepsilon}\cup Z)\subset V\cup \bigcup_{j,k,l}U_{j,k,l}^{\ell}$, and therefore
\[
\lim_{r\to0}
\frac{\rcapa_1(N_{\ell,\varepsilon}\cap B(x,r),B(x,2r))}
{\rcapa_1(B(x,r),B(x,2r))}=0
\quad\text{for every }x\in \Om\setminus(N_{\ell,\varepsilon}\cup Z).
\]

Combining this with \eqref{eq:cap-removablity}, it follows that
\begin{equation}\label{removablility-N}
\liminf_{r\to 0} \frac{\rcapa_1(N_{\ell,\varepsilon}\cap B(x,r),B(x,2r))}{\rcapa_1(B(x,r),B(x,2r))}<c_* \quad \text{for $1$-q.e. } x\in \Om.
\end{equation}
 By Proposition \ref{prop:localized isometric removability}, $N_{\ell,\varepsilon}$ is isometrically removable for $N^{1,1}(\Om)$ and $\mu(N_{\ell,\varepsilon})=0$.

\medskip
\noindent\textbf{Step 3. Construction and $L^1$-estimate of the path density.}

Define the function
\begin{equation}\label{eq:controlled discrete gradient}
g_{\ell,\varepsilon}\coloneqq \sum_{j,k,l}\frac{\osc_{U_{j,k,l}^{\ell}}w}{r_{j,k,l}^{\ell}}\ch_{2B_{j,k,l}^{\ell}}.
\end{equation}
 If $\nu(MB_{j,k,l}^{\ell})=0$, the convention following \eqref{eq:generalized Lip definition} and the finiteness of \eqref{eq:controlled pointwise hypothesis} imply that $\osc_{U_{j,k,l}^{\ell}}w=0$, so the corresponding summand vanishes.
If $\nu(MB_{j,k,l}^{\ell})>0$, we estimate
\begin{equation}\label{eq:gj estimate}
\begin{split}
\frac{\osc_{U_{j,k,l}^\ell}w}{r_{j,k,l}^\ell}\mu(2B_{j,k,l}^\ell)
& \le  2\mu(2B_{j,k,l}^\ell) \frac{\nu(MB_{j,k,l}^\ell)}{\mu(MB_{j,k,l}^\ell)}\left(\Lip_{w,\nu}^{\widetilde c_*,M}(x_{j,k,l}^\ell)+\eta\right)
\quad \textrm{by }\eqref{eq:Lip estimate}\\
& \le 2C_d \int_{MB_{j,k,l}^\ell} (h+2\eta)\,d\nu\quad 
\textrm{by }\eqref{eq:h estimate with a}.
\end{split}
\end{equation}
Here we use the fact that $x_{j,k,l}^\ell\in A_j$ for every $k,l\in \N$.
        
We also have
\begin{equation}\label{eq:D h nu estimate}
\begin{split}
\int_{W_{j,k+1}^\ell} (h+2\eta) \,d\nu
&\le \frac34\int_{W_{j,k}^\ell} (h+2\eta) \,d\nu
\quad\textrm{by \eqref{eq:estimate list} and \eqref{eq:first Djk plus one estimate}}.
\end{split}
\end{equation}
It follows that
\begin{equation}\label{eq:gj estimate with sum}
\begin{split}
\int_{X}\sum_{j,k,l}\frac{\osc_{U_{j,k,l}^\ell}w}{r_{j,k,l}^\ell}\chi_{2B_{j,k,l}^\ell}\,d\mu
&\le \sum_{j,k,l}\frac{\osc_{U_{j,k,l}^\ell}w}{r_{j,k,l}^\ell}\mu(2B_{j,k,l}^\ell)\\
&\le 2C_d\sum_{j,k,l} \int_{MB_{j,k,l}^\ell} (h+2\eta) \,d\nu\quad\textrm{by }\eqref{eq:gj estimate}\\
&\le 2C_dC_M\sum_{j,k} \int_{W_{j,k}^\ell} (h+2\eta) \,d\nu\quad 
\textrm{by Lemma }\ref{lem:Whitney} \ref{Whitney-item-3}\\
& \le 8C_dC_M\sum_j\int_{W_j} (h+2\eta) \,d\nu
\quad \textrm{by }\eqref{eq:D h nu estimate}
\\
&\le 8C_dC_M\left(\int_{\Om} (h+2\eta)\,d\nu+\varepsilon\right)
 \quad \textrm{by }\eqref{eq: choose Wj}.
\end{split}
\end{equation}
Since $2\eta\nu(\Om)\le2\varepsilon$, we obtain
\begin{equation}\label{eq: estimate g_ell}
\int_{X} g_{\ell,\varepsilon}\,d\mu\le C\int_\Om h\,d\nu+C\varepsilon.
\end{equation}

Set $Y_{\ell,\varepsilon}\coloneqq\Om\setminus N_{\ell,\varepsilon}$.
 Since $V\cup\bigcup_{j,k,l}U_{j,k,l}^{\ell}$ is $1$-finely open and $\capa_1(Z)=0$, the set $Y_{\ell,\varepsilon}=V\cup Z\cup\bigcup_{j,k,l}U_{j,k,l}^{\ell}$ is $1$-quasiopen;
recall \eqref{eq:finely and quasiopen}.
For each $s\in\N$, choose an open set $O_{s,\varepsilon}\subset \Om$ such that $\capa_1(O_{s,\varepsilon})<2^{-s}\varepsilon$ and
$Y_{\ell,\varepsilon}\cup O_{s,\varepsilon}$ is open, and set
\[
\widetilde Y_{\ell,\varepsilon}\coloneqq \bigcap_{s=1}^{\infty}(Y_{\ell,\varepsilon}\cup O_{s,\varepsilon}),
\qquad
Z_{\ell,\varepsilon}\coloneqq \widetilde Y_{\ell,\varepsilon}\setminus Y_{\ell,\varepsilon}.
\]
Then $\widetilde Y_{\ell,\varepsilon}$ is a $G_\delta$ subset of the complete separable space $X$, and hence is Polish, see \cite[Theorem 3.11]{SK95}, and $\capa_1(Z_{\ell,\varepsilon})=0$.

All curve families occurring in modulus estimates in this step consist of nonconstant rectifiable curves, parametrized by arc length on $[0,\ell_\gamma]$, where $0<\ell_\gamma<\infty$.
Let $\Gamma_Z$ be the family of curves meeting $Z$, and define
\[
\Gamma_{\ell,\varepsilon}
\coloneqq \Gamma_0(\mathcal S_m)\cup\Gamma_Z
\cup \Gamma_{Z_{\ell, \varepsilon}}
\cup\bigcup_{j,k,l}\{\gamma:\gamma^{-1}(U_{j,k,l}^{\ell})\text{ is not relatively open}\}.
\]
The $1$-path openness of $1$-quasiopen sets (see \cite[Remark 3.5]{S2})
and the capacity-zero properties of $Z$ and $Z_{\ell,\varepsilon}$ show that $\Mod_1(\Gamma_{\ell,\varepsilon})=0$;
see \cite[Proposition 1.48]{BB}.
Choose a nonnegative Borel function $q_{\ell,\varepsilon}$ such that
\[
\int_Xq_{\ell,\varepsilon}\,d\mu<\varepsilon,
\quad
\int_\gamma q_{\ell,\varepsilon}\,ds=\infty\ \textrm{ for every }\gamma\in\Gamma_{\ell,\varepsilon}.
\]
Define
\begin{equation}\label{eq:controlled rho definition}
\rho_{\ell,\varepsilon}\coloneqq g_{\ell,\varepsilon}+\sum_i\frac{\osc_{B_i\setminus G}w}{r_i}\ch_{2B_i}+b+q_{\ell,\varepsilon}.
\end{equation}
This function is Borel and belongs to $L^1(X)$.
By \eqref{eq: estimate g_ell} and the choice of $q_{\ell,\varepsilon}$, we know that
\begin{equation}\label{eq:controlled rho estimate}
\int_X\rho_{\ell,\varepsilon}\,d\mu
\le C\int_\Om h\,d\nu
+C\mathcal C_w(\mathcal S_m)+C\varepsilon.
\end{equation}

Let $\gamma\colon[0,\ell_\gamma]\to\widetilde Y_{\ell,\varepsilon}$ have length at least $\ell$ and endpoints outside $H(\mathcal S_m)$.
If $\gamma\in\Gamma_{\ell,\varepsilon}$, then its $q_{\ell,\varepsilon}$-integral is infinite.
 Hence the desired upper-gradient inequality for the pair $(w,\rho_{\ell,\varepsilon})$ holds for such a curve.
Suppose that $\gamma\notin\Gamma_{\ell,\varepsilon}$.
Then $\gamma$ avoids $Z$ and $Z_{\ell,\varepsilon}$, so $\gamma\subset Y_{\ell,\varepsilon}\setminus Z\subset V\cup\bigcup_{j,k,l}U_{j,k,l}^{\ell}$.
Recall that $V=G\cup \bigcup_i B_i$.
If $\gamma\cap G\ne\emptyset$, the definition \eqref{eq:control curve condition} implies the desired estimate. 

Suppose finally that $\gamma\cap G=\emptyset$.
Then we have $[0,\ell_\gamma]\subset \bigcup_{j,k,l}\gamma^{-1}(U_{j,k,l}^\ell)\cup \bigcup_i\gamma^{-1}(B_i)$,
i.e. the compact interval has an open cover.
A finite subcover contains a finite chain joining the two endpoints, and after deleting loops the sets in the chain can be taken to be distinct. 
It follows by the triangle inequality that
\[
|w(\gamma(0))-w(\gamma(\ell_\gamma))|
\le\sum_{\gamma\cap U_{j,k,l}^{\ell}\ne\emptyset}\osc_{U_{j,k,l}^{\ell}}w
+\sum_{\gamma\cap B_i\ne\emptyset}\osc_{B_i\setminus G}w.
\]
Here $r_{j,k,l}^{\ell}\le\ell/(j+k)\le\ell/2$ and $r_i<\ell/8$, and so
\[
\sum_{\gamma\cap U_{j,k,l}^{\ell}\ne\emptyset}\osc_{U_{j,k,l}^{\ell}}w
+\sum_{\gamma\cap B_i\ne\emptyset}\osc_{B_i\setminus G}w
\le\int_\gamma\left(g_{\ell,\varepsilon}+\sum_i\frac{\osc_{B_i\setminus G}w}{r_i}\ch_{2B_i}\right)\,ds
\le \int_\gamma\rho_{\ell,\varepsilon}\,ds.
\]
Thus every curve $\gamma\colon[0,\ell_\gamma]\to\widetilde Y_{\ell,\varepsilon}$
with length at least $\ell$ and endpoints outside $H(\mathcal S_m)$ satisfies
\begin{equation}\label{eq:controlled long curve estimate}
|w(\gamma(0))-w(\gamma(\ell_\gamma))|
\le\int_\gamma\rho_{\ell,\varepsilon}\,ds.
\end{equation}

 For each $m$, denote the objects just constructed by
\[
N_m\coloneqq N_{\ell_m,\varepsilon_m},\quad
Y_m\coloneqq \Om\setminus N_m,\quad
\widetilde Y_m\coloneqq \widetilde Y_{\ell_m,\varepsilon_m},\quad
H_m\coloneqq H(\mathcal S_m),\quad
\rho_m\coloneqq \rho_{\ell_m,\varepsilon_m}.
\]
Then $N_m$ is isometrically removable, $\capa_1(\widetilde Y_m\setminus Y_m)=0$, $\mu(H_m)\to0$, and by \eqref{eq:controlled rho estimate} we have $\sup_m\int_X\rho_m\,d\mu<\infty$.
Moreover, every rectifiable curve $\gamma\colon[0,\ell_\gamma]\to\widetilde Y_m$ with $\ell_\gamma\ge\ell_m$ and $\gamma(0),\gamma(\ell_\gamma)\notin H_m$ satisfies
\begin{equation}\label{eq:abstract long curve estimate}
|w(\gamma(0))-w(\gamma(\ell_\gamma))|\le\int_\gamma\rho_m\,ds.
\end{equation}

\medskip
\noindent\textbf{Step 4. Construction of the Newtonian approximations.}
\par\nopagebreak

Set the structural constant $A_0\coloneqq 8\lambda+4$ and define
\[
R_m\coloneqq\{z\in\Om:\dist(z,X\setminus\Om)\le A_0\ell_m\}.
\]
Since $\mu(\Om)<\infty$, we have $\mu(R_m)\to0$.
Choose numbers $K_m\to\infty$ sufficiently slowly that
\begin{equation}\label{eq:choose Km}
K_m\bigl(\mu(H_m)+\mu(R_m)\bigr)\to0
\quad \mathrm{and}\quad
\int_{\{|w|>K_m\}}|w|\,d\mu\to0.
\end{equation}
Put $F_m\coloneqq H_m\cup\{|w|>K_m\}$.

For all sufficiently large $m$, choose a maximal collection $\{x_{m,n}\}_n\subset Y_m\setminus H_m$ such that $\dist(x_{m,n},X\setminus\Om)>(A_0-3)\ell_m$ and $d(x_{m,n},x_{m,n'})\ge3\ell_m$ for $n\ne n'$.
Then the balls $B_{m,n}\coloneqq B(x_{m,n},\ell_m)$ are pairwise disjoint, every fixed dilation of these balls has uniformly bounded overlap, and
\[
\Om\setminus(N_m\cup H_m\cup R_m)\subset\bigcup_n3B_{m,n}.
\]
Denote by $I_m$ the set of indices $n$ for which $\mu(F_m\cap B_{m,n})\ge\mu(B_{m,n})/10$.
By disjointness and doubling, we have
\begin{equation}\label{eq:Im measure}
\begin{split}
\mu\left(\bigcup_{n\in I_m}3B_{m,n}\right)
&\le 10C_d^2\mu(F_m)\to0
\quad\textrm{as }m\to\infty.
\end{split}
\end{equation}

Fix $n\notin I_m$ and set $E_{m,n}\coloneqq B_{m,n}\setminus(F_m\cup N_m)$, so that $\mu(E_{m,n})\ge9\mu(B_{m,n})/10$.
By Lemma~\ref{lem:medians}, there exists $y_{m,n}\in E_{m,n}$ such that $w(y_{m,n})$ is a $\beta'$-median of $w$ on $E_{m,n}$ for some $2/5\le\beta'\le3/5$.
Consequently,
\[
\mu\bigl(E_{m,n}\cap\{w\le w(y_{m,n})\}\bigr)\ge\frac9{25}\mu(B_{m,n})
\quad\textrm{and}\quad
\mu\bigl(E_{m,n}\cap\{w\ge w(y_{m,n})\}\bigr)\ge\frac9{25}\mu(B_{m,n}).
\]
By inner regularity, choose a compact set
\[
A_{m,n}\subset E_{m,n}\cap\{w\le w(y_{m,n})\}
\quad\textrm{such that}\quad
\mu(A_{m,n})>\frac13\mu(B_{m,n}).
\]
We denote $p\coloneqq \mu(\{x\in B_{m,n}\colon w(x)\le w(y_{m,n})\})/\mu(B_{m,n})$; $q\coloneqq \mu(\{x\in B_{m,n}\colon w(x)\ge w(y_{m,n})\})/\mu(B_{m,n})$, respectively.
Choose
\[
\beta_{m,n}\in[\max\{1/4,1-q\},\min\{3/4,p\}].
\]
This interval is nonempty because $p,q\ge9/25$ and $p+q\ge1$.
Thus $w(y_{m,n})$ is a $\beta_{m,n}$-median of $w$ on $B_{m,n}$ with $1/4\le\beta_{m,n}\le3/4$.
Moreover, $A_{m,n}\subset Y_m\setminus H_m$, $w\le w(y_{m,n})$ on $A_{m,n}$, and $|w(y_{m,n})|\le K_m$.

For $z\in\widetilde Y_m$, define
\begin{equation}\label{eq:abstract vm definition}
\begin{split}
d_{m,n}(z)&\coloneqq\inf_\gamma\left\{\int_\gamma\rho_m\,ds:\gamma(0)\in A_{m,n},\ \gamma(\ell_\gamma)=z,\ \gamma\subset\widetilde Y_m\right\},\\
p_m(z)&\coloneqq\inf_{n\notin I_m}\bigl(w(y_{m,n})+d_{m,n}(z)\bigr),\\
v_m(z)&\coloneqq T_{K_m}(p_m(z))=\max\{-K_m,\min\{K_m,p_m(z)\}\}.
\end{split}
\end{equation}
Here the infimum of the empty set is $\infty$ and $T_{K_m}(\infty)\coloneqq K_m$.
Since $w(y_{m,n})\ge-K_m$ and $d_{m,n}\ge0$, the function $p_m$ never takes the value $-\infty$.

We first prove that $v_m$ is measurable.
For $L\in\N$, let
\[
\mathcal C_{m,L}\coloneqq\{\gamma\in C([0,1],X):\Lip(\gamma)\le L,\ \gamma([0,1])\subset\widetilde Y_m\},
\]
equipped with the topology of uniform convergence (note that here we do not require curves to be parametrized by arc-length).
Since $\widetilde Y_m$ is a $G_\delta$ subset of the complete separable space $X$, the space $\mathcal C_{m,L}$ is Polish, see \cite[Theorem 3.11, Theorem 4.19]{SK95}.
For $n\notin I_m$ and $a\in\R$, set
\[
\mathcal A_{m,n,L}(a)\coloneqq\left\{\gamma\in\mathcal C_{m,L}:\gamma(0)\in A_{m,n},\ w(y_{m,n})+\int_\gamma\rho_m\,ds<a\right\}.
\]
The evaluation map $e_0(\gamma)\coloneqq\gamma(0)$ is continuous, and $\{\gamma\in \mathcal{C}_{m,L}\colon e_0(\gamma)\in A_{m,n}\}=e_0^{-1}(A_{m,n})$ is closed because $A_{m,n}$ is compact.
By the argument in \cite[Proposition 1.5]{BBM17}, the map $J_m(\gamma)\colon \gamma\longmapsto\int_\gamma\rho_m\,ds$ is Borel on $\mathcal C_{m,L}$, and hence $\mathcal A_{m,n,L}(a)=e_0^{-1}(A_{m,n})\cap J_m^{-1}((-\infty, a- w(y_{m,n})))$ is Borel.
With $e_1(\gamma)\coloneqq\gamma(1)$, the definition in \eqref{eq:abstract vm definition} gives
\[
\{p_m<a\}=\bigcup_{n\notin I_m}\bigcup_{L\in\N}e_1(\mathcal A_{m,n,L}(a)).
\]
Each set on the right-hand side is analytic, being the continuous image of a Borel subset of a Polish space.
Thus $\{p_m<a\}$ is analytic and hence universally measurable by \cite[Theorem 21.10]{SK95}.
Consequently, $p_m$ and $v_m$ are universally measurable, and in particular $\mu$-measurable.
The same argument also shows that, for every $n\notin I_m$,
the function $d_{m,n}$ is $\mu$-measurable.

We next verify the upper-gradient property.
Let $\gamma\subset\widetilde Y_m$ be a rectifiable curve joining $x$ to $z$.
Suppose that $\int_\gamma \rho_m\,ds<\infty$ and $p_m(x)<\infty$, and fix $\delta>0$.
Choose $n\notin I_m$ such that $w(y_{m,n})+d_{m,n}(x)\le p_m(x)+\delta$.
By the definition \eqref{eq:abstract vm definition}, we can choose a rectifiable curve $\eta\subset\widetilde Y_m$ with $\eta(0)\in A_{m,n}$ and $\eta(\ell_\eta)=x$ such that $\int_\eta\rho_m\,ds\le d_{m,n}(x)+\delta$.
Concatenating $\eta$ with $\gamma$, we obtain an admissible curve in the definition of $d_{m,n}(z)$, and hence
\[
\begin{split}
p_m(z)
&\le w(y_{m,n})+d_{m,n}(z)\\
&\le w(y_{m,n})+\int_\eta\rho_m\,ds+\int_\gamma\rho_m\,ds\\
&\le w(y_{m,n})+d_{m,n}(x)+\delta+\int_\gamma\rho_m\,ds\\
&\le p_m(x)+2\delta+\int_\gamma\rho_m\,ds.
\end{split}
\]
Letting $\delta\to0$ and then applying the same argument to the reversed curve yields $|p_m(z)-p_m(x)|\le \int_\gamma \rho_m\,ds$.
If $p_m(x)=\infty$, then $p_m(z)=\infty$ as well, since otherwise the reversed-curve argument would imply $p_m(x)<\infty$.
It follows from the truncation in \eqref{eq:abstract vm definition} that
\[
|v_m(z)-v_m(x)|\le\int_\gamma\rho_m\,ds.
\]
Thus $\rho_m$ is an upper gradient of $v_m$ on $\widetilde Y_m$, and hence on $Y_m$.
Since $|v_m|\le K_m$ and $\mu(\Om)<\infty$, we have $v_m\in N^{1,1}(Y_m)$.
By Proposition~\ref{prop:localized isometric removability}, $v_m$ extends to a function in $N^{1,1}(\Om)$, still denoted by $v_m$, and therefore
\[
\int_\Om g_{v_m}\,d\mu\le\int_X\rho_m\,d\mu.
\]
Since $Y_m$ is $1$-quasiopen, we choose the extension to agree pointwise on $Y_m$ with the function defined by \eqref{eq:abstract vm definition}, and still denote it by $v_m$.

\medskip
\noindent\textbf{Step 5. $L^1$-convergence and conclusion.}
\par\nopagebreak

We first estimate the path distances $d_{m,n}$.
Fix $n\notin I_m$ and $L>0$, and set $d_{m,n}^L\coloneqq\min\{d_{m,n},L\}$ on $Y_m$.
We verify that $\rho_m$ is an upper gradient of $d_{m,n}^L$ on $Y_m$.
Let $\gamma\subset Y_m$ be a rectifiable curve joining $x$ to $z$.
The upper-gradient inequality is immediate if $\int_\gamma\rho_m\,ds=\infty$ or $d_{m,n}^L(x)=d_{m,n}^L(z)$.
Otherwise, reversing $\gamma$ if necessary, we may assume that $d_{m,n}^L(x)<d_{m,n}^L(z)$, which implies $d_{m,n}(x)=d_{m,n}^L(x)<L$.
For every $\delta>0$, by the definition of $d_{m,n}(x)$, there exists a rectifiable curve $\eta\subset\widetilde Y_m$ joining a point of $A_{m,n}$ to $x$ such that $\int_\eta\rho_m\,ds\le d_{m,n}(x)+\delta$.
Concatenating $\eta$ with $\gamma$ we obtain an admissible curve in the definition of $d_{m,n}(z)$, and hence
\[
d_{m,n}^L(z)
\le d_{m,n}(z)
\le \int_\eta\rho_m\,ds+\int_\gamma\rho_m\,ds
\le d_{m,n}^L(x)+\delta+\int_\gamma\rho_m\,ds.
\]
Letting $\delta\to0$, we obtain
\[
|d_{m,n}^L(z)-d_{m,n}^L(x)|
\le \int_\gamma\rho_m\,ds.
\]

On the other hand, the constant curves based at points of $A_{m,n}$ show that $d_{m,n}^L=0$ on $A_{m,n}$.
Since $Y_m$ is $1$-quasiopen, by Proposition \ref{prop:localized isometric removability} we can extend $d_{m,n}^L$ to
$d_{m,n}^L\in N^{1,1}(\Om)$ with $g_{d_{m,n}^L}\le \rho_m$ $\mu$-a.e. in $\Om$.
Since $\mu(A_{m,n})>\mu(B_{m,n})/3$ and $\mu(3B_{m,n})\le C\mu(B_{m,n})$, the zero set of $d_{m,n}^L$ has measure comparable to $\mu(3B_{m,n})$.
Using this fact to control the average $(d_{m,n}^L)_{3B_{m,n}}$ and then applying the Poincar\'e inequality, we obtain
\begin{equation}\label{eq:abstract set distance estimate}
\begin{split}
\int_{3B_{m,n}}d_{m,n}^L\,d\mu
&\le C\int_{3B_{m,n}}|d_{m,n}^L-(d_{m,n}^L)_{3B_{m,n}}|\,d\mu\\
&\le C\ell_m\int_{3\lambda B_{m,n}}\rho_m\,d\mu.
\end{split}
\end{equation}
The choice of $A_0$ ensures that $3\lambda B_{m,n}\Subset\Om$.
Letting $L\to\infty$ in \eqref{eq:abstract set distance estimate} and using monotone convergence, we have
\[
\int_{3B_{m,n}}d_{m,n}\,d\mu\le C\ell_m\int_{3\lambda B_{m,n}}\rho_m\,d\mu.
\]
Since the balls $3\lambda B_{m,n}$ have uniformly bounded overlap, we conclude that
\begin{equation}\label{eq:abstract summed distance estimate}
\begin{split}
\sum_{n\notin I_m}\int_{3B_{m,n}}d_{m,n}\,d\mu
&\le C\ell_m\int_X\rho_m\,d\mu\to0.
\end{split}
\end{equation}

We next control the median errors.
Let $\varphi\in\Lip_c(\Om)$ satisfy $\|w-\varphi\|_{L^1(\Om)}<\delta$.
By the definition of medians and $1/4\le\beta_{m,n}\le3/4$, we know that the set
\[
S_{m,n}\coloneqq B_{m,n}\cap\{|w-\varphi(x_{m,n})|\ge|w(y_{m,n})-\varphi(x_{m,n})|\}
\]
has measure $\mu(S_{m,n})\ge \mu(B_{m,n})/4$, and hence
\begin{equation}\label{eq:median-estimate}
|w(y_{m,n})-\varphi(x_{m,n})|
\le\frac{1}{\mu(S_{m,n})}\int_{S_{m,n}}|w-\varphi(x_{m,n})|\,d\mu
\le 4\vint{B_{m,n}}|w-\varphi(x_{m,n})|\,d\mu.
\end{equation}

Using the disjointness of the balls $B_{m,n}$, the bounded overlap of the balls $3B_{m,n}$,
we obtain
\begin{equation}\label{eq:abstract median convergence}
\begin{split}
&\sum_{n\notin I_m}\int_{3B_{m,n}}|w-w(y_{m,n})|\,d\mu\\
&\quad\le\sum_{n\notin I_m}\left(\int_{3B_{m,n}}|w-\varphi(x_{m,n})|\,d\mu
+C\int_{B_{m,n}}|w-\varphi(x_{m,n})|\,d\mu\right)
\quad \text{by \eqref{eq:median-estimate}}\\
&\quad \le C\sum_{n\notin I_m}\left(\int_{3B_{m,n}}|w-\varphi|\,d\mu+\int_{3B_{m,n}} |\varphi-\varphi(x_{m,n})|\,d\mu\right)\\
&\quad\le C\|w-\varphi\|_{L^1(\Om)}
+C\mu(\Om)\sup_{d(x,y)\le4\ell_m}|\varphi(x)-\varphi(y)|.
\end{split}
\end{equation}
First letting $m\to\infty$ and using the uniform continuity of $\varphi$, and then letting $\delta\to0$
shows that the left-hand side of \eqref{eq:abstract median convergence} tends to zero.

Set
\[
Y_m^\ast\coloneqq Y_m\cap\left(\left(\bigcup_{n\notin I_m}3B_{m,n}\right)\setminus(F_m\cup R_m)\right).
\]
For $z\in3B_{m,n}\cap Y_m^\ast$ and $n\notin I_m$, the definition of $p_m$ gives $p_m(z)\le w(y_{m,n})+d_{m,n}(z)$, and hence
\[
(p_m(z)-w(z))_+\le d_{m,n}(z)+|w(z)-w(y_{m,n})|.
\]
Since the balls $3B_{m,n}$,
$n\notin I_m$, cover $Y_m^\ast$, equations \eqref{eq:abstract summed distance estimate} and \eqref{eq:abstract median convergence} yield
\begin{equation}\label{eq:abstract upper deviation}
\int_{Y_m^\ast}(p_m-w)_+\,d\mu
\le\sum_{n\notin I_m}\int_{3B_{m,n}}d_{m,n}\,d\mu
+\sum_{n\notin I_m}\int_{3B_{m,n}}|w-w(y_{m,n})|\,d\mu\to0.
\end{equation}

To estimate the opposite deviation, fix $z\in Y_m^\ast$, $j\notin I_m$, $x\in A_{m,j}$, and a rectifiable curve $\gamma\subset\widetilde Y_m$ joining $x$ to $z$.
If $\ell_\gamma\ge\ell_m$, then using $x,z\notin H_m$, as well as \eqref{eq:abstract long curve estimate} and the inequality $w(x)\le w(y_{m,j})$, we have
\[
w(y_{m,j})+\int_\gamma\rho_m\,ds\ge w(x)+\int_\gamma\rho_m\,ds\ge w(z).
\]
If $\ell_\gamma<\ell_m$, then $z\in2B_{m,j}$ and $w(y_{m,j})+\int_\gamma\rho_m\,ds\ge w(y_{m,j})$.
Taking the infimum over all source points, curves, and indices gives
\[
p_m(z)\ge\min\left\{w(z),\inf_{\substack{j\notin I_m\\ z\in2B_{m,j}}}w(y_{m,j})\right\}.
\]
Consequently,
\[
(w(z)-p_m(z))_+
\le (w(z)-\inf_{\substack{j\notin I_m\\ z\in2B_{m,j}}}w(y_{m,j}))_+
\le\sum_{j\notin I_m}|w(z)-w(y_{m,j})|\ch_{2B_{m,j}}(z).
\]
Integration and \eqref{eq:abstract median convergence} show that
\begin{equation}\label{eq:abstract lower deviation}
\begin{split}
\int_{Y_m^\ast}(w-p_m)_+\,d\mu
&\le\sum_{j\notin I_m}\int_{2B_{m,j}}|w-w(y_{m,j})|\,d\mu\to0.
\end{split}
\end{equation}
Combining \eqref{eq:abstract upper deviation} and \eqref{eq:abstract lower deviation} we know $\int_{Y_m^\ast}|p_m-w|\,d\mu\to0$.
Since $|w|\le K_m$ on $Y_m^\ast$, by the definition \eqref{eq:abstract vm definition}, we have
\[
\int_{Y_m^\ast}|v_m-w|\,d\mu\le\int_{Y_m^\ast}|p_m-w|\,d\mu\to0.
\]

Let
\[
Q_m\coloneqq F_m\cup R_m\cup\bigcup_{n\in I_m}3B_{m,n}.
\]
The maximality of the centers shows that $\Om\setminus Y_m^\ast$ is contained in $Q_m$ up to the null set $N_m$.
By \eqref{eq:choose Km} and \eqref{eq:Im measure}, it follows that $\mu(Q_m)\to0$ and $K_m\mu(Q_m)\to0$, where we also use
\[
K_m\mu(\{|w|>K_m\})\le\int_{\{|w|>K_m\}}|w|\,d\mu.
\]
Since $w\in L^1(\Om)$, absolute continuity of the integral gives
\[
\int_{Q_m}|v_m-w|\,d\mu\le K_m\mu(Q_m)+\int_{Q_m}|w|\,d\mu\to0.
\]
Thus $v_m\to w$ in $L^1(\Om)$.
By lower semicontinuity of the total variation, now we have
\[
\|Dw\|(\Om)\le\liminf_{m\to\infty}\|Dv_m\|(\Om)\le\liminf_{m\to\infty}\int_X\rho_m\,d\mu.
\]
By \eqref{eq:controlled rho estimate} and letting $\varepsilon_m\to0$,
it follows that \eqref{eq:controlled BV estimate} holds when $h$ is lower semicontinuous.
For a general nonnegative $h\in L^1(\Om,\nu)$, by Theorem~\ref{thm:VitaliCar}, there exist lower semicontinuous functions $h_s\ge h$ with $h_s\to h$ in $L^1(\Om,\nu)$.
\end{proof}
	
By using the empty control at the scales $\ell_m=1/m$, we get the following.

\begin{theorem}\label{thm:main in text}
Let $Z\subset\Om$ satisfy $\capa_1(Z)=0$, let $w\in L^1_{\loc}(\Om)$, and suppose that there exist $\nu\in \mathcal M(\Om,(800C_1C_S^2C_d^6 \widetilde c_*)^{-1})$, $M\ge1$, and a nonnegative function $h\in L^1(\Om,\nu)$ such that $\nu(\Om)<\infty$ and
\[
\Lip_{w,\nu}^{\widetilde c_*,M}(x)\le h(x)<\infty
\quad\textrm{for every }x\in\Om\setminus Z.
\]
Then $w\in D^{\BV}(\Om)$ with $\|Dw\|\le Ch\nu$.
\end{theorem}
	
Recall that $D^{\BV}(\Om)$ denotes the Dirichlet space, that is, the BV space without the requirement
$w\in L^1(\Om)$.

\begin{proof}
Let $U\Subset\Om$ be a bounded open set.
 For the restricted measure, we have $\nu|_U\in\mathcal M(U,(800C_1C_S^2C_d^6\widetilde c_*)^{-1})$, and the generalized Lipschitz number relative to $U$ agrees with the original one at sufficiently small radii.
Apply Proposition~\ref{prop:controlled BV criterion} in $U$ with $E=\emptyset$, the capacity-zero set $Z\cap U$, $\ell_m=1/m$, and the empty $(\ell_m,w)$-controls.
Consequently,
\[
\|Dw\|(U)\le C\int_Uh\,d\nu.
\]
 Inner regularity on open sets and outer regularity of the Radon measures now give $w\in D^{\BV}(\Om)$ and $\|Dw\|\le Ch\nu$.
\end{proof}

For bounded functions, a set of finite codimension-one Hausdorff measure can be controlled by small balls
at a cost proportional to the global oscillation.

\begin{proposition}\label{prop:bounded finite H exceptional set}
Let $\Om$ be bounded, let $E\subset\Om$ satisfy $\mathcal H(E)<\infty$, and let $w\colon\Om\to\R$ be bounded and measurable.
Suppose that $M\ge1$, $\nu\in\mathcal M(\Om,(800C_1C_S^2C_d^6\widetilde c_*)^{-1})$, $\nu(\Om)<\infty$, and $h\in L^1(\Om,\nu)$ is nonnegative.
If
\[
\Lip_{w,\nu}^{\widetilde c_*,M}(x)\le h(x)<\infty
\quad\textrm{for every }x\in\Om\setminus E,
\]
then $w\in D^{\BV}(\Om)$ and
\[
\|Dw\|(\Om)
\le C\int_\Om h\,d\nu
+C\osc_\Om w\,\mathcal H(E),
\]
where $C$ depends only on the doubling and Poincar\'e constants and $M$.
\end{proposition}

\begin{proof}
Set $L\coloneqq \osc_\Om w$ and $\ell_m\coloneqq 1/m$, $m\in\N$.
By the definition of $\mathcal H$,
for every $m$ there are balls $B_{m,i}=B(z_{m,i},r_{m,i})$ with $2B_{m,i}\Subset\Om$ such that
\[
E\subset\bigcup_iB_{m,i},
\quad 0<r_{m,i}<\ell_m/8,
\quad \sum_i\frac{\mu(B_{m,i})}{r_{m,i}}\le\mathcal H(E)+2^{-m}.
\]
 More precisely, the radii may be chosen smaller than $\min\{\ell_m/8,\dist(z_{m,i},X\setminus\Om)/4\}$.
Use the control $\mathcal S_m=(\emptyset,\{B_{m,i}\}_i,0)$.
Then by $\osc_{B_{m,i}}w\le L$, and doubling, we obtain that
\[
\mu(H(\mathcal S_m))\le\sum_i\mu(2B_{m,i})\le C\ell_m(\mathcal H(E)+2^{-m})\to0.
\]
Moreover,
\[
\mathcal C_w(\mathcal S_m)=\sum_i\frac{\osc_{B_{m,i}}w}{r_{m,i}}\mu(2B_{m,i})\le CL(\mathcal H(E)+2^{-m}).
\]
Since $w$ is bounded and $\Om$ is bounded, $w\in L^1(\Om)$.
Applying Proposition~\ref{prop:controlled BV criterion} with $Z=\emptyset$ gives
\[
\|Dw\|(\Om)\le C\int_\Om h\,d\nu+CL\mathcal H(E).
\]
This is the desired estimate.
\end{proof}

	\begin{proof}[Proof of Theorem \ref{thm:main BV}]
		$\ref{item-1-thm 1.3}\Rightarrow \ref{item-2-thm 1.3}$: This was shown at the end of Subsection \ref{sec:only if}.
		
		$\ref{item-2-thm 1.3}\Rightarrow \ref{item-1-thm 1.3}$: Assume that $\Om$ is bounded; generalization to the unbounded case is then easy.
		By assumption, we have $\Lip_{u^*,\nu}^{\widetilde c_*,M}(x)\le 1$ for some
		$\nu \in \mathcal M(\Om,(800C_1C_S^2C_d^6\widetilde c_*)^{-1})$, $M\ge1$ 
		and for every $x\in \Om$, and so
		by applying Theorem \ref{thm:main in text} with $Z=\emptyset$, $w=u^*$, and $h=1$, we get
		$u\in D^{\BV}(\Om)$ with $|Du|\le C\nu$.
		Since $u\in L^1(\Om)$ by assumption, we get $u\in\BV(\Om)$.
	\end{proof}
	
	\begin{proof}[Proof of Theorem \ref{thm:main Sob}]
		$\ref{item-1-thm 1.4}\Rightarrow \ref{item-2-thm 1.4}$: This was shown at the end of Subsection \ref{sec:only if}.
		
		$\ref{item-2-thm 1.4}\Rightarrow \ref{item-1-thm 1.4}$: As above, assume that $\Om$ is bounded.
		By assumption, we have $\Lip_{u^*,a}^{\widetilde c_*,M}(x)\le 1$ for some
		nonnegative $a\in L^p(\Om)\subset L^1(\Om)$, $M\ge1$,
		and for every $x\in \Om$.
		By Lemmas \ref{lem:measures belonging to M} and \ref{lem:h nv is also M}, we know that
		$a\,d\mu\in \mathcal M(\Om,(800C_1C_S^2C_d^6\widetilde c_*)^{-1})$, and so
		by applying Theorem \ref{thm:main in text} with $Z=\emptyset$, $d\nu\coloneqq a\,d\mu$,
		$w=u^*$, and $h=1$, we get $u\in D^{\BV}(\Om)$ with $d\|Du\|\le Ca\,d\mu$.
		Since $u,a\in L^p(\Om)$, applying Lemma~\ref{lem:BV-Sobolev density} with $Ca$, there exists a representative $v\in N^{1,p}_{\loc}(\Om)$ such that $v=u$ $\mu$-almost everywhere.
		For $p<\infty$, we have $g_v\le Ca$ almost everywhere; for $p=\infty$, the representative is locally Lipschitz and $\Lip_v\le Ca$ almost everywhere.
		Since $v=u$ almost everywhere and $u\in L^p(\Om)$, we have $v\in L^p(\Om)$; together with the global $L^p$ bound for the gradient, it follows that $v\in N^{1,p}(\Om)$.

		If $\Om$ is unbounded, by the same argument on bounded relatively compact open subsets we have $u\in\BV_{\loc}(\Om)$ and $d\|Du\|\le Ca\,d\mu$. Applying Lemma \ref{lem:BV-Sobolev density} again and the preceding argument then implies $v\in N^{1,p}(\Om)$.
	\end{proof}
	
\subsection{Federer's characterization of sets of finite perimeter}\label{subsec:perimeter}

The ordinary boundary and the ordinary Lipschitz number share the same defect: both are too sensitive to small oscillations and irregular representatives.
For instance, if $F\subset\R^n$ is measurable and $\mathcal H^{n-1}(\partial F)<\infty$, then $F$ is a set of finite perimeter, but the converse is false in general.
Federer's theorem replaces the topological boundary by the measure-theoretic boundary.
For a measurable set $F\subset\R^n$, this is the set $\partial^*F$ of points $x\in\R^n$ where
\begin{equation}\label{eq:measure theoretic boundary}
\limsup_{r\to 0}\frac{\mathcal L^n(B(x,r)\cap F)}{\mathcal L^n(B(x,r))}>0
\quad\textrm{and}\quad
\limsup_{r\to 0}\frac{\mathcal L^n(B(x,r)\setminus F)}{\mathcal L^n(B(x,r))}>0;
\end{equation}
recall Subsection \ref{sec:BV functions}.
Then $F$ has finite perimeter if and only if $\mathcal H^{n-1}(\partial^*F)<\infty$, see Federer \cite[Section 4.5.11]{Fed} and Evans--Gariepy \cite{EvGa}.
Several variants of this principle are known in $\R^n$: the zero density threshold in \eqref{eq:measure theoretic boundary} can be replaced by a positive constant \cite[Section 6]{EB}, the upper densities can be replaced by lower densities \cite[Theorem 1.1]{L-Fed-new}, and a fine boundary can be used instead of the measure-theoretic boundary \cite[Theorem 1.1]{L-Fed}.

Metric counterparts of this theory were developed in \cite{A1}, and a Federer-type characterization was established in \cite[Theorem 1.1]{FedChar}.
Thus, under the standing doubling and Poincar\'e assumptions, finite perimeter admits a characterization in terms of a suitable codimension-one boundary also in metric measure spaces.

There is also a more analytic characterization: a set of finite perimeter $F\subset\R^n$ can be characterized by the existence of a sequence $\{u_j\}_{j=1}^{\infty}$ in $\liploc(\R^n)$ such that $u_j\to\ch_F$ in $L_{\loc}^1(\R^n)$ and
\[
\liminf_{j\to\infty}\int_{\R^n}\Lip_{u_j}\,d\mathcal L^n<\infty;
\]
see e.g.\ \cite[Theorem 3.9 and Remark 3.10]{AFP}.
Our generalized Lipschitz numbers connect this approximation point of view with Federer-style boundary criteria.
The bounded exceptional-set result, Proposition \ref{prop:bounded finite H exceptional set}, makes it possible to derive the following characterization 
(Corollary \ref{cor:finite perimeter char}) directly from the generalized Lipschitz-number construction.

Just as we consider the representative $u^*$ throughout the paper, for a measurable set $F\subset X$ we need to consider a suitable pointwise representative, which we choose to be $F^*\coloneqq \{\ch_F^*=1\}$.
Since $\ch_F^*$ is Borel, $F^*$ is Borel, and the Lebesgue differentiation theorem gives $\ch_{F^*}=\ch_F$ $\mu$-almost everywhere.
We define a boundary $\partial_X F$ as the set of points $x\in X$ where
\[
\limsup_{r\to 0}
\frac{\rcapa_1(B(x,r)\cap F^*,B(x,2r))}{\rcapa_1(B(x,r),B(x,2r))}\ge \widetilde c_*
\quad\textrm{and}\quad
\limsup_{r\to 0}\frac{\rcapa_1(B(x,r)\setminus F^*,B(x,2r))}{\rcapa_1(B(x,r),B(x,2r))}\ge \widetilde c_*.
\]
As usual, $\Om\subset X$ is an open set.
\begin{corollary}\label{cor:finite perimeter char}
Let $F\subset X$ be measurable.
Then $P(F,\Om)<\infty$ if and only if $\mathcal H(\partial_XF\cap\Om)<\infty$.
\end{corollary}

\begin{proof}
$\Rightarrow$: We assume $P(F,\Om)<\infty$.

Let $0<\gamma\le1/2$ be the constant in \eqref{eq:pointwise-properties}.
For $\mathcal{H}$-a.e. $y\in \Om$, we know $\ch_F^*(y)=1$ or $\ch_F^*(y)\le1-\gamma$.
We denote the $\mathcal{H}$ negligible set as $N$.

By Proposition \ref{prop:necessity}, we have $\Lip_{\ch_F^*,\|D\ch_F\|}^{\widetilde c_*/2,2\lambda}(x)\le 2C\widetilde c_*^{-1}$ for all $x\in\Om$.
Given $x\in\Om$, for all sufficiently small $r>0$, we find a set $U_r\subset B(x,r)$ such that
\begin{equation}\label{eq:capa 2cn estimate}
\frac{\rcapa_1(B(x,r)\setminus U_r,B(x,2r))}{\rcapa_1(B(x,r),B(x,2r))}\le \frac{2}{3}\widetilde c_*
\end{equation}
and
\begin{equation}\label{eq:perimeter oscillation}
\frac{\osc_{U_r} \ch_{F}^*}{2r}
\frac{\mu(B(x,2\lambda r))}{\|D\ch_F\|(B(x,2\lambda r))}
\le 3C\widetilde c_*^{-1}.
\end{equation}
We may take $U_r$ Borel, as in the construction in the proof of Proposition \ref{prop:necessity}.
After modifying $U_r$ on a $\mathcal{H}$-negligible set, we may assume that $U_r\cap N=\emptyset$.

Set
\[
a\coloneqq\frac{\gamma\widetilde c_*}{48C},
\qquad
S_0\coloneqq\left\{x\in\Om:\limsup_{t\to0}
\frac{\|D\ch_F\|(B(x,t))}{\mu(B(x,t))/t}>a\right\}.
\]
For every $R>0$ and $x\in S_0$, there exists $B(x,t)$ with $t<R/5$ and $\|D\ch_F\|(B(x,t))>a\mu(B(x,t))/t$.
By a standard $5$-covering argument, we may find a countable pairwise disjoint $\{B(x_i,t_i)\}_i$ such that $S_0\subset \bigcup_i B(x_i,5t_i)$.
Hence, by doubling and disjointness,
\begin{align*}
\mathcal H_R(S_0)
&\le\sum_i\frac{\mu(B(x_i,5t_i))}{5t_i}
\le C_d^3\sum_i\frac{\mu(B(x_i,t_i))}{t_i}\\
&\le\frac{C_d^3}{a}\sum_i\|D\ch_F\|(B(x_i,t_i))
\le\frac{C_d^3}{a}P(F,\Om).
\end{align*}
Letting $R\to0$ gives $\mathcal H(S_0)\le C_d^3a^{-1}P(F,\Om)$.
With $S\coloneqq N\cup S_0$, we therefore have $\mathcal H(S)\le C_d^3a^{-1}P(F,\Om)<\infty$.

Fix $x\in\Om\setminus S$.
There is $R_x>0$, small enough that $B(x,2\lambda r)\Subset\Om$ and $r<\tfrac18\diam X$, such that \eqref{eq:capa 2cn estimate}, \eqref{eq:perimeter oscillation}, and
\[
\frac{\|D\ch_F\|(B(x,2\lambda r))}{\mu(B(x,2\lambda r))/(2\lambda r)}
\le2a=\frac{\gamma\widetilde c_*}{24C}
\]
hold for all $0<r\le R_x$.
Combining this with \eqref{eq:perimeter oscillation} yields
\[
\osc_{U_r}\ch_F^*
\le\frac{6C}{\widetilde c_*}
r\frac{\|D\ch_F\|(B(x,2\lambda r))}{\mu(B(x,2\lambda r))}
\le\frac{\gamma}{8\lambda}<\gamma.
\]
Since $U_r\cap N=\emptyset$, for each such $r$, we have either $U_r\subset F^*$ or $U_r\cap F^*=\emptyset$.
It remains to show that the same one of these two alternatives holds for all $0<r\le R_x$.

In fact,
\begin{align*}
\frac{\mu(B(x,r)\setminus U_r)}{\mu(B(x,r))}
&\le C_SC_d\frac{\rcapa_1(B(x,r)\setminus U_r,B(x,2r))}{\rcapa_1(B(x,r),B(x,2r))}
\quad\textrm{by \eqref{eq:rcapa in small ball}}
\\
&\le \frac{2\widetilde c_*}{3}C_SC_d
\quad\textrm{by }\eqref{eq:capa 2cn estimate}\\
&< (1+C_d)^{-1}
\quad\textrm{by }\eqref{eq:cn definition},
\end{align*}

For $r/2\le s\le r\le R_x$, by doubling we know $\mu(B(x,r))\le C_d\mu(B(x,s))$, and consequently
\begin{align*}
\mu(U_s\cap U_r)
&\ge\mu(B(x,s))-\mu(B(x,s)\setminus U_s)
-\mu(B(x,s)\setminus U_r)\\
&\ge\bigl(1-\tfrac23C_SC_d\widetilde c_*(1+C_d)\bigr)\mu(B(x,s))>0.
\end{align*}
Thus $U_s$ and $U_r$ intersect, so they must select the same side of $F^*$.
Any two radii in $(0,R_x]$ can be joined by a finite chain of radii with consecutive ratios between $1/2$ and $1$, and hence we have either $U_r\subset F^*$ or $U_r\cap F^*=\emptyset$. 

Thus by \eqref{eq:capa 2cn estimate}, in both cases we have $x\notin \partial_X F$.
We conclude that $\partial_X F\cap \Om\subset S$, and so $\mathcal H(\partial_X F\cap \Om)<\infty$.

$\Leftarrow$: Assume that $\mathcal H(\partial_XF\cap\Om)<\infty$.
	Suppose that $\Om$ is bounded; then the unbounded case follows from an increasing relatively compact exhaustion of $\Om$ and the lower semicontinuity of perimeter.
Set $E\coloneqq \partial_XF\cap\Om$.
	For every $x\in\Om\setminus E$, we have
		\[
		\limsup_{r\to 0}\frac{\rcapa_1(B(x,r)\cap F^*,B(x,2r))}{\rcapa_1(B(x,r),B(x,2r))}<\widetilde{c_*}
		\quad\textrm{or}\quad \limsup_{r\to 0}\frac{\rcapa_1(B(x,r)\setminus F^*,B(x,2r))}{\rcapa_1(B(x,r),B(x,2r))}<\widetilde{c_*}.
		\]
	In the first case choose $U_r\coloneqq B(x,r)\setminus F^*$, and in the second case choose $U_r\coloneqq B(x,r)\cap F^*$.
	Since $\ch_{F^*}$ is constant on $U_r$, by these choices we have
	\[
	\Lip_{\ch_{F^*},\mu}^{\widetilde c_*,1}(x)=0
	\quad\textrm{for every }x\in\Om\setminus E.
	\]
	By Lemma \ref{lem:measures belonging to M} and \eqref{eq:cn definition}, $\mu\in \mathcal M(\Om,(800C_1C_S^2C_d^6\widetilde c_*)^{-1})$.
	Applying Proposition \ref{prop:bounded finite H exceptional set} with $w=\ch_{F^*}$, $\nu=\mu$, $M=1$, and $h=0$ yields
	\[
	\|D\ch_{F^*}\|(\Om)
	\le C\mathcal H(\partial_XF\cap\Om)<\infty.
	\]
	Since $\ch_{F^*}=\ch_F$ almost everywhere, this proves $P(F,\Om)<\infty$.
\end{proof}

\begin{remark}
In the usual Federer's characterization, one uses the measure-theoretic boundary that is defined by means of measure densities, rather than capacitary densities.
Thus it is natural to ask:
\begin{itemize}
\item Can one obtain analogs of Theorems \ref{thm:main BV} and \ref{thm:main Sob} via using a version of generalized Lipschitz numbers that is defined by using the measure $\mu$ in place of the variation capacity $\rcapa_1$?
\end{itemize}
\end{remark}

\endgroup

	\begingroup
\emergencystretch=3em
\section{Fine differentiability}\label{sec:fine differentiability}

Recall from \eqref{eq:def-approx-limit} the definition of approximate limit.
In the Euclidean setting, it is a classical fact that Sobolev and BV functions are
approximately differentiable almost everywhere.
More precisely, if $u\in \BV_{\loc}(\R^n)$, then for a.e. $x\in\R^n$ there exists a vector $\nabla u(x)\in\R^n$ such that
\[
\aplim_{y\to x}
\frac{|u(y)-u^*(x)-\langle \nabla u(x),y-x\rangle|}{|y-x|}=0;
\]
see e.g. Ambrosio--Fusco--Pallara \cite[Theorem 3.83]{AFP}. 
For Sobolev functions this vector is the weak gradient, while for BV functions it is the density of the absolutely continuous part of $Du$.
Haj{\l}asz--Mal\'y \cite[Theorem 1.1]{HaMa} obtained an analogous approximate differentiability result for the Hardy--Littlewood maximal function.

The picture can be strengthened from measure density to the $1$-fine topology.
In the Euclidean setting, for every $u\in \BV_{\loc}(\R^n)$, and hence every Sobolev function, $u^*$ is $1$-finely differentiable a.e.; see \cite[Theorem 3.6]{L-FQ}. 
We can extend this result into metric spaces.

We now recall the metric first-order differentiable structure
on which the definition of fine differentiability is based.
Cheeger \cite[Theorem 4.38]{Che} showed that complete doubling metric measure spaces supporting a Poincar\'e inequality admit a measurable differentiable structure: there exists a countable collection $\{(X_\alpha,\xi^\alpha)\}$ of measurable sets $X_\alpha$ and Lipschitz coordinate maps $\xi^\alpha=(\xi^\alpha_1,\ldots,\xi^\alpha_{N_\alpha})\colon X\to\R^{N_\alpha}$ with $\mu(X\setminus\bigcup_\alpha X_\alpha)=0$, such that every Lipschitz function $f$ has a differential $d^\alpha f(x)\in\R^{N_\alpha}$ satisfying
\[
\lim_{y\to x}
\frac{|f(y)-f(x)-d^\alpha f(x)\cdot(\xi^\alpha(y)-\xi^\alpha(x))|}{d(y,x)}=0
\]
for a.e. $x\in X_\alpha$.
A function $w$ is approximately differentiable at $x\in X_\alpha$ with respect to this structure if there exists $L^\alpha w(x)\in\R^{N_\alpha}$ such that
\[
\aplim_{y\to x}
\frac{|w(y)-w(x)-L^\alpha w(x)\cdot(\xi^\alpha(y)-\xi^\alpha(x))|}{d(y,x)}=0,
\]
as in \cite[Definition 2.9]{DCIKS}.
In a doubling metric measure space, the approximate differential is unique at $\mu$-a.e. point by Keith \cite[Remark 3.6]{Kei-P}.

By Durand-Cartagena--Ihnatsyeva--Korte--Szuma\'nska \cite[Lemma 2.10]{DCIKS}, after redefining the given measurable differentiable structure on a set of measure zero if necessary, one obtains a structure with respect to which the approximate differential is unique at every point of each chart.
We use this representative throughout the section.
In particular, at every point of a fixed chart, no nonzero chart-linear function has approximate differential zero.
We identify a coefficient vector $v=(v_i)_{i=1}^{N_\alpha}$ with the covector $\sum_i v_i\,d\xi_i^\alpha$.
The notation $|v|_{C,x}$ refers to its intrinsic Cheeger cotangent norm, whereas $|v|$ denotes the Euclidean norm of its coefficients.


\begin{definition}\label{def:fine differentiability}
Let $(X,d,\mu)$ be equipped with a measurable differentiable structure $\{(X_\alpha,\xi^\alpha)\}_\alpha$, where $\xi^\alpha\colon X\to\R^{N_\alpha}$ are the coordinate maps.
We say that a function $w\colon \Om\to[-\infty,\infty]$ is $1$-finely differentiable at $x\in \Om\cap X_\alpha$ with respect to $(X_\alpha,\xi^\alpha)$ if $w(x)\in\R$ and there exists $v\in\R^{N_\alpha}$ such that
\[
\underset{y\to x}{\finelim}
\frac{|w(y)-w(x)-\langle v,\xi^\alpha(y)-\xi^\alpha(x)\rangle|}{d(y,x)}=0.
\]
We call $v$ the fine differential of $w$ at $x$ with respect to $(X_\alpha,\xi^\alpha)$, and denote it by $d^\alpha_{\textrm{fine}}w(x)$.
\end{definition}

By \eqref{eq:rcapa in small ball}, fine differentiability is stronger than approximate differentiability.
Hence if such fine differential $v$ exists, then the uniqueness convention above gives its uniqueness, and $v$ agrees with the approximate differential.
In the Euclidean case, with the identity coordinate map $\xi(x)=x$, we write $\nabla_{\textrm{fine}}w(x)$ for this vector.


\begin{definition}[Fine-scale geometric condition]\label{def:FG}
Let $x\in\Om$.
We say that the geometric condition $(\mathrm{FG})_x$ holds if at least one of the following two alternatives is satisfied:
\begin{enumerate}
\item\label{item-1-FG}
There exist $s>1$, $R_x>0$ and $C_x\ge1$ such that for every $0<r<R\le R_x$,
\begin{equation}
\label{eq: reverse doubling ineq.-s>1}
\frac{\mu(B(x,r))}{\mu(B(x,R))}\le C_x\left(\frac{r}{R}\right)^s.
\end{equation}
\item\label{item-2-FG}
There exist $R_x>0$ and $\delta_x^*>0$ such that for every $0<r<R_x$ and every $y\in B(x,r)$,
\begin{equation}
\label{eq:FG-singleton-capacity}
\frac{\rcapa_1(\{y\},B(x,2r))}{\rcapa_1(B(x,r),B(x,2r))}
\ge\delta_x^*.
\end{equation}
\end{enumerate}
In both alternatives, we take $R_x$ sufficiently small that $B(x,2R_x)\Subset\Om$ and $R_x<\tfrac18\diam X$.
\end{definition}
\begin{lemma}\label{lem:fine-neighborhood-selection}
Let $x\in\Om$ satisfy $(\mathrm{FG})_x$ from Definition~\ref{def:FG}, and let $R'\colon\Om\setminus\{x\}\to[0,\infty]$.
Suppose that $F_m\coloneqq\{y\in\Om\setminus\{x\}:R'(y)\ge1/m\}$ is $1$-thin at $x$ for every $m\in\N$.
Then there is a $1$-finely open set $U\subset\Om$ containing $x$ such that $R'(y)\to0$ as $U\ni y\to x$.
\end{lemma}
\begin{proof}
Suppose first that Definition~\ref{def:FG} \ref{item-1-FG} holds.
Choose $\rho_m\downarrow0$ with $\rho_{m+1}<\rho_m/2$, $B(x,2\rho_1)\Subset\Om$, and $\rho_1<\min\{R_x,\tfrac18\diam X\}$, such that
\[
\sup_{0<r\le\rho_m}
\frac{\rcapa_1(F_m\cap B(x,r),B(x,2r))}{\rcapa_1(B(x,r),B(x,2r))}
\le2^{-m}.
\]
Set
\[
E\coloneqq\bigcup_{m=1}^{\infty}
\bigl(F_m\cap(B(x,\rho_m)\setminus B(x,\rho_{m+1}))\bigr).
\]
If $\rho_{k+1}\le r<\rho_k$, we have $\rho_m\le r$ for $m>k$.
By subadditivity, and \eqref{eq:rcapa over ball-measure}, we obtain
\[
\begin{split}
&\,\frac{\rcapa_1(E\cap B(x,r),B(x,2r))}{\rcapa_1(B(x,r),B(x,2r))}\\
&\le \frac{\rcapa_1(F_k\cap B(x,r), B(x,2r))}{\rcapa_1(B(x,r), B(x,2r))}+\sum_{m>k}\frac{\rcapa_1(F_m\cap B(x,\rho_m),B(x,2\rho_m))}{\rcapa_1(B(x,r),B(x,2r))}\\
&\le2^{-k}
+C\sum_{m>k}2^{-m}
\frac{\mu(B(x,\rho_m))/\rho_m}{\mu(B(x,r))/r}\\
&\le2^{-k}+CC_x\sum_{m>k}2^{-m}
\left(\frac{\rho_m}{r}\right)^{s-1}
\le C'2^{-k}.
\end{split}
\]
Thus $E$ is $1$-thin at $x$.
By \eqref{eq:characterization of fine closure}, the set $U\coloneqq\Om\setminus\overline E^{\,1}$ is $1$-finely open and contains $x$.
For $y\in U\cap(B(x,\rho_m)\setminus B(x,\rho_{m+1}))$, we have $R'(y)<1/m$, which proves the claim.

Suppose instead that Definition~\ref{def:FG} \ref{item-2-FG} holds.
For each $m$, thinness and \eqref{eq:FG-singleton-capacity} imply $F_m\cap B(x,r)=\emptyset$ for every sufficiently small $r$.
Consequently $R'(y)\to0$ in the ordinary topology, and we may take $U=\Om$.
\end{proof}

We next record the metric version of \cite[Lemma 3.2]{L-FD} under the geometric condition of Definition~\ref{def:FG}.

\begin{lemma}\label{lem:fine diff}
Let $(X,d,\mu)$ be equipped with a measurable differentiable structure $\{(X_\alpha,\xi^\alpha)\}_\alpha$, and suppose $w\colon \Om\to [-\infty,\infty]$, $x\in \Om\cap X_\alpha$, and $w(x)\in\R$.
Assume that $(\mathrm{FG})_x$ from Definition~\ref{def:FG} holds.
Then the following are equivalent:
\begin{enumerate}
\item \label{item-fine diff-1}
The function $w$ is $1$-finely differentiable at $x$ with respect to $(X_\alpha,\xi^\alpha)$.
\item \label{item-fine diff-2}
There exists $v\in\R^{N_\alpha}$ and a $1$-finely open set $U\ni x$ such that
\[
\lim_{U\cap\Om\ni y\to x}
\frac{|w(y)-w(x)-\langle v,\xi^\alpha(y)-\xi^\alpha(x)\rangle|}{d(y,x)}=0.
\]
\item \label{item-fine diff-3}
There are sets $U_k\subset \Om$, vectors $v_k\in\R^{N_\alpha}$, and positive numbers $\beta_k\to 0$ such that
\[
\limsup_{r\to 0}
\frac{\rcapa_1(B(x,r)\setminus U_k,B(x,2r))}{\rcapa_1(B(x,r),B(x,2r))}< \beta_k
\]
and
\begin{equation}
\label{lem-item-3-2}
\limsup_{U_k\ni y\to x}\frac{|w(y)-w(x)-\langle v_k,\xi^\alpha(y)-\xi^\alpha(x)\rangle|}{d(y,x)}< \beta_k
\end{equation}
for all $k\in\N$.
\end{enumerate}
\end{lemma}
\begin{proof}
Write $\xi=\xi^\alpha$, $N=N_\alpha$, and $L=\Lip(\xi)$.
For $E\subset X$, set
\[
\Theta_x(E)\coloneqq
\limsup_{r\to0}\frac{\rcapa_1(E\cap B(x,r),B(x,2r))}{\rcapa_1(B(x,r),B(x,2r))}.
\]
The implication \ref{item-fine diff-2}$\Rightarrow$\ref{item-fine diff-1} follows from the definition of the fine limit.
For \ref{item-fine diff-1}$\Rightarrow$\ref{item-fine diff-3}, set $\beta_k=1/k$ and intersect each fine neighborhood supplied by the limit with a sufficiently small metric ball contained in $\Om$.
This gives the required sets $U_k$ with $v_k=d^\alpha_{\mathrm{fine}}w(x)$.

It remains to prove \ref{item-fine diff-3}$\Rightarrow$\ref{item-fine diff-2}.
We first show that $\{v_k\}$ is Cauchy.
Otherwise there are $\varepsilon_0>0$ and indices $i_j,l_j\to\infty$ such that
\begin{equation}
\label{eq: estimate-vij-vlj}
|v_{i_j}-v_{l_j}|\ge\varepsilon_0.
\end{equation}
After passing to a subsequence, the vectors $a_j\coloneqq(v_{i_j}-v_{l_j})/|v_{i_j}-v_{l_j}|$ converge to $a\in \mathbb S^{N-1}$.
Set $E_j\coloneqq U_{i_j}\cap U_{l_j}$.
Then
\begin{equation}
\label{eq: sequence-thin-set}
\Theta_x(X\setminus E_j)
\le\Theta_x(X\setminus U_{i_j})+\Theta_x(X\setminus U_{l_j})
<\beta_{i_j}+\beta_{l_j}\to0.
\end{equation}
We obtain 
\begin{equation}
\begin{split}
\label{eq: estimate-affine}
\limsup_{E_j\ni y\to x}
\frac{|\langle a,\xi(y)-\xi(x)\rangle|}{d(y,x)}
&\le\limsup_{E_j\ni y\to x}
\frac{|\langle a_j,\xi(y)-\xi(x)\rangle|}{d(y,x)}+L|a-a_j|\\
&\le\frac{\beta_{i_j}+\beta_{l_j}}{\varepsilon_0}+L|a-a_j|\quad \textrm{by \eqref{lem-item-3-2}
and \eqref{eq: estimate-vij-vlj}}\\
&\to0.
\end{split}
\end{equation}
For every $t>0$, combining this with \eqref{eq: sequence-thin-set}, the set
\[
A_t\coloneqq\left\{y\in\Om\setminus\{x\}:
\frac{|\langle a,\xi(y)-\xi(x)\rangle|}{d(y,x)}\ge t\right\}
\]has $\Theta_x(A_t)=0$.
By \eqref{eq:rcapa in small ball}, each $A_t$ has zero measure density at $x$.
Hence the chart-linear function $y\mapsto\langle a,\xi(y)\rangle$ has approximate differential zero at $x$.
It also has the exact differential $a$, contradicting the fixed uniqueness convention.
Thus $v_k\to v\in\R^N$.

For $y\ne x$, define
\[
R(y)\coloneqq
\frac{|w(y)-w(x)-\langle v,\xi(y)-\xi(x)\rangle|}{d(y,x)}.
\]
Fix $t>0$ and $\eta>0$, and choose $k$ such that $\Theta_x(X\setminus U_k)<\eta$ and $\beta_k+L|v_k-v|<t$.
Then $\limsup_{U_k\ni y\to x}R(y)<t$, so $\Theta_x(\{R\ge t\})\le\Theta_x(X\setminus U_k)<\eta$.
Letting $\eta\to0$, it follows that every positive level set of $R$ is $1$-thin at $x$.
By Lemma \ref{lem:fine-neighborhood-selection} we can find a $1$-finely open set required in \ref{item-fine diff-2}.
\end{proof}

\begin{lemma}\label{lem:zero-BV-blowup}
Let $f\in\BV_{\loc}(\Om)$, let $x\in\Om$, and let $e$ be a pointwise function such that $e(x)=0$ and $|e(y)|\le |f|^\vee(y)$ for $\mathcal{H}$-a.e. $y\in\Om$.
Suppose that for some fixed $\lambda'\ge2$, we have
\begin{equation}
\label{eq:zero-normalized-BV-blowup}
\delta_x(r)
\coloneqq
\frac{\|Df\|(B(x,\lambda' r))}{\mu(B(x,r))}
+\frac{1}{r}\vint{B(x,\lambda' r)}|f|\,d\mu
\to 0
\quad\textrm{as }r\to0.
\end{equation}
Assume also that the geometric condition $(\mathrm{FG})_x$ in Definition~\ref{def:FG} holds.
Then
$e$ is $1$-finely differentiable at $x$ with respect to every Cheeger chart containing $x$, and
\[
d^\alpha_{\mathrm{fine}}e(x)=0.
\]
\end{lemma}
\begin{proof}
After modifying $e$ in an $\mathcal{H}$-negligible set, we may assume $|e(y)|\le|f|^\vee(y)$ for every $y\in\Om$
(note that $|f|^{\vee}(x)=0$ by \eqref{eq:zero-normalized-BV-blowup}).
Let $\beta_k\coloneqq 2^{-k}$.
Choose $c_0>0$, independent of $k$ but allowed to depend on the local data in $(\mathrm{FG})_x$ from Definition~\ref{def:FG}, which will be fixed below.
By \eqref{eq:zero-normalized-BV-blowup}, for every $k\in\N$ there exists $0<R_k<\min\{R_x,\tfrac18\diam X\}$ such that $B(x,2\lambda R_k)\Subset\Om$ and
\begin{equation}
\label{eq: delta-r-beta-k}
\sup_{0<r<R_k}\delta_x(r)\le c_0\beta_k^2.
\end{equation}
For this fixed $k$, put $r_{k,j}\coloneqq 2^{-j}R_k$ and define
\[
A_{k,j}
\coloneqq
B(x,r_{k,j})\setminus B(x,r_{k,j+1}),
\]
\[
E_{k,j}
\coloneqq
\left\{y\in A_{k,j}\colon
|e(y)|>\frac{\beta_k}{2}d(x,y)\right\}.
\]
Let $\eta_{k,j}$ be a Lipschitz cutoff satisfying $0\le \eta_{k,j}\le 1$, $\eta_{k,j}=1$ on $B(x,r_{k,j})$, $\eta_{k,j}=0$ on $X\setminus B(x,3r_{k,j}/2)$, and $g_{\eta_{k,j}}\le 2/r_{k,j}.$ Define
\[
\varphi_{k,j}\coloneqq
\eta_{k,j}\min\left\{1,\frac{4|f|}{\beta_k r_{k,j}}\right\}.
\]
Then $\varphi_{k,j}\in\BV(X)$ and $\supp \varphi_{k,j}\Subset B(x,2r_{k,j})$.
Put 
\[
F_{k,j}\coloneqq\left\{y\in B(x,r_{k,j}):|f|^\wedge(y)>\frac{\beta_k r_{k,j}}{4}\right\}.
\]
For $y\in E_{k,j}$, we know $|f|^\vee(y)\ge|e(y)|>\beta_k r_{k,j}/4$.
By Lemma \ref{lem:wedge and vee capacity} applied to $|f|$, we have
\[
\rcapa_1(E_{k,j},B(x,2r_{k,j}))
\le\rcapa_1(F_{k,j},B(x,2r_{k,j})).
\]
Moreover, $\varphi_{k,j}^\wedge=1$ on $F_{k,j}$, so this function is admissible for the
relative BV capacity of $F_{k,j}$.
By the Leibniz rule, see \cite[Theorem 1.2]{L-SL}, we have
\[
\begin{split}
\|D\varphi_{k,j}\|(B(x,2r_{k,j}))
&\le \frac{4}{\beta_k r_{k,j}}\left(\int_{B(x,2r_{k,j})}\eta_{k,j}\,d\|Df\|+\int_{B(x,2r_{k,j})} g_{\eta_{k,j}} |f|\,d\mu\right)\\
&\le \frac{C}{\beta_k r_{k,j}}\|Df\|(B(x,2r_{k,j}))+\frac{C}{\beta_k r_{k,j}^2}\int_{B(x,2r_{k,j})}|f|\,d\mu.
\end{split}
\]
By \eqref{eq:BV capa and 1-capa-relative}, we have
\[
\begin{split}
\rcapa_1(E_{k,j},B(x,2r_{k,j}))
&\le C\rcapa_{\BV}(F_{k,j},B(x,2r_{k,j}))\\
&\le \frac{C}{\beta_k r_{k,j}}\|Df\|(B(x,2r_{k,j}))
+\frac{C}{\beta_k r_{k,j}^2}\int_{B(x,2r_{k,j})}|f|\,d\mu.
\end{split}
\]Using \eqref{eq: delta-r-beta-k} and \eqref{eq:rcapa over ball-measure}, we obtain
\begin{equation}
\label{eq:single-bad-annulus-capacity}
\frac{\rcapa_1(E_{k,j},B(x,2r_{k,j}))}{\rcapa_1(B(x,r_{k,j}),B(x,2r_{k,j}))}
\le\frac{C}{\beta_k}\delta_x(r_{k,j})
\le Cc_0\beta_k.
\end{equation}

Set
\[
E_k\coloneqq\bigcup_{j=0}^\infty E_{k,j},
\quad
U_k\coloneqq\Omega\setminus E_k.
\]

Suppose first that Definition~\ref{def:FG} \ref{item-1-FG} holds.
Thus for $j\ge1$, we have
\[
\begin{split}
\frac{\rcapa_1(E_k\cap B(x,r_{k,j}),B(x,2r_{k,j}))}{\rcapa_1(B(x,r_{k,j}),B(x,2r_{k,j}))}
&\le\sum_{m\ge j}\frac{\rcapa_1(E_{k,m},B(x,2r_{k,j}))}{\rcapa_1(B(x,r_{k,j}),B(x,2r_{k,j}))}\\
&\le\sum_{m\ge j}\frac{\rcapa_1(E_{k,m},B(x,2r_{k,m}))}{\rcapa_1(B(x,r_{k,j}),B(x,2r_{k,j}))}
\quad \textrm{by }\eqref{eq: rcapa-compare-in-ball}\\
&\le \sum_{m\ge j} Cc_0\beta_k \frac{\mu(B(x,r_{k,m}))/r_{k,m}}{\mu(B(x,r_{k,j}))/r_{k,j}}
\quad \textrm{by }\eqref{eq:single-bad-annulus-capacity}\textrm{ and }\eqref{eq:rcapa over ball-measure}\\
&\le Cc_0 \beta_k
\quad \textrm{by }\eqref{eq: reverse doubling ineq.-s>1}
\textrm{ and the fact }s>1.
\end{split}
\] 
For $r_{k,j+1}\le r<r_{k,j}$, by \eqref{eq: rcapa-compare-in-ball} and the fact that $E_k\cap B(x,r)\subset E_k\cap B(x,r_{k,j})$, we can get the same estimate for every sufficiently small radius (with a new constant still denoted as $C$).
Choosing $c_0$ so that $Cc_0<1$, we obtain
\begin{equation}
\label{eq:conical-bad-set-capacity}
\limsup_{r\to0}
\frac{\rcapa_1(B(x,r)\setminus U_k,B(x,2r))}{\rcapa_1(B(x,r),B(x,2r))}
<\beta_k.
\end{equation}

Suppose next that Definition~\ref{def:FG} \ref{item-2-FG} holds.
Choose $c_0$ sufficiently small such that $Cc_0<\delta_x^*$.
We claim that $E_{k,j}=\emptyset$ for every $j$.
Indeed, if $y\in E_{k,j}$, then by \eqref{eq:FG-singleton-capacity} and \eqref{eq:single-bad-annulus-capacity},
\[
\begin{split}
\delta_x^*
&\le\frac{\rcapa_1(\{y\},B(x,2r_{k,j}))}{\rcapa_1(B(x,r_{k,j}),B(x,2r_{k,j}))}\\
&\le\frac{\rcapa_1(E_{k,j},B(x,2r_{k,j}))}{\rcapa_1(B(x,r_{k,j}),B(x,2r_{k,j}))}
\le Cc_0\beta_k
\le Cc_0,
\end{split}
\]
which is a contradiction.
Thus $E_k=\emptyset$ in this case.

Consequently, in both cases we
have \eqref{eq:conical-bad-set-capacity}.
If $y\in U_k\cap B(x,R_k)\setminus \{x\}$, then $y$ belongs to some $A_{k,j}$ but not to $E_{k,j}$.
Thus,
\[
\frac{|e(y)-e(x)|}{d(x,y)}
\le\frac{\beta_k}{2}
<\beta_k.
\]
Together with \eqref{eq:conical-bad-set-capacity}, applying Lemma \ref{lem:fine diff} \ref{item-fine diff-3} with the vectors $v_k=0$, it follows that
\[
d^\alpha_{\mathrm{fine}}e(x)=0.
\]\end{proof}

We have the following discrete convolution construction.
\begin{lemma}\label{lem:localized-discrete-convolution}
Let $u\in\BV(\Om)$, let $W\subset\Om$ be open, and let $\kappa>0$.
Then there exists $h\in\BV(\Om)$ such that
\begin{equation}
\label{eq:localized-discrete-convolution-modification}
\begin{split}
&h=u \quad\textrm{on }\Om\setminus W,\\
&h|_W\in N^{1,1}(W)\cap\Lip_{\loc}(W),\\
&\|h-u\|_{L^1(W)}\le\kappa,
\quad
\int_Wg_h\,d\mu\le C\|Du\|(W),
\end{split}
\end{equation}
where $C$ only depends on the doubling and Poincar\'e constants, and
\begin{equation}
\label{eq:variation-supported-in-modification-set}
\|D(h-u)\|(\Om\setminus W)=0.
\end{equation}
Moreover, suppose that $(\mathrm{FG})_x$ from Definition~\ref{def:FG} holds for $\mu$-a.e. $x\in\Om\setminus W$.
Set $e_0\coloneqq h-u$, and define $e=h^*-u^*$ wherever both representatives are finite, with $e=0$ elsewhere. 
Then $d^\alpha_{\mathrm{fine}}e(x)=0$ for $\mu$-a.e. $x\in\Om\setminus W$ with respect to every Cheeger chart containing $x$.
\end{lemma}

\begin{proof}
The existence of $h$ satisfying \eqref{eq:localized-discrete-convolution-modification} and \eqref{eq:variation-supported-in-modification-set} is precisely the localized discrete-convolution construction from \cite[Corollary 3.6]{LS17}.

It remains to prove the last assertion.
By \eqref{eq:variation-supported-in-modification-set}, $\|De_0\|(\Om\setminus W)=0$.
By the Radon--Nikodym differentiation theorem, see e.g. \cite[p. 82]{HKSTbook},
for every fixed dilation factor $\lambda'\ge1$, we have
\begin{equation}\label{eq:density zero}
\frac{\|De_0\|(B(x,\lambda' r))}{\mu(B(x,r))}
\to 0
\end{equation}
for $\mu$-a.e. $x\in\Om\setminus W$.

We may also require that $x$ be a density point of $\Om\setminus W$.
For sufficiently small $r$, the zero set of $e_0$ occupies at least half of $B(x,r)$.
We then have
\[
\left|\,\vint{B(x,r)} e_0\,d\mu\right|\le2\vint{B(x,r)}|e_0-(e_0)_{B(x,r)}|\,d\mu,
\quad
\vint{B(x,r)}|e_0|\,d\mu
\le3\vint{B(x,r)}|e_0-(e_0)_{B(x,r)}|\,d\mu.
\]
By the BV--Poincar\'e inequality \eqref{eq: BV-Poincare}, we have
\[
\frac{1}{r}\vint{B(x,r)}|e_0|\,d\mu
\le
C\frac{\|De_0\|(B(x,\lambda r))}{\mu(B(x,r))}
\to 0.
\]
Applying the same estimate with the radius $\lambda' r$, we obtain for any $\lambda'\ge 1$
\begin{equation}\label{eq:integral average of e zero}
\frac{1}{r}\vint{B(x,\lambda' r)}|e_0|\,d\mu\to0.
\end{equation}
Together with \eqref{eq:density zero},
this proves \eqref{eq:zero-normalized-BV-blowup} for $f=e_0$.

Outside an $\mathcal H$-negligible set, the representatives are finite by \eqref{eq:pointwise-properties},
and using again \eqref{eq:pointwise-properties} we know
\[
|h^*(y)-u^*(y)|
\le\limsup_{r\to0}\vint{B(y,r)}|h-u|\,d\mu
\le|e_0|^\vee(y).
\]
We may also choose $x$ outside the exceptional set above.
Since $|e_0|^\vee(x)=0$ by \eqref{eq:integral average of e zero}, we have $e(x)=0$.
Applying Lemma \ref{lem:zero-BV-blowup} yields $d^\alpha_{\mathrm{fine}}e(x)=0$ for $\mu$-a.e. $x\in(\Om\setminus W)\cap X_\alpha$.
\end{proof}

\begin{proposition}[Fine differentiability of Newtonian functions]\label{thm: fine-diff of N11}
Let $\Om\subset X$ be open and $u\in N^{1,1}_{\loc}(\Om)$.
Let $\{(X_\alpha,\xi^\alpha)\}_\alpha$, where $\xi^\alpha=(\xi^\alpha_1,\ldots,\xi^\alpha_{N_\alpha})$, be a Cheeger measurable differentiable structure.
Assume also that the fine-scale geometric condition $(\mathrm{FG})_x$ in Definition~\ref{def:FG} holds for $\mu$-a.e. $x\in\Om$.
Then for $\mu$-a.e. $x\in\Omega\cap X_\alpha$, $u$ is $1$-finely differentiable at $x$ with respect to the chart $(X_\alpha,\xi^\alpha)$, and
\[
d^\alpha_{\mathrm{fine}}u(x)=d^\alpha u(x).
\]
\end{proposition}

\begin{proof}
Fix a chart $(X_\alpha,\xi^\alpha)$.
Let $T_C^*X$ denote the Cheeger cotangent bundle associated with the fixed Cheeger differentiable structure.
From Lemma \ref{lem:BV-Sobolev density}
we have $\Lip_f\le Cg_f$ almost everywhere for Lipschitz $f$.
Thus by \cite[Theorem 1.10]{ES24} applied with $\omega_x(t)=Ct$, $\pi_{C,1}\colon T_C^*X\to T_1^*X$ is an isometric bijection, and satisfies $\pi_{C,1}(d_Cf)=d_1f$ for every bounded Lipschitz function $f$.
By \cite[Theorem 1.7]{ES24} with $p=1$, every $u\in N^{1,1}_{\mathrm{loc}}(\Omega)$ has a $1$-weak Sobolev differential $d_1u$, and $|d_1u|=g_u$ $\mu$-a.e. Therefore, for $\mu$-a.e. $x\in \Om\cap X_\alpha$, we define
\[
d^\alpha u(x)\coloneqq
\pi_{C,1}^{-1}(d_1u(x))
\in (T_C^*X)_x
\simeq (\mathbb R^{N_\alpha})^*.
\]It satisfies
\begin{equation}
\label{eq: N11 diff-three derivative}
|d^\alpha u(x)|_C
=|d_1u(x)|=g_u(x)
\quad \mu\text{-a.e. }x\in \Om\cap X_\alpha.
\end{equation}
These results apply locally by multiplication with compactly supported Lipschitz cutoffs; see \cite[Remark 1.2]{ES24}.
The same localization permits the use of the possibly unbounded coordinate functions.
In particular, $d^\alpha u\in L^1_{\mathrm{loc}}(\Om\cap X_\alpha;T_C^*X)$.
We may also write $d^\alpha u=\sum_{i=1}^{N_\alpha}a_i\,d\xi_i^\alpha$, with measurable coefficient functions $a_i$.

Define
\[
C_\alpha(x)\coloneqq
\sup_{\zeta\ne0}\frac{|\zeta|_{\ell^1}}{|\zeta|_{C,x}},
\]
where $|\zeta|_{\ell^1}\coloneqq \sum_{i=1}^{N_\alpha}|c_i|$ if $\zeta=\sum_{i=1}^{N_\alpha}c_i\,d\xi_i^\alpha$.

The supremum can be taken over nonzero rational coefficient vectors, so $C_\alpha$ is measurable.
After ignoring a measure-zero set, $|\cdot|_{C,x}$ is a norm on the finite-dimensional Cheeger cotangent fiber, so $C_\alpha(x)<\infty$ $\mu$-a.e. Set
\[
Y_{\alpha m}\coloneqq \{x\in X_\alpha:\ C_\alpha(x)\le m\}.
\]
Then $X_\alpha=\bigcup_{m=1}^{\infty}Y_{\alpha m}\cup N$ where $\mu(N)=0$.
On each fixed piece $Y=Y_{\alpha m}$, by \eqref{eq: N11 diff-three derivative}, we have
\begin{equation}
\label{eq: estimate ai-gu}
\sum_{i=1}^{N_\alpha}|a_i|
\le m|d^\alpha u|_{C}=m g_u
\qquad \mu\text{-a.e. on }Y.
\end{equation}
In particular, every $a_i\in L^1_{\mathrm{loc}}(\Om\cap Y)$.

Extend each $a_i\ch_Y$ by zero outside $Y$.
Choose $x\in\Om\cap Y$ to be a density point of $Y$ and a Lebesgue point of $u$, $g_u$, and these extended coefficient functions, with $u(x)=u^*(x)$.
We also require $(\mathrm{FG})_x$ from Definition~\ref{def:FG} and the pointwise identities above at $x$.
These conditions hold for $\mu$-a.e. $x\in\Om\cap Y$.
Choose $r_0>0$ so small that $B(x,2\lambda r_0)\Subset \Omega$.
Define
\[
\ell_x(y)=u(x)+
\sum_{i=1}^{N_\alpha}
a_i(x)\bigl(\xi_i^\alpha(y)-\xi_i^\alpha(x)\bigr).
\]
Set $v\coloneqq u-\ell_x$.
Then $v(x)=0$.
By linearity of the Sobolev-Cheeger differential,
\[
d^\alpha v
=
d^\alpha u-d^\alpha u(x)
=
\sum_{i=1}^{N_\alpha}
(a_i-a_i(x))\,d\xi_i^\alpha
\quad \mu\text{-a.e. on }Y.
\]
Since each coordinate function $\xi_i^\alpha$ is Lipschitz, we have $|d\xi_i^\alpha|_{C}\le \Lip(\xi_i^\alpha)$ $\mu$-a.e. Hence,
\[
g_v=|d^\alpha v|_{C}
\le C
\sum_{i=1}^{N_\alpha}|a_i-a_i(x)|
\quad \mu\text{-a.e. on Y.}
\]
For $\mu$-a.e. $y\in\Om\setminus Y$, since $g_u+g_{\ell_{x}}$ is a $1$-weak upper gradient of $v$, we have
\[
g_v(y)\le g_u(y)+\sum_{i=1}^{N_\alpha}|a_i(x)|g_{\xi_i^\alpha}(y)
\le g_u(y)+Cm g_u(x).
\]
Consequently,
\[
\begin{split}
\vint{B(x,r)}g_v\,d\mu
&\le\frac{C}{\mu(B(x,r))}\sum_{i=1}^{N_\alpha}
\int_{B(x,r)\cap Y}|a_i(y)-a_i(x)|\,d\mu(y)\\
&\quad+\frac{1}{\mu(B(x,r))}\int_{B(x,r)\setminus Y}g_u(y)\,d\mu(y)
+Cm g_u(x)\frac{\mu(B(x,r)\setminus Y)}{\mu(B(x,r))}.
\end{split}
\]
Recall that $x$ is a density point of $Y$, and a Lebesgue point of $u$, and using also the fact that
\[
\frac{1}{\mu(B(x,r))}\int_{B(x,r)\setminus Y}g_u(y)\,d\mu(y)\le \vint{B(x,r)}|g_u(y)-g_u(x)|\,d\mu(y)
+g_u(x)\frac{\mu(B(x,r)\setminus Y)}{\mu(B(x,r))}\to0,
\]
we have
\begin{equation}
\label{eq: N11 diff-gv}
\lim_{r\to0}\,\vint{B(x,r)}g_v\,d\mu=0.
\end{equation}
Let $r_i=2^{-i}r$.
Since $x$ is a Lebesgue point of $u$ and $\ell_x$ is continuous, then $v_{B(x,r_i)}\to v(x)=0$.
By the $1$-Poincar\'e inequality we have
\[
\begin{split}
\frac{1}{r}\vint{B(x,r)}|v|\,d\mu
&\le \frac{1}{r}\vint{B(x,r)} |v-v_{B(x,r)}|\,d\mu+\frac{1}{r}|v_{B(x,r)}|\\
&\le C\vint{B(x,\lambda r)}g_v\,d\mu+\frac{1}{r}\sum_{i=0}^{\infty}|v_{B(x,r_i)}-v_{B(x,r_{i+1})}|\\
&\le C\vint{B(x,\lambda r)}g_v\,d\mu+C\sum_{i=0}^{\infty} \frac{r_i}{r}\vint{B(x,\lambda r_i)}g_v\,d\mu\\
&\le C\sup_{0<t\le\lambda r}\vint{B(x,t)}g_v\,d\mu.
\end{split}
\]
Thus by \eqref{eq: N11 diff-gv} it follows that
\begin{equation}
\label{eq: N11 diff-remaining part}
\lim_{r\to0}\frac{1}{r}\vint{B(x,r)}|v|\,d\mu=0.
\end{equation}
Since $d\|Dv\|\le g_v\,d\mu$, doubling gives
\[
\frac{\|Dv\|(B(x,\lambda r))}{\mu(B(x,r))}
\le C\vint{B(x,\lambda r)}g_v\,d\mu\to0.
\]
Applying Lemma \ref{lem:zero-BV-blowup} with $f=v$ and $e=v^*$,
at the chosen point, we have $v^*(x)=0$, and $|v^*|\le|v|^\vee$ $\mathcal H$-a.e. by \eqref{eq:pointwise-properties}.
Thus $d_{\mathrm{fine}}^\alpha v^*(x)=0$.

Since $\ell_x(y)-\ell_x(x)=\langle d^\alpha u(x),\xi^\alpha(y)-\xi^\alpha(x)\rangle$ and $v^*=u^*-\ell_x$, we conclude that $u^*$ is $1$-finely differentiable at $x$ and $d_{\mathrm{fine}}^\alpha u^*(x)=d^\alpha u(x)$.
The same expansion holds for $u$, since $u(x)=u^*(x)$ for $1$-q.e. $x$; recall \eqref{eq:Lebesgue points}.
The chosen conditions hold at $\mu$-a.e. point of every refined chart piece, and hence at $\mu$-a.e. point of $\Om\cap X_\alpha$.
\end{proof}

\begin{proposition}[Fine differentiability of BV functions]
\label{thm:BV-finely-diff}
Let $\Om\subset X$ be open and let $u\in\BV_\loc(\Om)$.
Let $\{(X_\alpha,\xi^\alpha)\}_\alpha$, where $\xi^\alpha=(\xi^\alpha_1,\ldots,\xi^\alpha_{N_\alpha})$, be a Cheeger measurable differentiable structure.
Assume also that the fine-scale geometric condition $(\mathrm{FG})_x$ in Definition~\ref{def:FG} holds for $\mu$-a.e. $x\in\Om$.
Then, for $\mu$-a.e. $x\in\Om\cap X_\alpha$, the representative $u^*$ is $1$-finely differentiable at $x$ with respect to $(X_\alpha,\xi^\alpha)$, and
\[
d^\alpha_{\mathrm{fine}}u^*(x)=d^\alpha u(x).
\]
\end{proposition}

\begin{proof}
By an exhaustion of $\Om$ by relatively compact open subsets, it suffices to prove the assertion when $u\in\BV(\Om)$.
Fix $\varepsilon\in(0,1)$.
Since $\|D^su\|\perp\mu$, there is a Borel set $N\subset\Om$ with $\mu(N)=0$ and $\|D^su\|(\Om\setminus N)=0$.
By outer regularity, choose an open set $W$ such that
\begin{equation}
\label{eq:small-singular-neighborhood}
N\subset W\subset\Om,
\quad
\mu(W)<\varepsilon.
\end{equation}
By Lemma \ref{lem:localized-discrete-convolution}, there exists $h\in\BV(\Om)$ such that
\begin{equation}
\label{eq:localized-discrete-convolution-modification-h}
h=u\quad\text{on }\Om\setminus W,
\quad
h|_W\in N^{1,1}(W)\cap\liploc(W),
\end{equation}and
\begin{equation}
\label{eq:variation-supported-in-modification-set-h}
\|D(h-u)\|(\Om\setminus W)=0.
\end{equation}
For every Borel $A\subset\Om\setminus W$, we have $\|Dh\|(A)=\|Du\|(A)$.
Hence $\|Dh\|\mres(\Om\setminus W)\ll\mu$, while \eqref{eq:localized-discrete-convolution-modification-h} implies $\|Dh\|\mres W\ll\mu$.
Thus $\|Dh\|=a\mu$ for some nonnegative $a\in L^1(\Om)$.
Applying Lemma~\ref{lem:BV-Sobolev density} with $p=1$, we obtain $v\in N^{1,1}_{\loc}(\Om)$ such that $v=h$ $\mu$-a.e. and $v$ has a $1$-weak upper gradient $g\le Ca$ $\mu$-a.e.
Since $h\in L^1(\Om)$ and $g\in L^1(\Om)$, we have $v\in N^{1,1}(\Om)$.
By the $1$-q.e. Lebesgue-point property \eqref{eq:Lebesgue points}, we have
\[
v(x)=\lim_{r\to0}\vint{B(x,r)}v\,d\mu
=\lim_{r\to0}\vint{B(x,r)}h\,d\mu=h^*(x)
\quad\text{for }1\text{-q.e. }x\in\Om.
\]
Thus $h^*\in N^{1,1}(\Om)$ and $g_{h^*}\le g\le Ca$ $\mu$-a.e.
Set $e_0\coloneqq h-u$ and let $e$ be defined as in Lemma \ref{lem:localized-discrete-convolution}.
By Lemma \ref{lem:localized-discrete-convolution}, it follows that $d^\alpha_{\mathrm{fine}}e(x)=0$ for $\mu$-a.e.
$x\in(\Om\setminus W)\cap X_\alpha$.
By Proposition \ref{thm: fine-diff of N11}, we know $h^*$ is finely differentiable almost everywhere.
We may exclude the capacity-zero set where either $h^*$ or $u^*$ is not finite.
At points where these two conclusions hold, intersecting the corresponding fine neighborhoods and using $u^*=h^*-e$, it follows that
\[
d^\alpha_{\mathrm{fine}}u^*(x)=d^\alpha_{\mathrm{fine}}h^*(x).
\]
By \cite[Corollary 4.1]{DCIKS}, $u^*$ is approximately differentiable $\mu$-a.e.
Then by the uniqueness of the approximate differential, we have $d^\alpha_{\mathrm{fine}}u^*(x)=d^\alpha u(x)$.
Letting $\varepsilon\to 0$ proves the claim.
\end{proof}

\begin{proof}[Proof of Theorem \ref{thm:BV-finely-diff-main theorem}]
We verify the geometric condition $(\mathrm{FG})_x$ in Definition~\ref{def:FG} at every $x\in\Om$.
If $Q>1$, by Ahlfors regularity we have \eqref{eq: reverse doubling ineq.-s>1} with $s=Q$ and $C_x=C_A^2$.
Suppose that $Q=1$.
Fix $x\in\Om$ and choose $R_x>0$ so that $B(x,2R_x)\Subset\Om$ and $R_x<\tfrac18\diam X$.
For $0<r<R_x$ and $y\in B(x,r)$, by quasiconvexity, there exists a rectifiable curve from $y$ to a point outside $B(x,2r)$.
After removing loops and stopping at its first exit, we obtain an injective rectifiable arc $\gamma$ with endpoints $y$ and $z\in X\setminus B(x,2r)$.
By Ahlfors $1$-regularity and the $5$-covering theorem, it follows that $\mathcal H^1_d(A)\le CC_A\mu(A)$ for Borel sets $A$, where $\mathcal{H}_d^1$ is ordinary one-dimensional Hausdorff measure.
Thus every admissible function $v$ for $\rcapa_1(\{y\},B(x,2r))$ and every upper gradient $g$ of $v$ satisfy
\[
1\le|v(y)-v(z)|\le\int_\gamma g\,ds
\le CC_A\int_Xg\,d\mu.
\]Taking the infimum and using \eqref{eq:rcapa over ball-measure}, we obtain
\[
\rcapa_1(\{y\},B(x,2r))\ge(CC_A)^{-1},
\qquad
\rcapa_1(B(x,r),B(x,2r))\le C_dC_A.
\]Hence \eqref{eq:FG-singleton-capacity} holds with $\delta_x^*=(CC_dC_A^2)^{-1}$.
In both cases $(\mathrm{FG})_x$ from Definition~\ref{def:FG} holds at every $x\in\Om$, so the theorem follows from Proposition \ref{thm:BV-finely-diff}.
\end{proof}

\begin{remark}
The definition of $p$-thinness when $1<p<\infty$ has quite a different form compared to the definition of $1$-thinness, and thus in the case $u\in N^{1,p}_{\loc}(\Om)$ for $1<p<\infty$, we still only have differentiability with respect to the $1$-fine topology, and not the $p$-fine topology.
\end{remark}

At any chart point, for $v\in \R^{N_{\alpha}}$ we use the pointwise realization
\[
|v|_{C,x}\coloneqq\Lip_{\langle v,\xi^\alpha\rangle}(x).
\]
It is a norm by the fixed approximate-uniqueness convention, and it agrees almost everywhere with the intrinsic cotangent norm, see \cite[Section 1D]{ES24}.
If $w$ is finely differentiable at a point $x$ satisfying $(\mathrm{FG})_x$ from Definition~\ref{def:FG}, by Lemma \ref{lem:fine diff} \ref{item-fine diff-2} and the choice $U_r=U\cap B(x,r)$ we have
\begin{equation}
\label{eq:Lip and fine der}
\Lip_w^0(x)\le|d^\alpha_{\mathrm{fine}}w(x)|_{C,x}.
\end{equation}
In fact, the equality in \eqref{eq:Lip and fine der} holds for $\mu$-a.e. point.

\begin{lemma}\label{lem:chart-linear-oscillation}
For each chart $(X_\alpha,\xi^\alpha)$, the following holds for
$\mu$-a.e.\ $x\in X_\alpha$.
For every $v\in\R^{N_\alpha}$, set $L_v(y)\coloneqq
\langle v,\xi^\alpha(y)-\xi^\alpha(x)\rangle$.
Then, for every family of sets $E_r\subset B(x,r)$, defined for
all sufficiently small $r>0$ and satisfying
\[
\lim_{r\to0}
\frac{\mu^*(B(x,r)\setminus E_r)}{\mu(B(x,r))}=0,
\]
we have
\begin{equation}
\label{eq:Ur and U oscillation}
\lim_{r\to0}\frac{\osc_{E_r}L_v}{2r}=|v|_{C,x}.
\end{equation}
Here $\mu^*$ denotes the outer measure when $E_r$ is not measurable.
\end{lemma}
\begin{proof}
First let $f$ be a Lipschitz function, localized by a Lipschitz cutoff if necessary.
By \cite[Theorems 1.7 and 1.10]{ES24} and the $1$-Poincar\'e inequality, $g_f=\Lip_f$ almost everywhere.
For $\mu$-a.e. $x$ with $g_f(x)>0$ and every $0<q<g_f(x)$, by \cite[Theorem 1.1 and Corollary 1.3]{ES24}, there exists an absolutely continuous curve $\gamma$ and an interior parameter $t$ with
\[
\gamma(t)=x,
\qquad
\sigma\coloneqq|\gamma'|(t)>0,
\qquad
|(f\circ\gamma)'(t)|>q\sigma.
\]
For fixed $\eta>0$ and all sufficiently small $r$, the points $\gamma(t\pm r/((1+\eta)\sigma))$ lie in $B(x,r)$.
The two-sided derivative along the curve therefore yields
\[
\liminf_{r\to0}\frac{\osc_{B(x,r)}f}{2r}\ge\frac{q}{1+\eta}.
\]
Let $\eta\to0$ and $q\uparrow g_f(x)$.
The reverse inequality follows from $\osc_{B(x,r)}f\le2\sup_{B(x,r)}|f-f(x)|$. Hence
\[
\lim_{r\to0}\frac{\osc_{B(x,r)}f}{2r}=\Lip_f(x)=g_f(x)
\]
for almost every $x$, including points where $g_f(x)=0$.
Apply this to the countable family $f=\langle v,\xi^\alpha\rangle$ with rational $v$ and the union of their exceptional sets is $\mu$-null.
Fix $x\in X_\alpha$ outside this null set. For any $v\in \R^{N_\alpha}$, choose $v_j\in \Q^{N_\alpha}$ with $v_j\to v$. It follows that
\begin{equation*}
\begin{split}
\left\lvert\frac{\osc_{B(x,r)}L_v}{2r}-\lvert v\rvert_{C,x}\right\rvert
&\le \left\lvert\frac{\osc_{B(x,r)}L_v}{2r}-\frac{\osc_{B(x,r)}L_{v_j}}{2r}\right\rvert\\
&\quad +\left\lvert\frac{\osc_{B(x,r)}L_{v_j}}{2r}-\lvert v_j\rvert_{C,x}\right\rvert
+\bigl\lvert\lvert v_j\rvert_{C,x}-\lvert v\rvert_{C,x}\bigr\rvert\\
&\le 2\Lip(\xi^\alpha)\lvert v-v_j\rvert
+\left\lvert\frac{\osc_{B(x,r)}L_{v_j}}{2r}-\lvert v_j\rvert_{C,x}\right\rvert.
\end{split}
\end{equation*}
Since $\osc_{B(x,r)}L_{v_j}/(2r)\to\lvert v_j\rvert_{C,x}$ as $r\to0$ for every $j$, first taking the limit superior as $r\to0$ and then letting $j\to\infty$, it follows that
\[
\lim_{r\to0}\frac{\osc_{B(x,r)}L_v}{2r}=\lvert v\rvert_{C,x}.
\]
Thus, for $\mu$-a.e. $x\in X_\alpha$, \eqref{eq:Ur and U oscillation} holds for $E_r=B(x,r)$ and every $v\in\R^{N_\alpha}$.

Fix such a point $x\in X_\alpha$ and $0<\eta<1$.
For every fixed $0<\varepsilon<\eta$, by doubling, there exists $c(\varepsilon)>0$ such that
\[
\mu(B(y,\varepsilon r))\ge c(\varepsilon)\mu(B(x,r))
\quad\text{for }y\in B(x,(1-\eta)r).
\]
For sufficiently small $r$, no such ball can be contained in $B(x,r)\setminus E_r$.
Thus every $y\in B(x,(1-\eta)r)$ has a point of $E_r$ within distance $\varepsilon r$, and
\[
\osc_{E_r}L_v
\ge\osc_{B(x,(1-\eta)r)}L_v-2\Lip(L_v)\varepsilon r.
\]
Then it follows that
\[
\liminf_{r\to0}\frac{\osc_{E_r}L_v}{2r}
\ge(1-\eta)|v|_{C,x}-\Lip(L_v)\varepsilon.
\]
Let $\varepsilon\to0$ and then $\eta\to0$.
The reverse inequality follows from $E_r\subset B(x,r)$, proving \eqref{eq:Ur and U oscillation}.
\end{proof}

For $1\le p<\infty$, the $1$-Poincar\'e inequality also implies a $p$-Poincar\'e inequality.
The identifications in \cite[Theorems 1.7 and 1.10]{ES24} therefore identify the minimal $p$-weak upper gradient with the same Cheeger norm of the Sobolev differential.
For $p=\infty$, we use the locally Lipschitz representative of $u$ and so $g_{u,\infty}=\Lip_u$.
At $\mu$-a.e. $x\in\Om\cap X_\alpha$, by the Cheeger differentiability and the definition of $|\cdot|_{C,x}$, we have $|d^\alpha u(x)|_{C,x}=\Lip_u(x)=g_{u,\infty}(x)$.
Consequently, for $u\in N^{1,p}_{\loc}(\Om)$ in each of these cases,
\begin{equation}\label{eq:norm identification}
|d^\alpha u(x)|_{C,x}=g_u(x)
\quad\text{for }\mu\text{-a.e. }x\in\Om\cap X_\alpha.
\end{equation}

A natural question related to Theorem \ref{thm:main Sob} is the relationship between the weight $a$ and the minimal upper gradient $g_{u^*}$.
The next proposition makes precise the sense in which $g_{u^*}$ is the smallest admissible weight.
\begin{proposition}\label{prop:a new Sobo char}
Suppose that $X$ is Ahlfors $Q$-regular with $Q\ge1$ and $\Om$ is bounded.
These assumptions imply $(\mathrm{FG})_x$ from Definition~\ref{def:FG} at every $x\in\Om$, as verified in the proof of Theorem~\ref{thm:BV-finely-diff-main theorem}.
Let $u\in L^p(\Om)$, where $1\le p\le\infty$.
Then the following are equivalent:
\begin{enumerate}
\item\label{item-1-a new Sobo char}
$u^*\in N^{1,p}(\Om)$.
\item\label{item-2-a new Sobo char}
There are a nonnegative $a\in L^p(\Om)$ and $M\ge1$ such that $\Lip_{u^*,a}^{0,M}(x)\le1$ for $\mu$-a.e. $x\in\Om$,
and, for every $0<\delta<1$, $\Lip_{u^*,a}^{\delta,M}(x)\le C\delta^{-1}$ for every $x\in\Om$,
where $C$ depends only on the doubling and Poincar\'e constants.
\end{enumerate}
For every $\varepsilon>0$, we may choose $a=g_{u^*}+\varepsilon$.
Every weight $a$ in \ref{item-2-a new Sobo char} satisfies $a\ge g_{u^*}$ $\mu$-almost everywhere.
If $u\in N^{1,p}(\Om)$, then $g_{u^*}=g_u$ almost everywhere, so these conclusions hold with $g_u$ in place of $g_{u^*}$.
\end{proposition}
\begin{proof}
$\ref{item-1-a new Sobo char}\Rightarrow\ref{item-2-a new Sobo char}$: 
By Proposition \ref{prop:necessity} applied to $u^*$, we have $\Lip_{u^*, g_{u^*}}^{\delta, 2\lambda}(x)\le C\delta^{-1}$ for every $x\in \Om$.
By Proposition \ref{thm: fine-diff of N11}, \eqref{eq:Lip and fine der}, and the norm identification above,
\[
\Lip_{u^*}^0(x)\le|d^\alpha_{\mathrm{fine}}u^*(x)|_{C,x}=g_{u^*}(x)
\quad\text{for }\mu\text{-a.e. }x\in\Om.
\]
At Lebesgue points of $g_{u^*}$, the choice $a=g_{u^*}+\varepsilon$ therefore gives
$\Lip_{u^*,a}^{0,M}\le1$. Since $\mu(\Om)<\infty$, we have $g_{u^*}+\varepsilon\in L^p(\Om)$.

$\ref{item-2-a new Sobo char}\Rightarrow\ref{item-1-a new Sobo char}$: Using Theorem \ref{thm:main in text} applied to $u^*$ with $Z=\emptyset$, $d\nu=a\,d\mu$, and $h=C\widetilde c_*^{-1}$, together with Lemma \ref{lem:measures belonging to M},
since $u=u^*$ almost everywhere, we get $u\in\BV(\Om)$ and $d\|Du\|\le C'a\,d\mu$.
Applying Lemma~\ref{lem:BV-Sobolev density} with $C'a$ in place of $a$, we obtain a representative $v\in N^{1,p}_{\loc}(\Om)$ with $v=u$ $\mu$-a.e. and $v$ has a $1$-weak upper gradient $g\le Ca$ $\mu$-a.e.
Since $u\in L^p(\Om)$, we have $v\in N^{1,p}(\Om)$.
For $p<\infty$, we know
\[
u^*(x)=\lim_{r\to0}\vint{B(x,r)}u\,d\mu
=\lim_{r\to0}\vint{B(x,r)}v\,d\mu=v(x)
\quad\text{for }p\text{-q.e. }x\in\Om;
\]
for $p=1$, this is given by \eqref{eq:Lebesgue points};
for $1<p<\infty$, apply \cite[Theorem~9.2.8]{HKSTbook} locally using Lipschitz cutoffs and a countable cover of $\Om$.
Thus $u^*\in N^{1,p}(\Om)$.
For $p=\infty$, the representative $v$ is locally Lipschitz, so continuity and $v=u$ $\mu$-a.e. imply $u^*(x)=v(x)$ for every $x\in\Om$, yielding the same conclusion.

To prove minimality, fix $\kappa>0$.
By Lemma \ref{lem:chart-linear-oscillation}, for $\mu$-a.e. $x\in \Om\cap X_{\alpha}$, $x$ is a Lebesgue point of $a$, and for every $v$ and every family $E_r$, we have that $u^*$ is finely differentiable, and $\Lip_{u^*,a}^{0,M}(x)\le1$.
Let $U\ni x$ be a fine neighborhood on which $u^*$ has its first-order expansion, and choose an admissible family $U_r\subset B(x,r)$ such that
\begin{equation}
\label{eq:one plus kappa}
\limsup_{r\to0}\frac{\osc_{U_r}u^*}{2r}
\left(\vint{B(x,Mr)}(a+\kappa)\,d\mu\right)^{-1}
\le1+\kappa.
\end{equation}
Set $E_r\coloneqq U_r\cap U$.
The complement of $U_r$ has vanishing relative capacity and the complement of $U$ is thin at $x$.
Hence by \eqref{eq:rcapa in small ball} we have $\mu^*(B(x,r)\setminus E_r)/\mu(B(x,r))\to0$.
Then by Lemma \ref{lem:chart-linear-oscillation}, it follows that
\[
\lim_{r\to0}\frac{\osc_{E_r}u^*}{2r}
=|d^\alpha_{\mathrm{fine}}u^*(x)|_{C,x}=g_{u^*}(x).
\]
By \eqref{eq:one plus kappa}, we have 
\[
1+\kappa\ge
\limsup_{r\to0}\frac{\osc_{E_r}u^*}{2r}
\left(\vint{B(x,Mr)}(a+\kappa)\,d\mu\right)^{-1}
=\frac{g_{u^*}(x)}{a(x)+\kappa}.
\]
Letting $\kappa\to 0$, we conclude $g_{u^*}\le a$ almost everywhere. 
The fact that $g_{u^*}=g_u$ follows from $u\in N^{1,p}(\Om)$.
\end{proof}

\begin{proposition}\label{prop:equality}
Suppose that $X$ is Ahlfors $Q$-regular with $Q\ge1$.
These assumptions imply $(\mathrm{FG})_x$ from Definition~\ref{def:FG} at every $x\in\Om$, as verified in the proof of Theorem~\ref{thm:BV-finely-diff-main theorem}.
For $\mu$-a.e. $x\in\Om\cap X_\alpha$, every function $w\colon\Om\to[-\infty,\infty]$ that is $1$-finely differentiable at $x$ satisfies
\[
\Lip_w^0(x)=|d^\alpha_{\mathrm{fine}}w(x)|_{C,x}.
\]
\end{proposition}
\begin{proof}
The upper bound is \eqref{eq:Lip and fine der}.
Next we prove the opposite direction.
For $\mu$-a.e. $x\in\Om\cap X_\alpha$, the conclusion of Lemma \ref{lem:chart-linear-oscillation} holds for every $v$ and every family $E_r$.
Fix such a point $x$ and a function $w$ that is $1$-finely differentiable at $x$.
For any admissible family $U_r$ in the definition of $\Lip_w^0(x)$, intersect $U_r$ with the fine neighborhood $U$ from Lemma \ref{lem:fine diff} \ref{item-fine diff-2}.
Combining \eqref{eq:rcapa in small ball} and Lemma \ref{lem:chart-linear-oscillation}, it follows that
\[
\limsup_{r\to0}\frac{\osc_{U_r}w}{2r}
\ge\lim_{r\to0}\frac{\osc_{U_r\cap U}w}{2r}
=|d^\alpha_{\mathrm{fine}}w(x)|_{C,x}.
\]
Taking the infimum over the admissible families proves the lower bound.
\end{proof}

For $1\le p<\infty$, we write $D^p(\Om)$ for the Dirichlet space of measurable functions $u\colon\Om\to[-\infty,\infty]$, finite $\mu$-almost everywhere, that admit a $p$-weak upper gradient in $L^p(\Om)$.
For $p=\infty$, we use the convention that $D^\infty(\Om)$ consists of locally Lipschitz functions admitting an upper gradient in $L^\infty(\Om)$.
Under our standing assumptions, $D^p(\Om)\subset N^{1,p}_{\loc}(\Om)$ for $1\le p\le\infty$.
When the exponent needs to be specified, we denote the minimal $p$-weak upper gradient by $g_{u,p}$ for $1\le p<\infty$.

We can now formulate a characterization in terms of fine differentiability.


\begin{corollary}
Suppose that $X$ is Ahlfors $Q$-regular with $Q\ge1$.
These assumptions imply $(\mathrm{FG})_x$ from Definition~\ref{def:FG} at every $x\in\Om$, as verified in the proof of Theorem~\ref{thm:BV-finely-diff-main theorem}.
Let $u\colon\Om\to[-\infty,\infty]$ be measurable, and let $1\le p\le\infty$.
\begin{enumerate}
\item\label{item-1-Cor. finely diff}
If $u\in D^p(\Om)$,  then $d^\alpha_{\mathrm{fine}}u(x)$ and $d^\alpha_{\mathrm{fine}}u^*(x)$ exist for $\mu$-a.e. $x\in\Om\cap X_\alpha$, and
\[
|d^\alpha_{\mathrm{fine}}u(x)|_{C,x}=|d^\alpha_{\mathrm{fine}}u^*(x)|_{C,x}=g_u(x)
\quad\text{for }\mu\text{-a.e. }x\in\Om\cap X_\alpha.
\]
\item\label{item-2-Cor. finely diff}
If there exists a set $E\subset\Om$ with $\mathcal H(E)=0$ such that every $x\in\Om\setminus E$ belongs to a chart $X_\alpha$ in which $d^\alpha_{\mathrm{fine}}u(x)$ exists, and the intrinsic norm of this differential belongs to $L^p(\Om)$,
then $u\in L^1_{\loc}(\Om)$ and $u^*\in D^p(\Om)$.
\end{enumerate}
\end{corollary}
\begin{proof}
Claim \ref{item-1-Cor. finely diff} follows locally from Proposition \ref{thm: fine-diff of N11} and the norm identification \eqref{eq:norm identification}, since $u\in D^p(\Om)$ implies
$u\in N^{1,p}_{\loc}(\Om)\subset N^{1,1}_{\loc}(\Om)$.
For Claim \ref{item-2-Cor. finely diff}, let $E$ be the $\mathcal H$-negligible exceptional set and set $h=\Lip_u^0$ on $\Om\setminus E$, with $h=0$ on $E$.
The global hypothesis ensures that every $x\in\Om\setminus E$ lies in a chart where $u$ is finely differentiable, and by \eqref{eq:capa and Hausdorff measure} we have $\capa_1(E)=0$.
By \eqref{eq:Lip and fine der}, we know $h$ is finite on $\Om\setminus E$ and by Proposition \ref{prop:equality} we have $h\in L^p(\Om)$.

We first justify local integrability, which is needed to use Theorem \ref{thm:main in text} on $u$.
Let $u_k\coloneqq\max\{-k,\min\{k,u\}\}$.
On every bounded open $V\Subset\Om$, using truncation, we get
\[
\Lip_{u_k,\mu}^{\widetilde c_*,1}\le\Lip_u^0=h
\quad\text{on }V\setminus E.
\]
Applying Theorem \ref{thm:main in text} with $Z=E\cap V$ and $\nu=\mu|_V$, we then have $\|Du_k\|\le Ch\mu$ in $V$, where $C$ is independent of $k$ and $V$.
Fix a ball $B$ with $2\lambda B\Subset\Om$.
Since $u$ is finite almost everywhere, choose $A<\infty$ such that $|u|\le A$ on a subset of $B$ of measure at least $\mu(B)/2$.
On this subset $|u_k|\le A$ for every $k$, and by the BV--Poincar\'e inequality we have
\[
\vint{B}|u_k|\,d\mu
\le A+3\vint{B}|u_k-(u_k)_B|\,d\mu
\le A+Cr_B\frac{\int_{\lambda B}h\,d\mu}{\mu(\lambda B)}.
\]
Letting $|u_k|\to |u|$ yields $u\in L^1(B)$ and then $u_k\to u$ in $L^1(B)$.
By the lower semicontinuity, we have $u\in\BV_{\loc}(\Om)$ and $\|Du\|\le Ch\mu$.

Applying Lemma~\ref{lem:BV-Sobolev density} with $a=Ch$, we obtain a representative $v\in N^{1,p}_{\loc}(\Om)$ such that $v=u$ $\mu$-a.e.
If $p<\infty$, then $v$ has a $p$-weak upper gradient $g\in L^p(\Om)$ with $g\le Ch$ $\mu$-a.e.
By the same Lebesgue-point argument as in the proof of Proposition~\ref{prop:a new Sobo char}, we have $u^*=v$ $p$-q.e. in $\Om$.
Consequently, $g$ is also a $p$-weak upper gradient of $u^*$.
If $p=\infty$, then $v$ is locally Lipschitz. Since $v=u$ $\mu$-a.e., we have $u^*=v$ everywhere in $\Om$; moreover, $\Lip_v$ is an upper gradient of $u^*$ satisfying $\Lip_v\le Ch$ $\mu$-a.e.
Since $h\in L^p(\Om)$, we have $u^*\in D^p(\Om)$.
\end{proof}

In an Ahlfors $1$-regular space, for all $y\in B(x,r)$ and $B(x,2r)\Subset \Om$, we have $\rcapa_1(\{y\},B(x,2r))\ge \delta \rcapa_1(B(x,r),B(x,2r))$.
Thus fine differentiability coincides with ordinary chartwise differentiability.
In particular, $\capa_1(\{x\})=2$ in $\R$.
The function $u(t)=|t|$ on $(-1,1)$ shows that ``$\mu$-a.e.'' cannot be replaced by ``$\mathcal H$-a.e.'' in Claim \ref{item-1-Cor. finely diff}.
The Cantor--Vitali function $u\colon[0,1]\to[0,1]$ shows that ``$\mathcal H$-a.e.'' cannot be replaced by ``$\mu$-a.e.'' in Claim \ref{item-2-Cor. finely diff}.

\endgroup

\begingroup
\emergencystretch=3em
\section{Quasicontinuity and Sobolev regularity}\label{sec: lipf}

As before, $\Om\subset X$ is an open set and let $1\le p<\infty$.

Consider the usual Lipschitz number
\[
\lip_w(x)\coloneqq \liminf_{r\to 0}\sup_{y\in B(x,r)}\frac{|w(y)-w(x)|}{r}.
\]
As mentioned in the Introduction,
in the Euclidean setting, Balogh--Cs\"ornyei \cite[Theorem 1.2]{BC} showed that if
$w\colon \Om\to \R$ satisfies the following:
\begin{itemize}
\item $w$ is continuous;
\item $\lip_w<\infty$ outside a set of $\sigma$-finite $\mathcal H^{n-1}$-measure; and
\item $\lip_w\in L^p(\Omega)$,
\end{itemize}
then $w\in W^{1,p}_{\loc}(\Omega)$.
Generalizations of this result to metric measure spaces have been shown by Wildrick and Z\"urcher \cite{WZ, Zu}.

These hypotheses are not necessary in general: when $n\ge2$ and $1\le p\le n$, a Sobolev function $u\in W^{1,p}_{\loc}(\Om)$ may be discontinuous and have $\lip_{u^*}(x)=\infty$ for every $x\in \Om$.
Precisely because Sobolev functions are not generally continuous, one often uses the fact that they are nonetheless quasicontinuous, and so it was asked in \cite[Remark 3.15]{Zu} whether it would be enough to assume quasicontinuity in the above type of theorem.
The next proposition gives a related criterion for the capacity-relaxed upper Lipschitz number $\Lip_u^0$.

\begin{proposition}\label{prop:three lip conditions}
Suppose that $X$ is Ahlfors $Q$-regular with $Q\ge1$, $\Om$ is bounded, and $u\colon \Om\to [-\infty,\infty]$ is measurable such that:
\begin{enumerate}
\item\label{prop:three lip conditions-item-1}
$u$ is $p$-quasicontinuous in $\Omega$;
\item\label{prop:three lip conditions-item-2}
$\Lip_{u}^{0}<\infty$ outside a set of $\sigma$-finite $\mathcal H$-measure; and
\item\label{prop:three lip conditions-item-3}
$\Lip_{u}^{0}\in L^p(\Omega)$.

\end{enumerate}
Then $u\in D^p(\Om)$.

Conversely, if $u\in D^p(\Om)$, then
\ref{prop:three lip conditions-item-1} and \ref{prop:three lip conditions-item-3} are satisfied.
\end{proposition}

The following Lemma is similar in spirit to \cite[Lemma 3.4]{LaZh2} and \cite[Proposition 3.2]{Zu}, with continuity replaced by $1$-quasicontinuity.

\begin{lemma}\label{lem:quasicontinuous control}
Let $\Om$ be bounded, let $E\subset\Om$ have $\sigma$-finite $\mathcal H$-measure, and let $v\colon\Om\to\R$ be bounded and $1$-quasicontinuous.
For every $0<\ell\le1$ and every $\varepsilon>0$, there is an $(\ell,v)$-control $\mathcal S=(G,\{B_i\}_i,b)$ of $E$ such that, with $H=G\cup\bigcup_i2B_i$,
\begin{equation}
\label{eq:quasicontinuous control conclusion}
\begin{split}
\mu(H)&<\varepsilon,\\
\mathcal{C}_v(\mathcal{S})
&=\int_Xb\,d\mu+\sum_i\frac{\osc_{B_i\setminus G}v}{r_i}\mu(2B_i)<\varepsilon.
\end{split}
\end{equation}
\end{lemma}

\begin{proof}
Set $L\coloneqq \osc_\Om v$.
Since $X$ is proper, choose compact sets $K_j\Subset\operatorname{int}K_{j+1}\Subset\Om$ increasing to $\Om$, and write $E=\bigcup_iE_i$ with $\mathcal H(E_i)<\infty$.
Choose positive numbers $a_{i,j}$ and $s_{i,j}$ so small that
\begin{equation}
\label{eq:quasicontinuous control weights}
\sum_{i,j=1}^{\infty}(a_{i,j}+s_{i,j})(\mathcal H(E_i)+1)<\frac{\varepsilon}{10C_d}.
\end{equation}
By $1$-quasicontinuity, there is a relatively open set $G\subset\Om$ of arbitrarily small $1$-capacity such that $v$ is continuous on $\Om\setminus G$.
Choose a Borel function $\eta\in N^{1,1}(X)$ and a nonnegative Borel $1$-weak upper gradient $g_\eta$ such that $0\le\eta\le1$, $\eta=1$ on $G$, and
\[
\int_X(\eta+g_\eta)\,d\mu<\frac{\varepsilon}{10(1+L)(1+\ell^{-1})}.
\]
In particular, $\mu(G)\le\int_X\eta\,d\mu<\varepsilon/10$.

For each $j$, the function $v$ is uniformly continuous on the compact set $K_{j+1}\setminus G$.
Hence, by the definition of $\mathcal H$, we can cover $(E_i\cap K_j)\setminus G$ by balls $B_{i,j,q}=B(z_{i,j,q},r_{i,j,q})$ such that $2B_{i,j,q}\subset K_{j+1}$, $0<r_{i,j,q}<\min\{\ell/8,s_{i,j}\}$, $\osc_{B_{i,j,q}\setminus G}v\le a_{i,j}$, and
\[
\sum_q\frac{\mu(B_{i,j,q})}{r_{i,j,q}}\le\mathcal H(E_i\cap K_j)+2^{-i-j}.
\]
Then $E\setminus G\subset \bigcup_{i,j,q}B_{i,j,q}$.
By doubling and \eqref{eq:quasicontinuous control weights}, we have
\[
\begin{split}
\sum_{i,j,q}\frac{\osc_{B_{i,j,q}\setminus G}v}{r_{i,j,q}}\mu(2B_{i,j,q})
&\le \sum_{i,j}a_{i,j}\sum_q\frac{\mu(2B_{i,j,q})}{r_{i,j,q}}\\
&\le C_d\sum_{i,j}a_{i,j}(\mathcal{H}(E_i)+1)
\le \frac{\varepsilon}{5},
\end{split}
\]
and
\[
\mu\left(\bigcup_{i,j,q}2B_{i,j,q}\right)\le C_d\sum_{i,j}s_{i,j}\left(\mathcal H(E_i)+1\right)<\frac{\varepsilon}{5}.
\]
Hence, $\mu(H)\le \mu(G)+\mu\big(\bigcup_{i,j,q}2B_{i,j,q}\big)<\varepsilon$.

Define
\[
b\coloneqq L\left(2g_\eta+\frac{2}{\ell}\eta\right).
\]
When $L=0$, take $b=0$.

Our choice of $\eta$ gives $\int_Xb\,d\mu<\varepsilon/5$.
Reindex the balls $\{B_{i,j,q}\}$ as $B_i$;
then \eqref{eq:quasicontinuous control conclusion} follows.
Outside a curve family of $1$-modulus zero, the upper-gradient inequality for $\eta$ holds on every subcurve.
Let $\gamma$ be such a curve with $\ell_\gamma\ge\ell$, endpoints outside $H$, and $\gamma\cap G\ne\emptyset$.
If $\eta\ge1/2$ everywhere on $\gamma$, then $(2/\ell)\int_\gamma\eta\,ds\ge1$.
Otherwise, a subcurve joins a point of $G$, where $\eta=1$, to a point where $\eta<1/2$, and hence $2\int_\gamma g_\eta\,ds\ge1$.
In either case,
\[
\int_\gamma b\,ds\ge L\ge|v(\gamma(0))-v(\gamma(\ell_\gamma))|.
\]
Thus $\mathcal S=(G,\{B_i\}_i,b)$ is an $(\ell,v)$-control.
\end{proof}

Combining the vanishing-cost controls from Lemma \ref{lem:quasicontinuous control} with the controlled BV criterion yields the required exceptional-set result for quasicontinuous functions.

\begin{proposition}\label{prop:quasicontinuous sigma finite H exceptional set}
Let $\Om$ be bounded, let $E\subset\Om$ have $\sigma$-finite $\mathcal H$-measure, and let $w\colon\Om\to[-\infty,\infty]$ be $1$-quasicontinuous.
Suppose that $M\ge1$, $\nu\in\mathcal M(\Om,(800C_1C_S^2C_d^6\widetilde c_*)^{-1})$, $\nu(\Om)<\infty$, and $h\in L^1(\Om,\nu)$ is nonnegative.
If
\[
\Lip_{w,\nu}^{\widetilde c_*,M}(x)\le h(x)<\infty\quad\textrm{for every }x\in\Om\setminus E,
\]
then $w\in D^{\BV}(\Om)$ and $\|Dw\|\le Ch\nu$, where $C$ depends only on doubling and Poincar\'e constants and on $M$.
\end{proposition}

\begin{proof}
First, $1$-quasicontinuity implies that $w$ is measurable and finite almost everywhere.

For $k\in\N$, set $w_k\coloneqq T_kw=\max\{-k,\min\{k,w\}\}$.
The function $w_k$ is bounded and $1$-quasicontinuous, and truncation decreases oscillation, so
\[
\Lip_{w_k,\nu}^{\widetilde c_*,M}\le\Lip_{w,\nu}^{\widetilde c_*,M}\le h
\quad\textrm{on }\Om\setminus E.
\]
For $\ell_m=1/m$, by Lemma \ref{lem:quasicontinuous control}, there exists an $(\ell_m,w_k)$-control $\mathcal S_{k,m}$ with $\mu(H(\mathcal S_{k,m}))<2^{-m}$ and $\mathcal C_{w_k}(\mathcal S_{k,m})<2^{-m}$.
Since $w_k$ is bounded, it belongs to $L^1(\Om)$.
Applying Proposition~\ref{prop:controlled BV criterion} with $Z=\emptyset$, it follows that
\[
\|Dw_k\|(\Om)\le C\int_\Om h\,d\nu
\]
with a constant independent of $k$.
For every open $U\Subset\Om$, 
applying the same argument in $U$ and using regularity of the measures, we have $\|Dw_k\|\le Ch\nu$.

It remains to let $k\to\infty$.
Since $w$ is finite almost everywhere, $w_k\to w$ almost everywhere.
Fix a ball $B=B(x,r)$ with $B(x,2\lambda r)\Subset\Om$.
Choose $A_B<\infty$ such that $\mu(B\cap\{|w|\le A_B\})>\mu(B)/2$.
If $m_k$ is a median of $w_k$ on $B$, then $|m_k|\le A_B$ for every $k\ge A_B$.
By the median form of the BV--Poincar\'e inequality and the uniform variation estimate therefore we have
\[
\int_B|w_k|\,d\mu
\le A_B\mu(B)+Cr\|Dw_k\|(B(x,\lambda r))
\le A_B\mu(B)+CrC\int_{B(x,\lambda r)}h\,d\nu.
\]
Thus the truncations are uniformly bounded in $\BV(B)$.
Since $|w_k|\nearrow|w|$ almost everywhere, we get $w\in L^1_{\loc}(\Om)$ and $w_k\to w$ in $L^1_{\loc}(\Om)$.
By lower semicontinuity, it follows that for every open $U\Subset\Om$,
\[
\|Dw\|(U)\le\liminf_{k\to\infty}\|Dw_k\|(U)\le C\int_Uh\,d\nu.
\]
Combining this with the exhaustion of $\Om$ by such open sets yields $w\in D^{\BV}(\Om)$ and $\|Dw\|\le Ch\nu$.
\end{proof}

\begin{proof}[Proof of Proposition \ref{prop:three lip conditions}]
To prove the first claim, denote by $E$ the subset of $\Om$ where $\Lip_u^{0}=\infty$.
Since $u$ is $p$-quasicontinuous, it is also $1$-quasicontinuous by \eqref{eq:pq quasicontinuity}, and thus also measurable.
Hence for every $x\in \Om\setminus E $, $\Lip_{u,\mu}^{\widetilde c_*,1}(x)\le\Lip_u^0(x)$ and $E$ has $\sigma$-finite $\mathcal H$-measure.
By Lemma \ref{lem:measures belonging to M} and \eqref{eq:cn definition}, we know $\mu\in \mathcal M(\Om,(800C_1C_S^2C_d^6\widetilde c_*)^{-1})$, and $\mu(\Om)<\infty$.
Applying Proposition \ref{prop:quasicontinuous sigma finite H exceptional set} with $w=u$, $\nu=\mu$, $M=1$, and $h=\Lip^0_{u}$, we have $u\in D^{\BV}(\Om)$ with $\|Du\|\le C\Lip^0_u\,\mu$.
Since $\Lip_u^0\in L^p(\Om)$, applying Lemma~\ref{lem:BV-Sobolev density} with $a=C\Lip_u^0$, there exists a representative $v\in N^{1,p}_{\loc}(\Om)$ with a $p$-weak upper gradient $g\le C\Lip_u^0$ almost everywhere.
By \eqref{eq:quasicontinuous}, $v$ is $p$-quasicontinuous, and $u$ is $p$-quasicontinuous by assumption.
Since $u=v$ $\mu$-a.e., by \cite[Proposition~5.23]{BB} we have $u=v$ $p$-q.e. in $\Om$ and  $g\in L^p(\Om)$ is also a $p$-weak upper gradient of $u$, and hence $u\in D^p(\Om)$.

For the converse claim, 
we have $u\in N^{1,p}_{\loc}(\Om)\subset N^{1,1}_{\loc}(\Om)$,
and \eqref{eq:quasicontinuous} and \eqref{eq:pq quasicontinuity} give $p$-quasicontinuity and $1$-quasicontinuity.
Under the Ahlfors $Q$-regularity assumption, condition $(\mathrm{FG})_x$ from Definition~\ref{def:FG} holds at every $x\in\Om$, as verified in Section~\ref{sec:fine differentiability}.
Proposition~\ref{thm: fine-diff of N11} therefore implies $1$-fine differentiability almost everywhere.
For $\mu$-a.e. $x\in\Om\cap X_\alpha$, by Proposition~\ref{prop:equality} we have $\Lip_u^0(x)=|d^\alpha_{\textrm{fine}}u(x)|_{C,x}=g_{u,1}(x)\le g_{u,p}(x).$
In particular, $\Lip_u^0$ is measurable and belongs to $L^p(\Om)$.
\end{proof}

The next example shows that condition \ref{prop:three lip conditions-item-2} is not necessary, but cannot simply be omitted from the sufficient conditions.

\begin{example}\label{ex:Cantor function}
		Let $g\in L^1(\R)$ be a nonnegative, lower semicontinuous function on the real line with 
		$g=\infty$ in an uncountable set $A$ (with $\mathcal L^1(A)=0$, of course).
		Let 
		\[
		u(x):=\int_{-\infty}^x g(t)\,dt,\quad x\in\R.
		\]
		Now $u\in D^1(\R)$, and $u$ (as well as $u^*$) is continuous and $\Lip_{u^*}^{0}=g$ a.e.,
		so that $\Lip_{u^*}^{0}\in L^1(\R)$. But
		$\Lip_{u^*}^{0}=\infty$ in $A$, which is not a set of $\sigma$-finite $\mathcal H^0$-measure.
		
		On the other hand, the Cantor--Vitali function $u\colon [0,1]\to [0,1]$ is a continuous function
		for which $\Lip_{u^*}^0=0$ a.e. but $u\notin W_{\loc}^{1,1}((0,1))$.
	\end{example}

This example demonstrates that the quantity $\Lip_{u^*}^{0}$ is insufficient for characterizing Sobolev functions.
This is the main motivation for considering the function $a$ and the generalized Lipschitz number $\Lip_{u^*,a}^{\delta}$, or $\Lip_{u^*,\nu}^{\delta,M}$.

\endgroup

\end{document}